\documentclass[12pt]{amsart}
\def\elystyle{0}
\usepackage{amsmath,amssymb,amsthm,amsfonts,enumitem,graphicx,comment,xparse,physics,tikz,pgfmath,dsfont,tikz,tikz-cd,verbatim,subcaption,hyperref,mathrsfs,stmaryrd}
\usepackage[style=numeric,
sorting=nyt,
date=year,
isbn=false,
doi=false,
url=false,
maxbibnames=99
]{biblatex}
\AtEveryBibitem{\clearfield{pages}} 
 \ifx\elystyle\undef
 \usepackage[margin=1in]{geometry} 
 \else
 \if\elystyle0
  \usepackage[margin=1in]{geometry} 
 \else
 \fi
 \fi
\newcommand{\pic}[2]{\includegraphics[width=#1mm]{#2}}

\renewcommand{\phi}{\varphi}
\setenumerate[0]{label=\textnormal{(\alph*)}}
 
 \newcommand{\Z}{\mathbb{Z}}

 \newcommand{\R}{\mathbb{R}}
 \newcommand{\C}{\mathbb{C}}

  \renewcommand{\P}{\mathbb{P}}

  \newcommand{\calF}{\mathcal{F}}
  \newcommand{\calG}{\mathcal{G}}
  \newcommand{\calD}{\mathcal{D}}
  \newcommand{\calL}{\mathcal{L}}
  
  \newcommand{\calN}{\mathcal{N}}
    \newcommand{\calX}{\mathcal{X}}
  \newcommand{\ul}{\underline}

  \newcommand{\p}{\partial}

\DeclareMathOperator{\id}{id}

\renewcommand{\div}{\operatorname{div}}
\renewcommand{\grad}{\operatorname{grad}}

\ifx\elystyle\undefined

 \else
 \if\elystyle0
  \newtheorem{thm}{Theorem}[section]
  \newtheorem{lem}[thm]{Lemma}
    \newtheorem{defin}[thm]{Definition}

 \fi
 \fi
\newcommand{\w}{\omega}
\newcommand{\fr}{\frac{\sqrt{7}}{2}}

\renewcommand{\P}{\ul{p}}
\newcommand{\OOmega}{\ul{\omega}}
\newcommand{\rt}{\tilde{\rho}}
\newcommand{\ut}{\tilde{u}}
 \allowdisplaybreaks
 
\begin{document}
\title[Hunter Self-similar Implosion profiles]{Hunter self-similar implosion profiles for the gravitational Euler-Poisson system}
\author{Ely Sandine}
 \thanks{This material is based upon work supported by the National Science Foundation Graduate Research Fellowship Program under Grant No. DGE 2146752. Any opinions, findings, and conclusions or recommendations expressed in this material are those of the author(s) and do not necessarily reflect the views of the National Science Foundation.}
 \address{Ely Sandine, Department of Mathematics, University of California, Berkeley, 94720.}
 \email{ely\_sandine@berkeley.edu}
\date{\today}
\maketitle

\begin{abstract}
Our result is a construction of infinitely many radial self-similar implosion profiles for the gravitational Euler-Poisson system. The problem can be expressed as solving a system of non-autonomous non-linear ODEs. The first rigorous existence result for a non-trivial solution to these ODEs is due to Guo, Had\v{z}i\'c and Jang \cite{ghj}, in which they construct a solution found numerically by Larson and Penston independently \cite{larson,penston}. The solutions we construct belong to a different regime and correspond to a strict subset of the family of profiles discovered numerically by Hunter \cite{hunter}. Our proof adapts a technique developed by Collot, Rapha\"el and Szeftel in \cite{crs}, in which they study blowup for a family of energy-supercritical focusing semilinear heat equations. In our case, the quasilinearity presents complications, most severely near the sonic point where the system degenerates. 
\end{abstract}

\tableofcontents

\section{Introduction}
We begin by describing the setup of the problem. We consider physical space $\R_x^3$, a gas with density $\rho>0$ and its velocity vector field $u$. The (non-relativistic) gravitational force per unit mass generated by the gas is $-2\grad(\Delta^{-1}\rho)$, units having been chosen so that the gravitational constant is $(2\pi)^{-1}$.  We assume that the gas is ideal and isothermal, and choose units so that the ratio of its pressure to its density is $1$. If we model the evolution of this gas by coupling the compressible Euler equations with self-gravitation, we get the following system of equations for $\rho\colon I_t\times \R_{x}^3\to \R_{>0}$ and $u\colon I_t \times \R_x^3\to \R^3$, which we refer to as the Euler-Poisson system,
\begin{equation}
\begin{cases}
\partial_t \rho+\div(\rho u)=0\\
\partial_t u+u\cdot \grad{u}+\frac{\grad \rho}{\rho}+2\grad(\Delta^{-1} \rho)=0
\end{cases}.
\label{eulerpoissonsystem}
\end{equation}
We will focus on solutions with the property that $\rho(0)\to \infty$ in finite time. We will make the highly simplifying assumption that $\rho$ and $u$ are both radially symmetric. For notational convenience we define
\[
\div_r f(r)=(\partial_r +\frac{2}{r})f\qquad \text{and}\qquad \div_r^{-1}f(r)=\frac{1}{r^2}\int_0^{r}(r')^2 f(r')dr',
\]
so that under radial symmetry \eqref{eulerpoissonsystem} becomes the following system of two equations on $I_t\times \R_{r}$:
\[
\begin{cases}
\partial_t \rho+\div_r(\rho u)=0\\
\partial_t u+u \partial_r u +\frac{\partial_r \rho}{\rho}+2\div_r^{-1}(\rho)=0
\end{cases}.
\]
By time-translation symmetry we may assume the blowup time occurs at $t=0$. We then note that we have the following dynamical scaling symmetry:
\begin{equation}
(\rho(t,r),u(t,r))\text{\quad solves \eqref{eulerpoissonsystem}}\iff \left(\frac{1}{\lambda^2} \rho(\frac{t}{\lambda},\frac{r}{\lambda}),u(\frac{t}{\lambda},\frac{r}{\lambda})\right)\text{\quad solves \eqref{eulerpoissonsystem}}.
\label{rescaling}
\end{equation}
A natural first step in understanding singularity formation is to understand the self-similar solutions, namely those which are fixed by this symmetry (any such non-zero solution blows up at $t=0$). 

We now summarize the rest of the introduction. In Subsection \ref{equationsder}, we will derive the main system of ODEs, state the theorem and discuss related works. We then compare the system to a class of semilinear focusing heat equations in Subsection \ref{nlhsubsect}, highlighting connections that are important for the proving the theorem. In Subsection \ref{outlinesec} we give an outline of the proof and describe the organization of the rest of the paper.

\subsection{Derivation of equations and statement of the theorem}
\label{equationsder}
As a first comment, we shall only study classical solutions of \eqref{eulerpoissonsystem}, by which we mean that $\rho$ and $u$ are $C^1$ and $\rho$ has sufficient decay for the gravitational force to be well-defined. As a comment, we do not require that either the total mass $\int_{\R_x^3}\rho dx$ or the total energy,
\[
E(\rho,u)=\int_{\R^3} \frac{1}{2}\rho \abs{u}^2+\rho\log(\rho)-\abs{\grad \Delta^{-1}(\rho)}^2 dx,
\]
are finite. In fact, the solutions we discuss will have infinite mass and energy.

We will study the system through the technique of similarity variables, by which we mean switching to the variables $(\tilde{\rho},\tilde{u},s,y)$ defined by
\[
\rho(t,r)=\frac{1}{t^2}\tilde{\rho}(\frac{r}{-t}),\qquad u(t,r)=\tilde{u}(\frac{r}{-t})\qquad s=-\log(-t),\qquad y=\frac{r}{-t}.
\]
Performing this change of variables leads to the following system:
\begin{equation}
\begin{cases}
\partial_s \tilde{\rho}+y\partial_y \tilde{\rho}+2\tilde{\rho}+\div_y(\tilde{\rho}\tilde{u})=0\\
\partial_s \tilde{u}+y\partial_y\tilde{u}+\tilde{u}\partial_y \tilde{u}+\frac{\partial_y \tilde{\rho}}{\tilde{\rho}}+2\div_y^{-1}(\tilde{\rho})=0
\end{cases}.
\label{fullselfsimeq}
\end{equation}
We see that $(\rho,u)$ is a solution to \eqref{eulerpoissonsystem} on $(-\infty,0)\times \R_x^3$ which is both radial and invariant under the scaling symmetry \eqref{rescaling} if and only if $(\tilde{\rho},\tilde{u}):=(\rho(-1,y),u(-1,y))$ is an $s$-independent solution to \eqref{fullselfsimeq}. In this case, the equations for $(\tilde{\rho}(y),\tilde{u}(y))$ become (from now on $\div=\div_y$)
\begin{equation}
\begin{cases}
y\partial_y\tilde{\rho}+2\tilde{\rho}+\div(\tilde{\rho}\tilde{u})=0\\
y\partial_y \tilde{u}+\tilde{u}\partial_{y}\tilde{u}+\frac{\partial_y \tilde{\rho}}{\tilde{\rho}}+2\div^{-1}(\tilde{\rho})=0
\end{cases}.
\label{selfsimeq}
\end{equation}

We comment that the use of self-similar implosion solutions to study compressible fluid flow has an extensive history. Existence of self-similar solutions to the compressible Euler equations (without gravity) was shown by Guderley \cite{guderley} (see also Sedov \cite{sedov}). Recently, Merle-Raphaël-Rodnianski-Szeftel \cite{mrrs1} constructed smooth self-similar solutions, and used these profiles to rigorously describe implosion for the defocusing non-linear Schrödinger,  compressible Euler and compressible Navier-Stokes equations \cite{mrrs3,mrrs2}. For further developments, we refer readers to the recent works of Buckmaster-Cao-Labora-Gómez-Serrano \cite{buckmastercaolaboragomezserrano}, Jenssen-Tsikkou \cite{jenssentsikkou} and Cao-Labora-Gómez-Serrano-Shi-Staffilani \cite{caolaboragomezserranoshistaffilani}, and their discussions. In the Euler-Poisson system, the presence of gravity removes a degree of scaling freedom, leading to a different type of ODE problem than the one studied by these authors.

To understand the structure of the system \eqref{selfsimeq} we eliminate the non-local term by rewriting the first equation as
\[
\tilde{\rho}=\div(y\tilde{\rho}+\tilde{u}\tilde{\rho})
\]
and substitute this into the second equation. This leads to the following formulation of the \emph{self-similar Euler-Poisson system} as a system of quasilinear, non-autonomous ODEs:
\begin{equation}
\mqty(\tilde{u}+y & \tilde{\rho}\\\frac{1}{\tilde{\rho}} & \tilde{u}+y)\mqty(\tilde{\rho}'\\ \tilde{u}')+\mqty(\frac{2\tilde{\rho}(\tilde{u}+y)}{y}\\ 2\tilde{\rho}(\tilde{u}+y))=0.
\label{selfsimeq2}
\end{equation}

We now discuss some essential features of the system (we adopt the terminology and discussion of Guo-Hadzic-Jang \cite{ghj}, now formulated using $\tilde{u}$ instead of $\tilde{\omega}$). Firstly, the system is singular at $y=0$, where the equations impose the boundary conditions
\begin{equation}
\tilde{u}(0)=0\qquad \text{and} \qquad \tilde{u}'(0)=-\frac{2}{3}.
\label{bdryconds}
\end{equation}
Secondly, if for some point $y_*$ we have $(\tilde{u}(y_*)+y_*)^2=1$ then the system is singular (the coefficient matrix has determinant zero). Such a point $y_*$ is refered to as a \emph{sonic point}. 

We note that if $\tilde{u}(0)=0$ and $\tilde{u}$ is bounded, then by continuity there exists at least one sonic point $y_*$ for which $u(y_*)+y_*=1$. At such a point, the equations specify that the solution must satisfy either
\[
\tilde{\rho}(y_*)=\frac{1}{y_*},\qquad \tilde{\rho}'(y_*)=-\frac{1}{y_*^2},\qquad \tilde{u}(y_*)=1-y_*,\qquad \tilde{u}'(y_*)=-\frac{1}{y_*}\qquad \text{or}
\]
\begin{equation}
\tilde{\rho}(y_*)=\frac{1}{y_*},\qquad \tilde{\rho}'(y_*)=\frac{1}{y_*}-\frac{3}{y_*^2},\qquad \tilde{u}(y_*)=1-y_*,\qquad \tilde{u}'(y_*)=\frac{1}{y_*}-1.
\label{hunterexpans}
\end{equation}
We use the convention that the first type of first order Taylor expansion is called a Larson-Penston expansion and the second type is called a Hunter expansion.

Two explicit solutions to \eqref{selfsimeq2} are the Friedman solution $(\tilde{\rho}_F,\tilde{u}_F)=(\frac{1}{3},-\frac{2}{3}y)$ and the far-field solution $(\tilde{\rho}_f,\tilde{u}_f)=(y^{-2},0)$. Neither are classical solutions as for the former the gravitational field is not well-defined and for the latter there is a singularity at $y=0$.

The historically first non-trivial solution to \eqref{selfsimeq2} to receive attention is called the Larson-Penston solution, which we now label $(\tilde{\rho}_0,\tilde{u}_0)$, named for the authors who discovered it independently through numerical integration \cite{larson,penston}. The solution is analytic and satisfies the qualitative properties $\tilde{\rho}_0\lesssim \ev{y}^{-2}$, $\tilde{u}_0\lesssim 1$ and $\tilde{u}_0>0$. Moreover, the solution has a unique sonic point $y_*\in (2,3)$ where it has a Larson-Penston expansion. This solution was rigorously proved to exist by Guo-Had\v{z}i\'c-Jang \cite{ghj}. Further solutions were found numerically by Hunter \cite{hunter}, which we label as $\{(\tilde{\rho}_k,\tilde{u}_{k})\}_{k=1}^{\infty}$. These are still analytic and obey $\tilde{\rho}_N\lesssim_N \ev{y}^{-2}$ and $\tilde{u}_N\lesssim_N 1$, but the velocity is no longer positive. These solutions each have a unique sonic point $y_*\in (0,2)$ where $\tilde{\rho}_k$ has a Hunter-type expansion. From the plots of the Larson-Penston and Hunter solutions (see Figure 4 of Maeda-Harada \cite{maedaharada} for instance) we infer that for $k\geq 0$, $\tilde{\rho}_k$ intersects the far-field density $y^{-2}$ precisely $k+1$ times on $(0,\infty)$.
 
We also comment that Shu \cite{shu} criticized the Larson-Penston solution on physical grounds and proposed instead a certain ``expanding wave'' self-similar solution. Further, non-smooth, self-similar solutions were introduced by Summers-Whitworth \cite{whitworthsummers} and were later argued to be unstable by Ori-Piran \cite{op4}. We focus on the Larson-Penston and Hunter solutions for simplicity and also because their stability has been investigated numerically through both linear mode analyses and simulations of the full nonlinear equation. Some references for computations of unstable eigenvalues are Hanawa-Nakayama \cite{hanawanakayama} and Maeda-Harada \cite{maedaharada}. The numerics seem to suggest that for $k\geq 0$ the operator obtained by formally linearizing \eqref{fullselfsimeq} about $(\tilde{\rho}_k,\tilde{u}_k)$ has $k+1$ unstable modes, with the least unstable mode coming from time-translation symmetry. This suggests that the Larson-Penston and Hunter solutions fit into a sort of ``non-linear spectrum'' with the Larson-Penston solution playing the role of the ground state. We also mention that there have also been numerical simulations of the radially-symmetric Euler-Poisson system indicating a self-similar collapse regime. Some sources include Hunter \cite{hunter}, Foster-Chevalier \cite{fosterchevalier} and Brenner-Witelski \cite{brennerwitelski}.

The goal of these notes is to prove the following result which are consistent with an extrapolation of the numerical works on the Hunter solutions.
\begin{thm} For some integer $N\gg 1$ there exists real analytic solutions $\{(\tilde{\rho}_k,\tilde{u}_k)\}_{k=N}^{\infty}$ to \eqref{selfsimeq2} on $[0,\infty)$ subject to the boundary conditions \eqref{bdryconds} and the requirement $\tilde{\rho}_k>0$. In addition, $\tilde{\rho}_k$ intersects $y^{-2}$ exactly $k+1$ times. For these solutions, the sonic point condition $(\tilde{u}_k(y)+y)^2=1$ is true at exactly one point $y=y_*\in (0,2)$, where $\tilde{\rho}$ admits the Hunter-type expansion \eqref{hunterexpans}. Finally, these solutions satisfy the asymptotic bounds $\tilde{\rho}_k\lesssim_k \ev{y}^{-2}$, $\tilde{u}_k \lesssim_k 1$.
\label{mainthm}
\end{thm}
The proof of the theorem is contained in Section \ref{matchingsection} and the structure is based on an approach in Collot-Raphaël-Szeftel \cite{crs} (which we discuss more in the next section) adapted to the quasilinear nature of the Euler-Poisson system. While we solve the same system of ODEs as considered in \cite{ghj}, our method of proof is very different, resulting in different types of solutions. In their work, the authors use nonlinear invariances of the flow to perform a shooting argument resulting in the Larson-Penston solution, which is expected to be the stable profile (and hence the most physically relevant). In contrast, our argument perturbs off of the far-field solution and the isothermal sphere (a solution to the static Euler-Poisson system). Relative to the Larson-Penston solution, the solutions we construct are expected to be highly unstable. In particular, we do not construct the first Hunter solution, which is believed to play a role in describing critical collapse (as argued by Harada-Maeda-Semelin \cite{haradamaedasemelin} for instance). Another technical difference is that in \cite{ghj} they work with real analytic functions, while in our case the bulk of the argument is concerned with $C^1$ functions. To show the solutions we construct are analytic at $0$ and $y_*$ we use the analytic local-wellposedness theory of \cite{ghj} along with a singular ODE argument at the $C^1$-level. As a final point of comparison, we note that the methods of \cite{ghj} have been adapted to more complex systems. In particular, the relativistic analog of the Larson-Penston solution was first studied numerically by Ori-Piran \cite{op1,op2,op3} leading to a prediction of the formation of a naked singularity. Guo-Had\v{z}i\'c-Jang \cite{ghj2} were able to both rigorously construct this profile and show the existence of a spacetime exhibiting a naked singularity. Returning to the Newtonian setting, these authors, along with Schrecker, \cite{ghjs} proved the existence of polytropic analogs of the Larson-Penston solution discovered numerically by Yahil \cite{yahil}. It is currently unclear whether the methods in this work extend to either the polytropic or relativistic scenarios.

We finally comment that while we have so far only discussed self-similar implosions, other dynamical regimes for the Euler-Poisson system have also been investigated mathematically. For pressure laws $p=\rho^{\gamma}$ for $1<\gamma<\frac{4}{3}$, Guo-Hadžić-Jang \cite{continuedcollapse} have shown existence of gravitational collapse which is dust-like rather than self-similar. For $\gamma=\frac{4}{3}$, there exists a class of homologous collapsing solutions found by Goldreich-Weber \cite{goldreichweber} and further investigated mathematically by Makino \cite{makino}, Fu-Lin \cite{fulin}, Deng-Liu-Yang-Yao \cite{dengliuyangyao} and Deng-Xiang-Yang \cite{dengxiangyang}. For a further introduction to the literature we mention the survey article \cite{hadzicsurvey}.
\subsection{Comparison with a class of focusing semilinear heat equations}
\label{nlhsubsect}
In this work, we will view \eqref{eulerpoissonsystem} to be modeled by the following equation on $I_t\times \R_x^3$,\begin{equation}
\partial_t \phi=\Delta \phi+\phi^p,\qquad p>5.
\label{nlh}
\end{equation}
While this equation is in many respects different from \eqref{eulerpoissonsystem}, there are also a few similarities.

At a very heuristic level, the $\rho^{-1}\grad \rho$ term in \eqref{eulerpoissonsystem} and the $\Delta \phi$ term in \eqref{nlh} correspond in the sense that they work to prevent the formation of bell-shaped singularities. In contrast, the $2\grad \Delta^{-1}\rho$ term in \eqref{eulerpoissonsystem} and the $\phi^p$ term in \eqref{nlh} accelerate this type of blow up. In each case, the dynamics are driven by this competition of forces, and implosion results from inertia and the latter force winning out over the former.

More formally, each equation admits a $1$-parameter scaling symmetry, and in each case the equation is energy-supercritical. When we apply the scaling symmetry of the Euler-Poisson system, the energy transforms as
\[
E(\frac{1}{\lambda^2}\rho(\frac{t}{\lambda},\frac{x}{\lambda}),u(\frac{t}{\lambda},\frac{x}{\lambda}))=\lambda E(\rho(t,r),u(t,r)).
\]
As a result, rescaling initial data to have smaller initial energy leads shortens the dynamical timescale.

The nonlinear heat equation also has a 1-dimensional scaling symmetry, namely
\[
\phi(t,x)\text{\quad solves \eqref{nlh}}\iff \frac{1}{\lambda^{2/p-1}}\phi(\frac{t}{\lambda^2},\frac{x}{\lambda})\text{\quad solves \eqref{nlh}}.
\]
In this case, an energy which is dissipated by the flow is
\[
E(\phi)=\int_{\R^3}\abs{\nabla \phi}^2-\frac{1}{p+1}\phi^{p+1}dx.
\]
As $p>5$, we see that this equation too is energy-supercritical.

Another shared quality of these equations is the existence of a 1-parameter family of time-independent radial solutions. For the Euler-Poisson system, if one imposes $u=0$, and lets $\rho=e^w$, then taking a divergence of the second equation in \eqref{eulerpoissonsystem} leads to the equation
\begin{equation}
\Delta w+2e^w=0.
\label{stareq}
\end{equation}
Under the assumption of radial symmetry this is the well-studied isothermal sphere ODE and there is a family of smooth, decreasing solutions parameterized by
\[
\lambda\mapsto Q_\lambda=Q(\cdot/\lambda)-2\log(\lambda),
\]
where $Q$ is the solution to \eqref{stareq} with $Q(0)=0$. As $\lambda\to 0$, the family $Q_{\lambda}$ converges pointwise on $\R^3\setminus \{0\}$ to the far-field solution (which is both stationary and self-similar)
\[
Q_{f}=-2\log(r),\qquad e^{Q_{f}}=\frac{1}{r^2}.
\]
For the nonlinear heat equation, imposing $\partial_t\phi =0$ and radial symmetry leads to the classical Lane-Emden ODE
\[
\phi''+\frac{2}{r}\phi'+\phi^{p}=0.
\]
For $p>5$, this ODE may be studied in the same way as for the isothermal sphere equation, and there is a 1-parameter family of smooth, positive, decreasing solutions parameterized by
\[
\lambda\mapsto \Phi_{\lambda}=\lambda^{-\frac{2}{p-1}}\Phi(\cdot/\lambda),
\]
where $\Phi$ satisfies $\Phi(0)=1$. In this case, as $\lambda\to 0$, the limiting far-field solution is
\[
\Phi_{f}=\frac{c_{\infty}}{r^{\frac{2}{p-1}}},\qquad c_{\infty}=\left(\frac{2}{p-1}(1-\frac{2}{p-1})\right)^{\frac{1}{p-1}}.
\]
One can see Chandrasekhar \cite{chandrasekhar}, for instance, for a treatment of both the isothermal sphere and Lane-Emden equations.

In comparison to the Euler-Poisson system, self-similar singularity formation for \eqref{nlh} is better understood mathematically. For these equation, self-similar variables were used to study the stability of ODE-type blowups by Giga-Kohn \cite{gigakohn1,gigakohn2,gigakohn3}, Merle-Zaag \cite{merlezaag} and Bricmont-Kupiainen \cite{bricmontkupiainen}. The first constructions of self-similar solutions to \eqref{nlh} are due to Troy \cite{troy}, Budd-Qi \cite{buddqi} and Lepin \cite{lepin}. Collot-Rapha\"el-Szeftel \cite{crs} proved finite-codimension stable self-similar blowup for the equation outside of radial symmetry. An ingredient of their proof is a new construction of a discrete set of highly oscillatory self-similar solutions to \eqref{nlh} using a perturbative argument based on the Lane-Emden and far-field solutions (Proposition 2.5 of their monograph). Our proof of Theorem \ref{mainthm} is an adaptation of their argument. We comment that while we learned the method from this paper, there is a long history of using bifuraction-type methods in related contexts. A few sources in this direction are the works of Budd-Norbury \cite{buddnorbury},  Biernat-Bizoń \cite{biernat}, Dancer-Guo-Wei \cite{dancerguowei} and Koch \cite{koch}.

We note that in addition to proving existence of self-similar profiles, \cite{crs} proves the linear and non-linear stability of the associated blowups, questions which are not addressed in this work. Their results also fit into the general theory of self-similar blowup for heat equations, which is much more mathematically advanced than the corresponding theory for the Euler-Poisson system. In particular, in the radial setting there is a classification of self-similar blowups due to Matano-Merle \cite{matanomerle1,matanomerle2}. 

The fact that we are still able to construct solutions in the hyperbolic, quasilinear setting is evidence for the strength of the approach of \cite{crs} and its potential applicability to other equations which are focusing, energy-supercritical and have a 1-parameter scaling symmetry.
\subsection{Proof outline}
\label{outlinesec}
We now present a more detailed outline of the proof of Theorem \ref{mainthm}.

We consider a fixed $y_0\ll 1$ and perturbatively construct $C^1$ solutions defined on $[y_0,\infty)$ and $[0,y_0]$ respectively. We call these the ``exterior'' and ``interior'' regions. Matching these solutions at $y_0$ results in a discrete family. Figure \ref{huntersketch} is a sketch illustrating the qualitative locations of $y_0,y_*$ with respect to the features of the solutions we construct (one should imagine many more oscillations and $y_0\ll y_*$). We note that this picture is similar to the numerical solutions plotted in Figure 4 of \cite{maedaharada} (their independent variable is the reciprocal of ours so one should horizontally flip the image). We also comment that we shall see oscillations in the linear analyses for both the exterior and interior regions as $y\to y_0$.
\begin{figure}[h]
\centering
\pic{100}{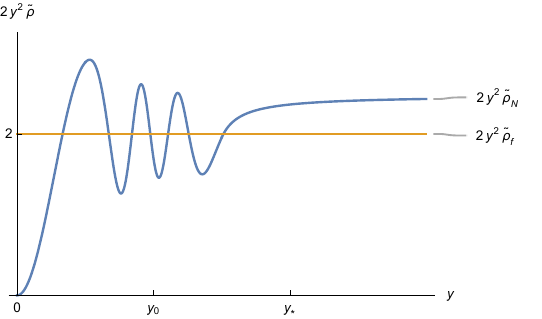}
\caption{Qualitative sketch.}
\label{huntersketch}
\end{figure}

The exterior solutions are modeled on the far-field solution, $(\tilde{\rho}_f,\tilde{u}_f)=(y^{-2},0)$. We let $\epsilon$ be a small parameter, require the solution have a Hunter-type sonic point at $1+\epsilon$ and then rescale the variables to put the degeneracy at $1$ (a similar technique is employed in \cite{ghj}). We then study a linearization of the problem. In comparison to \cite{crs}, our linearized operator $L$ is analogous to their operator $\calL_{\infty}$. In their case, a unique (up to scaling) homogeneous solution $\psi_1$ is identified by imposing decay conditions at $\infty$. In our case, requiring the homogeneous solution be smooth at the singular point $1$ identifies a unique (up to scaling) homogeneous solution $(p_{hom},\omega_{hom})$.

To explain when we mean by this, and illustrate some of the difficulties which do not appear in the problem of \cite{crs}, we now discuss the following model ODE on $(-\infty,\infty)$:
\begin{equation}
\mqty(x & 0\\ 0 & 1)\mqty(f\\ g)'+\mqty(0 & 1\\ 0 & 0)\mqty(f\\ g)=\mqty(F\\ G).
\label{modelequation}
\end{equation}
In this case, we see that two linearly independent homogeneous solutions are
\[
(1,0) \qquad \text{and}\qquad (\log\abs{x},1).
\]
In particular, within the space of homogeneous solutions, there is a $1$-D subspace of distinguished solutions which pass through the origin continuously. Returning to the inhomogeneous problem, if we assume that $f,g$ are a $C^1$ solutions to this equation and plug in $x=0$, we see that from the first line that we must have $g(0)=F(0)$. We also see from inverting the first matrix, that
\[
f'+\frac{g-g(0)}{x}=\frac{F(x)-F(0)}{x}.
\]
Thus, if $f,g$ are $C^1$, the righthand side must be continuous, so $F$ must be differentiable at $0$. Assuming $F,G$ are continuous, and $F$ is differentiable at $0$, we can write down a $1$-D space of solutions to \eqref{modelequation}, namely
\[
g(x)=F(0)+\int_0^x G(x')dx',\qquad f(x)=\int_0^x\frac{F(x')-g(x')}{x'}dx'+C.
\]
A key takeaway is that we need to assume that a component of our non-linearity is differentiable at $0$ in order to have a $C^1$ solution. The linearized operator we study is similar, and the iteration scheme works because under the assumption that our solution is $C^1$, it can be shown that a component of the non-linearity, instead of being merely continuous, is in fact differentiable at $0$.

As another point of comparison, in their problem, $\psi_1$ can be written down using the Tricomi solution to Kummer's equation whereas in our case we can use Gauss's hypergeometric function to express $(p_{hom}(y),\omega_{hom}(y))$. The connection formulae for hypergeometric functions allow us to see that the asymptotics of $(p_{hom}(y),\omega_{hom}(y))$ as $y\to 0$ are given by
\begin{align*}
p_{hom}(y)=\frac{c_1}{y^{1/2}}\sin(\fr \log(y)+d_1)+O(y^{3/2}),\\
\w_{hom}(y)=\frac{c_1}{y^{1/2}}\sin(\fr \log(y)+d_1+\theta_0)+O(y^{3/2}),
\end{align*}
for some explicit $c_1>0$, $d_{1}\in \R/2\pi \Z$ and
\begin{equation}
\theta_0:=\arctan(\frac{\sqrt{7}}{3})+\pi.
\label{theta0def}
\end{equation}

We comment that the frequency $\frac{7}{2}$ is the limit as $p\to\infty$ of the frequency $\omega$ used in the argument of \cite{crs}. To briefly explain where the numerology comes from in our problem (in the exterior region), we note that if one linearizes \eqref{selfsimeq2} about $(\tilde{\rho}_f,\tilde{u}_f)$, the resulting equation, written in matrix form, is
\begin{equation}
\partial_y \mqty(\tilde{\rho}\\ \tilde{u})+\left(
\begin{array}{cc}
 \frac{2 y^2-4}{y^3-y} & \frac{-2}{y^6-y^4} \\
 \frac{2 y^2}{y^2-1} & \frac{2}{y^3-y} \\
\end{array}
\right)\mqty(\tilde{\rho}\\ \tilde{u})=0.
\label{introlinearizedeq}
\end{equation}
This system, viewed as a complex ODE in $y$, has a pole of order $4$ at $y=0$. One obtains a simple pole at $y=0$ if one changes variables to, for instance, $\tilde{p}=y^{2}\tilde{\rho}$, $\tilde{v}=y^{-1}\tilde{u}$. In this case, the resulting system is
\[
\p_y \mqty(\tilde{p}\\ \tilde{v})+\mqty(-\frac{2}{y(y^2-1)} & -\frac{2}{y(y^2-1)}\\ \frac{2}{y(y^2-1)} & \frac{y^2+1}{y(y^2-1)})\mqty(\tilde{p}\\ \tilde{v})=0.
\]
As the pole at $y=0$ is simple, one can locally construct homogeneous solutions by means of the method of Frobenius. In particular, the leading order asymptotics of these homogeneous solutions are governed by the residue of the coefficient of the zeroth order term. We may diagonalize (the negative of) this residue as
\[
\mqty(-2 & -2\\ 2 & 1)=\mqty(1 & 1\\ e^{i\theta_0} & e^{-i\theta_0})\mqty(-\frac{1}{2}+i\frac{\sqrt{7}}{2} & 0\\ 0 & -\frac{1}{2}-i\frac{\sqrt{7}}{2})\mqty(1 & 1\\ e^{i\theta_0} & e^{-i\theta_0})^{-1},
\]
and in doing so we see the constant $\theta_0$ and frequency $\fr$.

\sloppy We use a contractive mapping argument to construct a family of exterior solutions $(\tilde{\rho}_{ext}[\epsilon](y),\tilde{u}_{ext}[\epsilon](y))$ for $\epsilon\ll 1$ which are $C^1$ on $[y_0,\infty)$, solving \eqref{selfsimeq2}, with a Hunter-type sonic point at $y=1+\epsilon$ of the form
\[
\begin{cases}
\tilde{\rho}_{ext}=\frac{1}{y^2}+\frac{\epsilon}{y^2}p_{hom}+\epsilon \rho_{ext}\\
\tilde{u}_{ext}=\epsilon y \omega_{hom}+\epsilon u_{ext}
\end{cases}.
\]

In the interior region, after a rescaling, we approximate $\tilde{\rho}$ by $e^Q$, where $Q$ is the radial solution to \eqref{stareq} with $Q(0)=0$. We define $u_*$ by solving the self-similar conservation of mass equation
\begin{equation}
\begin{cases}
(y\partial_y+2)e^Q+\div(e^Q u_*)=0\\
u_*(0)=0,\quad u_*'(0)=-\frac{2}{3}
\end{cases}.
\label{ustareq}
\end{equation}
By studying the isothermal sphere equation, it is shown that $e^Q$ and $u_*$ have the following asymptotics as $y\to \infty$, for some $c_2>0$ and $d_2\in \R/2\pi \Z$,
\begin{align*}
e^Q&=\frac{1}{y^2}+\frac{c_2 \sin(\fr \log(y)+d_2)}{y^{5/2}}+O(\frac{1}{y^3}),\\
u_*&=c_2y^{1/2} \sin(\fr\log(y)+d_2+\theta_0)+O(1).
\end{align*}
We note that this numerology is consistent with the previous section. In this case, one can read off the frequency $\fr$ from the isothermal sphere ODE by either looking associated phase portrait in homologous variables, or by computing the fundamental solutions to $\Delta+\frac{1}{y^2}$ and the constant $\theta_0$ arises from integrating when solving \eqref{ustareq}. It is shown (again through an iterative construction) that for every $\lambda\leq \frac{y_0}{10}$ one can construct an interior solution $(\tilde{\rho}_{int}[\lambda](y),\tilde{u}_{int}[\lambda](y))$ which is $C^1$ on $[0,y_0]$ (we actually construct a $C^2$ solution, but this distinction doesn't matter for the matching), solving \eqref{selfsimeq2}, with $\tilde{\rho}(0)=\lambda^{-2}$, of the form
\[
\begin{cases}
\tilde{\rho}_{int}=\frac{1}{\lambda^2}e^Q(\cdot/\lambda)+\rho_{int}(\cdot/\lambda)\\
\tilde{u}_{int}=\lambda u_*(\cdot/\lambda)+\lambda^3 u_{int}(\cdot/\lambda)
\end{cases},
\]
where $\rho_{int},u_{int}$ are suitably small (see Lemma \ref{intconstructionlemma} for more details).

With both regimes having been studied, the remaining step is to implement the matching using the parameters $\epsilon,\lambda$. Heuristically this matching is possible because the interior and exterior asymptotic behaviors match as we approach $y_0$ (both regimes witness the oscillation). To simplify the algebra, we note that without loss of generality, $y_0\ll 1$ can be chosen so that
\begin{equation}
\fr\log(y_0)+d_1=\frac{\pi}{2}\mod 2\pi.
\label{wigglecondition}
\end{equation}
We then consider the exterior solutions defined for $\epsilon\lesssim y_0^{5/2}\ll y_0^{1/2}$. We have the following asymptotics uniform with respect to $\epsilon,y_0$
\begin{equation}
\begin{cases}
\tilde{\rho}_{ext}(y_0)-\frac{1}{y_0^2}&=\epsilon \frac{1}{y_0^{5/2}}c_1\sin(\fr \log(y_0)+d_1)+O(\epsilon y_0^{-1/2})\\
\frac{\tilde{u}_{ext}(y_0)}{y_0^3}&=\epsilon  \frac{1}{y_0^{5/2}}c_1\sin(\fr \log(y_0)+d_1+\theta_0)+O(\epsilon y_0^{-1/2})
\end{cases}.
\label{extsummaryasymptot}
\end{equation}

Similarly, for the same fixed $y_0$, we consider interior solutions for $\lambda^{1/2}\lesssim y_0^{5/2}\ll y_0^{1/2}$, we have the following asymptotics uniform with respect to $\lambda,y_0$:
\begin{equation}
\begin{cases}
\tilde{\rho}_{int}(y_0)-\frac{1}{y_0^2}&=\lambda^{1/2}\frac{1}{y_0^{5/2}}c_2 \sin(\fr \log(y_0/\lambda)+d_2)+O(\lambda^{1/2}y_0^{-1/2})\\
\frac{\tilde{u}_{int}(y_0)}{y_0^3}&=\lambda^{1/2}\frac{1}{y_0^{5/2}}c_2\sin(\fr \log(y_0/\lambda)+d_2+\theta_0)+O(\lambda^{1/2}y_0^{-1/2})
\end{cases}.
\label{intsummaryasymptot}
\end{equation}
If we match \eqref{extsummaryasymptot} with \eqref{intsummaryasymptot}, and implement \eqref{wigglecondition}, we get
\[
\begin{cases}
\epsilon c_1+O(\epsilon y_0^2)&=c_2\lambda^{1/2}\cos(\fr\log(\lambda)+d_1-d_2)+O(\lambda^{1/2}y_0^2)\\
\epsilon c_1\cos(\theta_0)+O(\epsilon y_0^2)&=c_2\lambda^{1/2}\cos(\fr\log(\lambda)+d_1-d_2+\theta_0)+O(\lambda^{1/2}y_0^2)
\end{cases}.
\label{matchingintro}
\]
We may then solve the first equation for $\epsilon$, substitute it into the second equation and simplify, so that the matching condition becomes
\[
G[\lambda](y_0)=0
\]
for some function $G[\lambda](y_0)$ such that
\begin{equation}
G[\lambda](y_0)=\sin(\fr \log(\lambda)+d_1-d_2)+O(y_0^2).
\label{secondmatching}
\end{equation}
We then define $\lambda_{k,\pm}=\frac{2}{\sqrt{7}}\exp(-k\pi-d_1+d_2\pm \frac{1}{10})$. Assuming $y_0$ has been chosen sufficiently small, and $k>N$ for some large $N$ (so that $\lambda_{k,\pm}$ are small) we see that by the intermediate value theorem there will exist $\lambda_k\in (\lambda_{k,-},\lambda_{k,+})$ which solve $G[\lambda](y_0)=0$.  Each $\lambda_k$ will correspond to a classical solution obtained by gluing together $(\tilde{\rho}_{ext},\tilde{u}_{ext})$ and $(\tilde{\rho}_{int},\tilde{u}_{int})$. The perturbative nature of this argument also allows us to exactly count the zeroes of $y^2\tilde{\rho}_k-1$ and show that they increase by $1$ as $k$ increases. This oscillation-type result allows us to enumerate the solutions. 

This argument (done rigorously) finishes the bulk of the proof of Theorem \ref{mainthm}. Two things remain: to show each solution has a unique sonic point and to show that the solutions are real analytic. As a first step in counting the sonic points, we note that by construction, any matched solution will have one sonic point near $y=1$. To show that there are no other sonic points, separate arguments are used on $[0,y_0]$ and $[y_0,\infty)$. For the interior, we show $\tilde{u}+y=O(y)$, which rules out sonic points assuming $y_0$ is sufficiently small. For the exterior, we expect $\tilde{u}+y$ to be close to $y$ as we are perturbing off of the far-field solution. Motivated by this, we then show that for the exterior solutions $(\tilde{u}+y)'\geq \frac{1}{2}$, which combined with the absence of sonic points in the interior suffices to show that there is at most one sonic point.

In regards to regularity of the solutions, by standard local ODE theory, it follows that the $C^1$ solutions are locally real analytic at all points except possibly $y\in \{y_*,0\}$. For each of these points, by the methods of \cite{ghj}, it follows that given an appropriate boundary condition there exists a locally a unique real analytic Hunter solution. We then prove that this local solution is the unique one of regularity $C^1$ by linearizing and using its analyticity to derive asymptotics for the principal matrix solution.

The rest of the paper is organized as follows. Sections \ref{extsection} and \ref{intsection} detail the construction of the exterior and interior solutions respectively. The proofs of several lemmas are deferred. Section \ref{matchingsection} details the matching argument and contains the proof of the main theorem. Finally, Sections \ref{extleftovers} and \ref{intleftovers} contain the proofs deferred from Sections \ref{extsection} and \ref{intsection} respectively.

\section{Solving the self-similar equations on $[y_0,\infty)$}
\label{extsection}

\subsection{Choice of variables}
In the exterior region we will first change variables from $\tilde{\rho}, \tilde{u}$ to $\tilde{p},\tilde{\omega}$ defined implicitly by
\[
\tilde{\rho}=\frac{\tilde{p}}{y^2}\qquad \text{and} \qquad \tilde{u}=y(\tilde{\omega}-1).
\]
In this case, \eqref{selfsimeq2} with the Hunter boundary conditions become
\begin{equation}
\begin{cases}
(\tilde{p}\tilde{\omega}y)'-\tilde{p}=0\\
\tilde{\omega}(y(\tilde{\omega}-1))'+\frac{\tilde{p}'}{y \tilde{p}}+\frac{2}{y^2}(\tilde{p}\tilde{\omega}-1)=0\\
\tilde{p}(y_*)=y_*,\qquad \tilde{p}'(y_*)=y_*-1,\qquad \tilde{\omega}(y_*)=\frac{1}{y_*},\qquad \tilde{\omega}'(y_*)=0
\end{cases}.
\label{maedavars}
\end{equation}
The principal motivation for these variables is that the far-field solution takes the relatively simple form $(\tilde{p},\tilde{\omega})=(1,1)$.

We now perform a $y_*$-dependent change of variables that will fix the domain $[y_0,\infty)$, move the sonic point from $y_*$ to $1$ and simplify the values of $p,\omega$ at the sonic point (a similar change of variables is done in \cite{ghj}). In particular, we define
\[
\alpha=\frac{1-y_0}{y_*-y_0},\qquad \beta=\frac{(y_*-1)y_0}{y_*-y_0},\qquad z=\alpha y+\beta,
\]
\[
\tilde{p}=y_* \P(\alpha y+\beta),\qquad \tilde{\omega}=\frac{1}{y_*}\OOmega(\alpha y+\beta).
\]
Under this change of variables, \eqref{maedavars} becomes the following system for $\P(z),\OOmega(z)$,
\begin{equation}
\begin{cases}
\left(\P\OOmega(z-\beta)\right)'-y_* \P=0\\
\OOmega\left((\OOmega-y_*)(z-\beta)\right)'+(\alpha y_*)^2\frac{\P'}{\P(z-\beta)}+2\frac{(\alpha y_*)^2}{(z-\beta)^2}(\P\OOmega -1)=0.\\
\P(1)=1,\qquad \P'(1)=\frac{y_*-1}{\alpha y_*},\qquad \OOmega(1)=1,\qquad \OOmega'(1)=0.
\end{cases}
\label{POmegaeq}
\end{equation}
We next let $y_*=1+\epsilon$ for some parameter $\epsilon$ small. As a note, for notational convenience we continue to use both $\epsilon$ and $y_*$ in the equations. We then switch variables to $\hat{p}(z),\hat{\omega}(z)$ defined by
\[
\P=1+\frac{\epsilon}{\alpha y_*}(\hat{p}-1),\qquad \OOmega=1+\epsilon(\hat{\omega}+1).
\]
We then define the following quantities (which depend on $\epsilon$ as well as on $\hat{p},\hat{\omega}$):
\begin{align*}
A_1(\epsilon, \hat{p},\hat{\omega})&=
\epsilon (1+\frac{\epsilon}{y_*\alpha}(\hat{p}-1))(\hat{\omega}(z-\beta))'+\frac{\epsilon}{y_*\alpha}\hat{p}'(y_*+\epsilon \hat{\omega})(z-\beta),\\
A_2(\epsilon, \hat{p},\hat{\omega})&:=\epsilon(y_*+\epsilon \hat{\omega})(\hat{\omega}(z-\beta))'+\epsilon \hat{p}'\frac{\alpha y_*}{z-\beta}\frac{1}{(1+\frac{\epsilon}{y_*\alpha}(\hat{p}-1))}\\
&+\frac{2(y_*\alpha)^2}{(z-\beta)^2}\left((1+\frac{\epsilon}{y_*\alpha}(\hat{p}-1))(y_*+\epsilon\hat{\omega})-1\right).
\end{align*}
We also define the linear operator
\[
L(\hat{p},\hat{\omega})=\mqty(L_1(\hat{p},\hat{\omega})\\ L_2(\hat{p},\hat{\omega})):=\mqty(z & z\\ \frac{1}{z} & z)\mqty(\hat{p}'\\ \hat{\omega}')+\mqty(0 & 1\\ \frac{2}{z^2} & \frac{2}{z^2}+1)\mqty(\hat{p}\\ \hat{\omega}).
\]
We then note that formally $(A_1(\epsilon, \hat{p},\hat{\omega}),A_2(\epsilon, \hat{p},\hat{\omega}))=\epsilon L+O(\epsilon^2)$. We can thus define
\[
\hat{N}_{i}(\epsilon \hat{p},\hat{\omega}):=\frac{1}{\epsilon}(\epsilon L_i(\hat{p},\hat{\omega})-A_i(\hat{p},\hat{\omega})),\qquad \hat{N}(\epsilon, \hat{p},\hat{\omega}):=(\hat{N}_1(\epsilon, \hat{p},\hat{\omega}),\hat{N}_2(\epsilon, \hat{p},\hat{\omega})).
\]
With this notation, \eqref{POmegaeq} becomes
\begin{equation}
\begin{cases}
L(\hat{p},\hat{\omega})=\hat{N}(\epsilon, \hat{p},\hat{\omega})\\
\hat{p}(1)=1,\quad \hat{p}'(1)=1,\quad \hat{\omega}(1)=-1,\quad \hat{\omega}'(1)=0
\end{cases}.
\label{prediffeq}
\end{equation}
We view this system as a family of ODEs parameterized by $\epsilon$. In the next section we will study the linearized operator. In the subsequent section we will use this information to show that the equation is well-posed for $\epsilon$ close to zero.
\subsection{Properties of the linearized operator}
We now consider the following equation on $(0,\infty)$:
\begin{equation}
L(p,\omega)=(P,\Omega).
\label{linearproblem}
\end{equation}
We remark that $L$ is a variable coefficient operator which is singular at $z=1$. We shall see that when solving from $z=1$ the structure is similar to the model problem \eqref{modelequation} discussed in the introduction. In comparison with that problem, the constraints are similarly easy to derive, but due to the variable coefficients it is harder to identify the distinguished subspace of homogeneous solutions and to show the existence of solutions for the inhomogeneous problem.

In particular, if we assume $p,\omega\in C^1$, then multiplying the system on the left by $\mqty(1 & -1)$, $\mqty(1 & 1)$ respectively, and substituting in $z=1$ give
\begin{equation}
\begin{cases}
p(1)+\omega(1)=\frac{P(1)-\Omega(1)}{-2}\\
p'(1)+\omega'(1)+p(1)+2\omega(1)=\frac{P(1)+\Omega(1)}{2}
\end{cases}.
\label{firstconstraints}
\end{equation}
The first equation yields the constraint $p(1)+\omega(1)=\frac{P(1)-\Omega(1)}{-2}$. We next note that for $z\neq 1$ we may invert the highest order coefficient in \eqref{linearproblem} to obtain
\begin{equation}
\mqty(p\\ \omega)'+\frac{1}{z(z^2-1)}\mqty(-2 & -2\\ 2 & z^2+1)\mqty(p\\ \omega)=\mqty(z & z\\ \frac{1}{z} & z)^{-1}\mqty(P\\ \Omega).
\label{standardform}
\end{equation}
Multiplication by $\mqty(1 & -1)$ then yields that as $z\to 1$
\[
p'=\frac{(P-\Omega)+2(p+\omega)}{2(z-1)}+O(1).
\]
In particular, if $p,\omega\in C^1$ we must have that $\frac{(P-\Omega)-(P(1)-\Omega(1))}{z-1}$ is continuous, or in other words $P-\Omega$ is differentiable at $z=1$. In this case, we again multiply \eqref{linearproblem} by $\mqty(1 & -1)$, and this time expand to order $o(z-1)$. In this case, we obtain
\[
2(z-1)p'(z)-2(1-2(z-1))(p(z)-\omega(z))=P(1)-\Omega(1)+(P-\Omega)'(1)(z-1)+o(z-1).
\]
Simplifying this using \eqref{firstconstraints} gives the equation
\[
\omega'(1)=-(P(1)+\Omega(1))-\frac{1}{2}(P-\Omega)'(1).
\]
This equation, along with \eqref{firstconstraints} imply that if $P,\Omega,\frac{P-\Omega}{z-1}$ are continuous and
\begin{equation}
P(1)=\Omega(1)=(P-\Omega)'(1)=0.
\label{nonlinconstraints}
\end{equation}
 then if a solution to $L(p,\omega)=(P,\Omega)$ exists, then up to scaling by a constant, its Taylor expansion at $z=1$ must be
\begin{equation}
p=1+(z-1)+o(z-1),\qquad \omega=-1+o(z-1).
\label{homogeneoustaylor}
\end{equation}
In particular, any homogeneous solution must have such an expansion.

Our next goal is to find the distinguished subspace of homogeneous solutions and to solve inhomogeneous problem for data subject to the constraints. A key observation is that the equation $L(p,\omega)=0$ can be transformed into a rational-coefficient 2nd order ODE with three simple singularities. This transformation allows one to write down homogeneous solutions using Gauss's hypergeometric function.

We use the terminology that a fundamental matrix associated to $L$ on an interval $I$ is a smooth matrix-valued function $U(z)$ defined on $I$ such that $L U=0$. We observe that if $U$ is a fundamental matrix, and $M$ is a constant invertible matrix then $U\cdot M$ is also a fundamental matrix, and moreover any two fundamental matrices are related in this way.
\begin{lem}
\begin{enumerate}[wide,label=(\alph*)]
\item
There exists unique $(p_{hom}(z),\omega_{hom}(z))$ analytic on $(0,\infty)$ such that $L(p_{hom},\omega_{hom})=0$ and such that their first order Taylor polynomials at $z=1$ agree with \eqref{homogeneoustaylor}.
\item
There exists analytic fundamental matrices $U_{\infty}(z),U_0(z)$ for $L$ defined on $(1,\infty)$ and $(0,1)$ respectively. Moreover, these can be chosen such that as $z\to \infty$
\[
U_{\infty}(z)=\mqty(1 & 0\\ 0 & \frac{1}{z})+\mqty(O(\frac{1}{z^2}) & O(\frac{1}{z^2})\\ O(\frac{1}{z^3}) & O(\frac{1}{z^3})),\qquad U_{\infty}^{-1}(z)=\mqty(1 & 0\\ 0 & z)+\mqty(O(\frac{1}{z^2}) & O(\frac{1}{z})\\ O(\frac{1}{z^2}) & O(\frac{1}{z})).
\]
and as $z\to 0$
\[
U_{0}(z)=\frac{1}{z^{1/2}}\mqty(\cos(\frac{\sqrt{7}}{2}\log(z))& \sin(\frac{\sqrt{7}}{2}\log(z))\\ \cos(\fr\log(z)+\theta_0) &\sin(\fr\log(z)+\theta_0))+O(z^{3/2}).
\]
\[
U_{0}^{-1}(z)=\frac{z^{1/2}}{\sin(\theta_0)}\mqty(\sin(\frac{\sqrt{7}}{2}\log(z)+\theta_0)& -\sin(\frac{\sqrt{7}}{2}\log(z))\\ -\cos(\fr\log(z)+\theta_0) &\cos(\fr\log(z)))+O(z^{5/2}).
\]
\item
The asymptotics of the fundamental solution $(p_{hom},\omega_{hom})$ as $z\to \infty,0$ respectively are, for some $\mu_3,\mu_4,c_1,d_1$ with $\sqrt{\mu_3^2+\mu_4^2},c_1\neq 0$,
\begin{equation}
\mqty(p_{hom}\\ \omega_{hom})=\mqty(\mu_3\\ \frac{\mu_4}{z})+\mqty(O(\frac{1}{z^2})\\ O(\frac{1}{z^3}))
\label{infinityasymptot}
\end{equation}
\begin{equation}
\mqty(p_{hom}\\ \omega_{hom})=\frac{c_1}{z^{1/2}}\mqty(\sin(\fr \log(z)+d_1)\\ \sin(\fr \log(z)+d_1+\theta_0))+O(z^{3/2}).
\label{0asymptot}
\end{equation}
\item
Also, there exists constant invertible matrices $M_0,M_{\infty}$ such that as $z\to 1^{\pm}$, we have respectively
\begin{equation}
U_\infty M_\infty,U_0 M_0=\mqty(1 & \log\abs{z-1}+\frac{1}{2}\\ -1 & -\log\abs{z-1}+\frac{1}{2})+\mqty(O(z-1) & O(\abs{z-1}\log\abs{z-1})\\ O(z-1) & O(\abs{z-1}\log\abs{z-1}))
\label{near1asymptot}
\end{equation}
\begin{multline*}
(U_\infty M_\infty)^{-1},(U_0 M_0)^{-1}=\mqty(-\log\abs{z-1}+\frac{1}{2} & -\log\abs{z-1}-\frac{1}{2}\\ 1 & 1)\\
+\mqty(O(\abs{z-1}\log\abs{z-1}) & O(\abs{z-1}\log\abs{z-1})\\ O(z-1) & O(z-1))
\end{multline*}
\item
If we define the principal matrix solution on $(0,1)\times (0,1)\cup (1,\infty)\times (1,\infty)$ to be
\[
S(z,z')=\begin{cases}
U_0(z)U_0^{-1}(z')\mqty(z' & z'\\ \frac{1}{z'} & z')^{-1}, & (z,z')\in (0,1)\times (0,1)\\
U_\infty(z)U_\infty^{-1}(z')\mqty(z' & z'\\ \frac{1}{z'} & z')^{-1}, & (z,z')\in (1,\infty)\times (1,\infty)
\end{cases}.
\]
then if $P(z),\Omega(z),\frac{P-\Omega-(P(1)-\Omega(1))}{z-1}$ are continuous on $(0,\infty)$ and satisfy the constraints \eqref{nonlinconstraints} then
\begin{equation}
(p,\omega):=R(P(z),\Omega(z)):=\int_{1}^z S(z,z')\mqty(P(z')\\ \Omega(z'))dz'
\label{Rdef}
\end{equation}
defines an element of $C^1((0,\infty))$ which solves $L(p,\omega)=(P,\Omega)$ and 
\[
p(0)=p'(0)=\omega(0)=\omega'(0)=0.
\]
\end{enumerate}
\label{explicitlemma}
\end{lem}
For the proof, see Subsection \ref{explicitlemmaproof}.

\subsection{Difference equations and iteration scheme}
We now consider \eqref{prediffeq}, and substitute the following ansatz, where $(p,\omega)$ are now unknown,
\[
(\hat{p},\hat{\omega})=(p_{hom},\omega_{hom})+(p,\omega).
\]
This gives the equation
\begin{equation}
\begin{cases}
L(p,\omega)=N(\epsilon, p,\omega)\\
p(1)=p'(1)=\omega(1)=\omega'(1)=0
\end{cases}.
\label{fulldiffeq}
\end{equation}
where we have defined
\[
N(\epsilon, p,\omega)=\hat{N}(\epsilon, p_{hom}+p,\omega_{hom}+\omega).
\]
The solution operator defined in Lemma \ref{explicitlemma}.c allows us to rewrite \eqref{fulldiffeq} as the integral equation
\begin{equation}
(p,\omega)=R N(\epsilon, p,\omega).
\label{integraleq}
\end{equation}
As a note, for $N(\epsilon,p,\omega)$ to be well-defined, we impose the mass-positivity condition
\[
1+\frac{\epsilon}{y_*\alpha}(p_{hom}+p-1)>0.
\]
on solutions to \eqref{integraleq}. The integral equation can then be solved on $[y_0,\infty)$ using a classical fixed point method with suitable Banach spaces. Before defining these Banach spaces, we first introduce the notation
\[
D_{\infty}f(z)=-z^2\partial_z f.
\]
\begin{defin}
We define our solution norm to be
\begin{align*}
\norm{(p,\w)}_{X_{y_0}}:=&\sup_{z\geq 1}\abs{p}+\abs{z \w}+\abs{D_{\infty}p}+\abs{D_{\infty}(z\omega)}
+\sup_{y_0\leq z\leq 1}z^{1/2}(\abs{p}+\abs{\w})+z^{3/2}(\abs{p'}+\abs{\w'})).
\end{align*}
We define our solution Banach space to be
\begin{align*}
X_{y_0}:=&\{(p,\w)\in C^1([y_0,\infty))\colon\quad \lim_{z\to \infty}(D_{\infty}(p),D_\infty(z\omega))\quad \text{exists},\\
&\norm{(p,\w)}_{X_{y_0}}<\infty\qquad \text{and} \qquad p(1)=\omega(1)=p'(1)=\omega'(1)=0\}.
\end{align*}
\label{lindef}
\end{defin}

\begin{defin} We define our nonlinear norm to be
\begin{align*}
\norm{\mqty(P,\Omega)}_{N_{y_0}^0}:=\sup_{\frac{1}{2}\leq z\leq 2}\abs{\frac{P-\Omega}{z-1}}+\sup_{z\geq 1}\abs{z P}+\abs{z^2 \Omega}+\sup_{y_0\leq z\leq 1}(\abs{z P}+\abs{z^3 \Omega}).
\end{align*}
We define our ``nonlinear'' Banach space by letting
\begin{align*}
N_{y_0}&:=\{(P,\Omega)\in C^0([y_0,\infty))\colon \frac{P-\Omega}{z-1}\in C^0([y_0,\infty)),\quad \lim_{z\to\infty}(z P,z^2 \Omega)\quad \text{exists},\\
&\qquad \norm{(P,\Omega)}_{N_{y_0}}<\infty\quad \text{and}\qquad P(1)=\Omega(1)=(P-\Omega)'(1)=0\}.
\end{align*}
\end{defin}
As a comment, the first term in the norm is non-standard and it requires a computation to show that it is, in fact, finite for non-linearities we consider.

We now state the three lemmas which will lead to our construction of exterior solutions.
\begin{lem} For $y_0\leq \frac{1}{2}$ there exists a $C_1\geq 1$ (independent of $y_0$) so that if $(P,\Omega)\in N_{y_0}$
  then 
\[
(p,\omega):=R(P,\Omega)\in X_{y_0}\qquad \text{and}\qquad \norm{(p,\omega)}_{X_{y_0}}\leq C_1y_0^{-1/2} \norm{(P,\Omega)}_{N_{y_0}}.
\]
\label{linearlemma}
\end{lem}
For the proof, see Subsection \ref{linearlemmaproof}.

\begin{lem}
If $y_0\leq \frac{1}{2}$, then for all $C>0$ there exists a $C_0(C)\ll 1$ (independent of $y_0$) such that if $\epsilon\leq C_0(C)y_0^{1/2}$ (ensuring mass-positivity) then there exists $C_2(C)\geq 1$ (independent of $y_0,\epsilon$), such that,
\begin{enumerate}[wide,label=(\alph*)]
\item
If $(p,\omega)\in X_{y_0}$, $\norm{(p,\w)}_{X_{y_0}}\leq C$ then
\begin{equation}
\mqty(P,\Omega):=N(\epsilon, p,\omega)\in N_{y_0}\quad \text{and}\qquad \norm{(P,\Omega)}_{N_{y_0}}\leq \abs{\epsilon}C_2.
\label{bddnesseq}
\end{equation}
\item
We have that $(\epsilon,p,\omega)\mapsto N(\epsilon, p,\omega)$ is $C^{1}$ as a function of $(\epsilon,p,\omega)$ with domain
\[
\calD_{C}:=(-C_0(C)y_0^{1/2},C_0(C)y_0^{1/2})\times \{(p,\omega)\in X_{y_0}\colon \norm{p,\omega}_{X_{y_0}^1}<C\}
\]
valued in $N_{y_0}$, with 
\begin{equation}
\norm{(\pdv{N}{p},\pdv{N}{\omega})}_{L(X_{y_0},N_{y_0})}\leq C_2 \abs{\epsilon_0},\qquad \norm{\pdv{N}{\epsilon}}_{N_{y_0}}\leq C_2
\label{derivativebddnesseq}
\end{equation}
for all $(\epsilon_0,p_0,\omega_0)\in \calD_{C}$. In addition, if $(\epsilon_1,p_1,\omega_1)\in \calD_{C}$ we have the following Lipschitz control of the derivative
\begin{equation}
\norm{\pdv{N}{\epsilon}|_{(\epsilon_1,p_1,\omega_1)}-\pdv{N}{\epsilon}|_{(\epsilon_0,p_0,\omega_0)}}_{N_{y_0}}\leq C_2(\abs{\epsilon_1-\epsilon_0}+\norm{(p_1,\omega_1)-(p_0,\omega_0)}_{X_{y_0}}),
\label{derivativelipschitz1}
\end{equation}
\begin{multline}
\norm{(\pdv{N}{p},\pdv{N}{\omega})|_{(\epsilon_1,p_1,\omega_1)}-(\pdv{N}{p},\pdv{N}{\omega})|_{(\epsilon_0,p_0,\omega_0)}}_{L(X_{y_0},N_{y_0})}\\
\qquad\leq C_2\left(\abs{\epsilon_1-\epsilon_0}+\max(\abs{\epsilon_0},\abs{\epsilon_1})\norm{(p_1,\omega_1)-(p_0,\omega_0)}_{X_{y_0}}\right).
\label{derivativelipschitz2}
\end{multline}
\end{enumerate}
\label{nonlinearlemma}
\end{lem}
For the proof, see Subsection \ref{nonlinearlemmaproof}.

 \begin{lem}
\begin{enumerate}[wide,label=(\alph*)]
\item
There exists a universal constant $C_3$ so that if one fixes $y_0\leq 1$, then for $\abs{\epsilon}\ll y_0^{1/2}$, there exists a unique solution $(p[\epsilon](y),\w[\epsilon](y))\in X_{y_0}$ to \eqref{integraleq} with (all inequalities independent of $y_0,\epsilon$)
\[
\norm{(p[\epsilon],\w[\epsilon])}_{X_{y_0}}\leq C_3 \abs{\epsilon} y_0^{-1/2}.
\]
\item
Moreover, for this range of $\epsilon$, we have that $\epsilon\mapsto (p[\epsilon],\w[\epsilon])$ is $C^1$ valued in $X_{y_0}$. In particular, we may define $p_{lin}[\epsilon],\omega_{lin}[\epsilon]$ to be the unique solution to
\[
(p_{lin},\omega_{lin})=R\pdv{N}{\epsilon}|_{(\epsilon,p[\epsilon],\omega[\epsilon])}+R(\pdv{N}{p},\pdv{N}{\omega})|_{(\epsilon,p[\epsilon],\omega[\epsilon])}\cdot(p_{lin},\omega_{lin})
\]
 and
 \[
\pdv{p[\epsilon]}{\epsilon},\pdv{\omega[\epsilon]}{\epsilon}=(p_{lin}[\epsilon],\omega_{lin}[\epsilon]).
 \]
In addition, we have
\[
(p[0],\omega[0])=0,\qquad \norm{\pdv{p[\epsilon]}{\epsilon},\pdv{\omega[\epsilon]}{\epsilon}}_{X_{y_0}}\lesssim y_0^{-1/2}.
\]

\end{enumerate}
\label{extconstructionlemmapre}
\end{lem}
For the proof, see Subsection \ref{extconstructionlemmapreproof}.

In preparation for the matching, we convert this statement to the original variables. In particular, we define for each $\epsilon$,
\begin{align*}
\rho_{ext}[\epsilon](y)&:=\frac{1}{\epsilon y^2}\left(y_*(1+\frac{\epsilon}{y_*\alpha}(p_{hom}(\alpha y+\beta)+p[\epsilon](\alpha y+\beta)-1))-1-\epsilon p_{hom}(y)\right),\\
&=\frac{1}{y^2}\left((1-\frac{1}{\alpha})+\left(\frac{1}{\alpha}p_{hom}(\alpha y+\beta)-p_{hom}(y)\right)+\frac{1}{\alpha}p[\epsilon](\alpha y+\beta)\right),\\
u_{ext}[\epsilon](y)&=\frac{y}{\epsilon}\left(\frac{1}{y_*}(1+\epsilon(\omega_{hom}+\omega+1))-1-\epsilon \omega_{hom}\right),\\
&=y\left(\left(\frac{1}{y_*}\omega_{hom}(\alpha y+\beta)-\omega_{hom}(y)\right)+\frac{1}{y_*}\omega[\epsilon](\alpha y+\beta)\right).
\end{align*}

We can then compute
\begin{align*}
\rho_{ext}[\epsilon](y_0)&=\frac{1}{y_0^2}\left(1-\frac{1}{\alpha}+(\frac{1}{\alpha}-1)p_{hom}(y_0)+\frac{1}{\alpha}p[\epsilon](y_0)\right),\\
u_{ext}[\epsilon](y_0)&=y_0\left((\frac{1}{y_*}-1)\omega_{hom}(y_0)+\frac{1}{y_*}\omega[\epsilon](y_0)\right).
\end{align*}
In particular, we see that if $p[\epsilon](y_0),\omega[\epsilon](y_0)$ are $C^1$ then so are $\rho[\epsilon](y_0)$ and $u[\epsilon](y_0)$. Moreover, as $\norm{(y^{-2},0)}_{X_{y_0}},\norm{(y^{-2}p_{hom}',y\omega_{hom}')}_{X_{y_0}}<\infty$, we have that 
\[
\norm{y^2\rho_{ext},y^{-1}u_{ext}}_{X_{y_0}}\lesssim \epsilon+\norm{(p[\epsilon],\omega[\epsilon])}_{X_{y_0}}.
\]
Using these variables leads to the following corollary of \ref{extconstructionlemmapre}.
\begin{lem} For a fixed $y_0\leq 1$, $\epsilon \ll y_0^{1/2}$, there exists functions $\tilde{\rho}$, $\tilde{u}$ which are $C^1$ on $[y_0,\infty)$ with $\tilde{\rho}>0$ satisfying \eqref{selfsimeq2} with a Hunter-type sonic point at $y_*=1+\epsilon$ of the form
\[
\begin{cases}
\tilde{\rho}_{ext}(y)=\frac{1}{y^2}+\frac{\epsilon}{y^2}p_{hom}(y)+\epsilon \rho_{ext}[\epsilon](y)\\
\tilde{u}_{ext}(y)=\epsilon y \omega_{hom}(y)+\epsilon u_{ext}[\epsilon](y)
\end{cases},
\]
\label{extconstructionlemma}
where $p_{hom},\w_{hom}$ are defined in Lemma \ref{explicitlemma}. We have the following bounds (independent of $y_0,\epsilon$):
\begin{equation}
\sup_{y_0\leq y\leq 1}\abs{y^{5/2}\rho_{ext}},\abs{y^{-1/2} u_{ext}},\abs{y^{7/2}\rho_{ext}'},\abs{y^{1/2}u_{ext}'}\lesssim \epsilon y_0^{-1/2}\qquad \text{and}
\label{exterror}
\end{equation}
\begin{equation}
\sup_{1\leq y\leq \infty}\abs{y^2\rho_{ext}},\abs{u_{ext}},\abs{y^3\rho_{ext}'},\abs{yu_{ext}'}\lesssim \epsilon y_0^{-1/2}.
\label{exterrorb}
\end{equation}
 Moreover, we have that $\epsilon\mapsto (\epsilon \rho_{ext}[\epsilon],\epsilon u_{ext}[\epsilon])(y_0)$ is $C^1$ on a neighborhood of $0$. In addition, as $\epsilon\to 0$ we have (inequality independent of $y_0$)
\begin{equation}
\epsilon \rho[\epsilon](y_0)=O(\epsilon^2 y_0^{-3}),\qquad \epsilon u[\epsilon]=O(\epsilon^2).
\label{epsiloncontrol}
\end{equation}
Finally, we have the following bound on the velocity (independent of $y_0,\epsilon$)
\begin{equation}
(\tilde{u}+y)'\geq \frac{1}{2}.
\label{exthuntercontrol}
\end{equation}
\end{lem}
\begin{proof}
We start with Lemma \ref{extconstructionlemmapre} and implement the change of variables. This yields proves the lemma up to and including \eqref{epsiloncontrol}. To prove \eqref{exthuntercontrol} one computes
\[
\tilde{u}'=\epsilon(yw_{hom}'+w_{hom}+u').
\]
We then use the asymptotics for $\omega_{hom}$ given in Lemma \ref{explicitlemma} and \eqref{exterror}, \eqref{exterrorb} for $u'$ to derive the following bounds on $[y_0,\infty)$:
\[
w_{hom}\lesssim \frac{1}{y}\ev{\frac{1}{y}}^{-1/2},\qquad y w_{hom}'\lesssim \frac{1}{y}\ev{\frac{1}{y}}^{-1/2},\qquad u'\lesssim (\epsilon y_0^{-1/2})\frac{1}{y}\ev{\frac{1}{y}}^{-1/2}.
\]
If we then use $y\geq y_0$, we get the global bound $\tilde{u}'\lesssim \epsilon y_0^{-1/2}$. As we assuming $\epsilon y_0^{-1/2}\ll 1$, by shrinking the range of admissible $\epsilon$ if necessary, we obtain $\abs{\tilde{u}'}\leq 1/2$, yielding \eqref{exthuntercontrol}.
\end{proof}

\section{Solving the self-similar equations on $[0,y_0]$}
\label{intsection}
\subsection{Choice of variables and review of the isothermal sphere equation}
In the interior region, we will change variables from $(\tilde{\rho},\tilde{u})$ to $(\tilde{w},\tilde{u})$, where $\tilde{w}$ is defined by the relation
\[
\tilde{\rho}=e^{\tilde{w}}.
\]
In this region, we will start with the formulation \eqref{selfsimeq} and take a divergence of the second equation, which looses no generality under the assumption that $\tilde{\rho},\tilde{u}$ are $C^2$ on $[0,y_0]$. When we do this, the system becomes
\begin{equation}
\begin{cases}
(2+y\partial_y)(e^{\tilde{w}})+\div(e^{\tilde{w}}\tilde{u})=0\\
\div((y+\tilde{u})\partial_{y}\tilde{u})+\Delta \tilde{w}+2e^{\tilde{w}}=0\\
\tilde{u}(0)=0,\qquad \tilde{u}'(0)=-\frac{2}{3}
\end{cases}.
\label{originselfsimeq}
\end{equation}
The motivation for this formulation of the equation is that it highlights the connection with the isothermal sphere equation, \eqref{stareq}. We next review properties of this equation. We switch the radial coordinate from ``$r$'' to ``$y$'' to maintain consistency.

As mentioned in the introduction there exists a ground state solution $Q$, uniquely determined within the scaling symmetry class by requiring
\begin{equation}
Q(0)=0.
\label{Qdef}
\end{equation}
 The asymptotics of this ground state as $y\to \infty$ are
\begin{equation}
Q=-2\log(y)+O(y^{-1/2}).
\label{Qasymexp}
\end{equation}
Converting to the mass-density variable gives
\begin{equation}
e^Q=\frac{1}{y^2}+O(\frac{1}{y^{5/2}}).
\label{exponentialasymptot}
\end{equation}
Moreover, from the equation \eqref{stareq}, it follows that the asymptotic relations \eqref{Qasymexp} and \eqref{exponentialasymptot} can be differentiated arbitrarily many times. We also note that the linearized operator associated to \eqref{stareq} about $Q_{f}=-2\log(y)$, 
\[
\tilde{H}:=-(\Delta+\frac{2}{y^2}),
\]
has a two-dimensional kernel spanned by the real and imaginary parts of $y^{-\frac{1}{2}+i\fr}$.

We now discuss homogeneous solutions for the linearization about $Q$ of the isothermal sphere equation. We define an operator
\[
H:=-(\Delta+2e^Q).
\]
From the scaling symmetry, we see that one element of the kernel is 
\[
v_1:=y\partial_y Q+2.
\]
If one takes any solution $v$ to $Hv=0$, and defines the Wronskian with $v_1$ to be $W=v_1v'-v v_1'$, then differentiating the Wronskian identity yields $W=Cy^{-2}$. By choosing $C=1$ and solving the ODE $v_1v'-v v_1'=W$ for $v$, we can pick another homogeneous solution, linearly independent from $v_1$, to be (here $Y_0$ is an arbitrary constant)
\[
v_2:=-v_1(\int_{y}^{Y_0}\frac{1}{(v_1)^2(y')^2} dy').
\]
We note that $v_1,v_2=O(\frac{1}{y^{1/2}})$ as $y\to\infty$. In fact, one can obtain the following more detailed asymptotics of $v_1,v_2$.
\begin{lem} As $y\to\infty$, the functions $v_1,v_2$  satisfy, for some $c_3,c_4\in \R^*$ and $d_3,d_4\in \R/2\pi \Z$,
\begin{equation}
(v_1,v_2)=\left(c_3\frac{\sin(\fr\log(y)+d_3)}{y^{1/2}},c_4\frac{\sin(\fr\log(y)+d_4)}{y^{1/2}} \right)+O(\frac{1}{y}),
\label{kernelasymptoticseq}
\end{equation}
and one can differentiate these asymptotic relations arbitrarily many times.
\label{kernelasymptotics}
\end{lem}
This lemma is proved in Subsection \ref{kernelasymptoticsproof}.

As $y\partial_y Q+2=v_1$, by \eqref{Qasymexp} we can express
\[
Q(y)=-2\log(y)-\int_{y}^{\infty}\frac{v_1}{y'}dy'.
\]
This allows us to refine the error in \eqref{Qasymexp} to, for some $c_2>0$ and $d_2\in \R/2\pi \Z$,
\[
Q=-2\log(y)+\frac{c_2\sin(\fr \log(y)+d_2)}{y^{1/2}}+O(\frac{1}{y}).
\]
Taking an exponential yields
\begin{equation}
e^Q=\frac{1}{y^2}+\frac{c_2\sin(\fr \log(y)+d_2)}{y^{5/2}}+O(\frac{1}{y^3}),
\label{Qrefinedasymptot}
\end{equation}
and we can differentiate these asymptotic relations.

\subsection{Derivation of the linearized operators}
The goal in this section will be to view \eqref{originselfsimeq} as a perturbation of the isothermal sphere equation, \eqref{stareq}. In particular, we assume $\tilde{w}(0)=-2\log(\lambda)$ and scale by defining
\[
\tilde{w}=\bar{w}(\cdot/\lambda)-2\log(\lambda),\qquad \tilde{u}=\lambda \bar{u}(\cdot/\lambda).
\]
Substituting these relations into \eqref{originselfsimeq} yields
\begin{equation}
\begin{cases}
(2+y\partial_y)e^{\bar{w}}+\div(e^{\bar{w}}\bar{u})=0\\
\Delta \bar{w}+2e^{\bar{w}}=-\lambda^2 \div((y+\bar{u})\partial_y \bar{u})\\
\bar{w}(0)=0,\qquad \bar{u}(0)=0,\qquad \bar{u}'(0)=-\frac{2}{3}
\end{cases}.
\label{scaledselfsimeq}
\end{equation}
We now consider $\lambda$ to be a small parameter, and formally do an asymptotic expansion, $\bar{w}=\sum_{n=0}^{\infty}(\lambda^2)^j w_j$, $\bar{u}=\sum_{n=0}^{\infty}(\lambda^2)^j u_j$. The zeroth order equation in $\lambda$ is
\[
\begin{cases}
(2+y\partial_y)e^{w_0}+\div(e^{w_0}u_0)=0\\
\Delta w_0+2e^{w_0}=0\\
w_0(0)=0,\qquad u_0(0)=0,\qquad u_0'(0)=-\frac{2}{3}.
\end{cases},
\]
We note that the second equation is exactly \eqref{stareq}, and so we can put $w_0=Q$. With this being known, one can write the first equation as
\begin{equation}
\begin{cases}
(2+y\partial_y)e^Q+J u_0=0\\
u_0(0)=0,\qquad u_0'(0)=-\frac{2}{3}
\end{cases}
\end{equation}
where
\[
J u:=\div(e^Q u).
\]
We recall for further use that that the ODE
\[
\begin{cases}
J u+f=0\\
u(0)=0,\qquad u'(0)=-\frac{f(0)}{3}
\end{cases}
\]
is solved by $u=T(f)$, where
\begin{equation}
T(f):=-\frac{1}{y^2 e^Q}\int_0^y (y')^2 f dy'.
\label{Tdef}
\end{equation}
In particular, in this case the solution is
\begin{equation}
u_0=u_*:=T((2+y\partial_y)e^Q).
\label{ustardef}
\end{equation}
By the asymptotics of $e^Q$, we see that as $y\to \infty$ $u_*$ has asymptotics
\begin{equation}
u_*=c_2\sin(\fr \log(y)+d_2+\theta_0)y^{1/2}+O(1)
\label{ustarasymptot}
\end{equation}
and we can differentiate the asymptotics (the remainder decreasing in order by $1$ each time).

To motivate definitions for the iteration scheme, we expand \eqref{scaledselfsimeq} to order $\lambda^2$ obtaining the following inhomogeneous linear ode for $w_1,u_1$:
\begin{equation}
\begin{cases}
Lw_1+\div(e^Q u_1)=0\\
-Hw_1+\div((y+u_*)\partial_y u_*)=0\\
w_1(0)=u_1(0)=u_1'(0)=0
\end{cases}.
\label{nearlinearization}
\end{equation}
where
\[
Lw:=(2+y\partial_y)(e^{Q}w)+\div(e^Q u_* w).
\]
We observe that this system admits a hierarchical structure. To solve the second equation for $w_1$, we may perform variation of constants using the homogeneous solutions $v_1,v_2$ studied above. In particular, we define a solution operator satisfying $H\circ S=\id$ by
\begin{equation}
S(f)(y):=\left(\int_0^y f(y') v_2(y')\cdot (y')^2dy' \right)v_1(y)-\left(\int_0^y f(y') v_1(y')\cdot (y')^2 dy'\right)v_2(y).
\label{Sdef}
\end{equation}
For $f$ smooth up to the origin the formula gives the unique solution to $Hw=f$ which is smooth at the origin and satisfies $w(0)=0$. Once $w_1$ has been found, we can the solve the first equation for $u_1$ with the solution operator $T$. 

When we use the heuristic that applying $S$ multiplies asymptotic expansions by $y^2$, we see that as $y\to\infty$,
\begin{equation}
w_1=O(y^{3/2}).
\label{heuristic1}
\end{equation}
If we then solve for $u_1$, using the heuristic that $T$ multiplies asymptotic expansions by $y^3$, we see that as $y\to \infty$,
\begin{equation}
u_1=O(y^{5/2}).
\label{heuristic2}
\end{equation}
In the next section we study the full difference equation. These last two heuristic growth rates help motivate the weights for the function spaces.

\subsection{Difference equations and iteration scheme}
We now discuss the full construction. We begin by establishing the notation
\begin{equation}
Q_{\lambda}=Q(\cdot/\lambda)-2\log(\lambda),\qquad u_{\lambda}=\lambda u_*(\cdot/\lambda),\qquad v_{i,\lambda}=v_i(\cdot/\lambda).
\label{scaledvarsdef}
\end{equation}
We consider a fixed $\lambda_0\ll y_0$, and show existence of solutions with $\tilde{\rho}(0)=\frac{1}{\lambda^2}$ for $\lambda$ close to $\lambda_0$. To do this, we introduce variables $w,u$ defined implicitly by
\[
\tilde{w}=Q_{\lambda}+\lambda_0^2 w(\cdot/\lambda_0),\qquad \tilde{u}=u_{\lambda}+\lambda_0^3 u(\cdot/\lambda_0).
\]
We adopt the notations
\[
y_1=y_0/\lambda_0,\qquad r=\lambda/\lambda_0
\]
and note the scaling identities
\[
Q_{\lambda}=Q_{r}(\cdot/\lambda_0)-2\log(\lambda_0),\qquad u_{\lambda}=\lambda_0 u_{r}(\cdot/\lambda_0).
\]
We then compute that \eqref{originselfsimeq}, with $w(0)=-2\log r\lambda_0$, on $[0,y_0]$ is equivalent to the following ODE system on $[0,y_1]$,
\begin{equation}
\begin{cases}
(2+y\p_y)(e^{Q_{r}+\lambda_0^2 w})+\div(e^{Q_{r}+\lambda_0^2 w}(u_{r}+\lambda_0^2 u))=0\\
\lambda_0^2\div((y+u_{r}+\lambda_0^2 u)\p_y(u_{r}+\lambda_0^2 u))+\Delta(Q_{r}+\lambda_0^2 w)+2e^{Q_{r}+\lambda_0^2 w}=0\\
w(0)=u(0)=u'(0)=0
\end{cases}.
\label{scaledoriginselfsimeq}
\end{equation}
We next define scaled versions of operators defined in the previous subsection.
\begin{align*}
H_{r}w:=&-r^2(\Delta+2e^{Q_r})w,\\
S_rf:=&\frac{1}{r^3}\left(\left(\int_0^y f v_{2,r}(y')^2 dy'\right)v_{1,r}-\left(\int_0^y f v_{1,r}(y')^2 dy'\right)v_{2,r}\right),\\
J_r u:=&r^2\div(e^{Q_r}u),\\
T_rf :=&-\frac{1}{r^2 y^2e^{Q_r}}\int_0^y f (y')^2 dy'.\\
\end{align*}
We note that $H_rS_r=\id$, $-J_r T_r=\id$, $J u_r+r^2 (2+y\p_y)e^{Q_r}=0$. We also define
\[
L_rw :=r^2\left((2+y\p_y)(e^{Q_r} w)+\div(u_r e^{Q_r} w)\right),\qquad K_r:=T_r\circ L_r\circ S_r.
\]
We then use Taylor's theorem, in particular
\begin{align*}
e^{Q_{r}+\lambda_0^2 w}&=e^{Q_{r}}+\lambda_0^2 w\int_0^1 e^{Q_{r}+t\lambda_0^2 w}dt\\
&=e^{Q_{r}}+\lambda_0^2 we^{Q_{r}}+\lambda_0^4w^2\int_0^1(1-t)e^{Q_{r}+t\lambda_0^2 w}dt,
\end{align*}
to write \eqref{scaledoriginselfsimeq} as the difference equations
\begin{equation}
\begin{cases}
0=L_r w+J_r u+F_r(w,u)\\
0=-H_rw+G_r(w,u)\\
w(0)=u(0)=u'(0)=0
\end{cases},
\label{diffeq}
\end{equation}
where we define nonlinear functions
\begin{align*}
F_r(w,u):=&\lambda_0^2r^2\left((2+y\partial_y)(w^2 \int_0^1(1-t)e^{Q_r+t\lambda_0^2 w} dt)\right.\\
&\left.+\div(u w\int_0^1 e^{Q_r+t\lambda_0^2 w}dt)+\div(u_rw^2\int_0^1(1-t)e^{Q_r+t\lambda_0^2 w}dt)\right) ,\\
G_r(w,u):=&r^2 \div((y+u_r)\partial_y u_r)+\lambda_0^2r^2\left(2w^2\int_0^1(1-t)e^{Q_r+t\lambda_0^2 w}dt\right.\\
&\left.+\div(u\partial_y u_r+(y+u_r)\partial_y u+\lambda_0^2 u\partial_y u)\right).
\end{align*}
We then note that \eqref{diffeq} can be written in integral form as
\begin{equation}
\begin{cases}
u=T_r(F_r(w,u))+K_r(G_r(w,u))\\
w=S_r(G_r(w,u))
\end{cases}.
\label{integralform}
\end{equation}
As done in the exterior region, we solve this difference equation by showing a contraction in a suitable space. In particular, we adopt the following definitions.
\begin{defin}
We define the solution Banach spaces $Z_{y_1},Y_{y_1}$ and nonlinearity Banach spaces $\calG_{y_1},\calF_{y_1}$ by
\begin{align*}
Z_{y_1}&:=\{f\in C^2([0,y_1])\colon \norm{f}_{Z_{y_1}}<\infty,\quad f(0)=0\},\\
Y_{y_1}&:=\{f\in C^2([0,y_1])\colon \norm{f}_{Y_{y_1}}<\infty,\quad f(0)=f'(0)=0\},\\
\calG_{y_1}&=\{f\in C^0([0,y_1])\colon \norm{f}_{\calG_{y_1}}<\infty\},\\
\calF_{y_1}&=\{f\in C^1([0,y_1])\colon \norm{f}_{\calF_{y_1}}<\infty,\quad f(0)=0\},
\end{align*}
where
\begin{align*}
\norm{f}_{Z_{y_1}}&:=\sup_{y\in [0,y_1]}\langle y\rangle^{-3/2}(\abs{f}+\ev{y}\abs{\partial_r f}+\ev{y}^2 \abs{\partial_r^2f}),\\
\norm{f}_{Y_{y_1}}&:=\sup_{y\in [0,y_1]}\langle y\rangle^{-5/2}(\abs{f}+\ev{y}\abs{\partial_r f}+\ev{y}^2 \abs{\partial_r^2f}),\\
\norm{f}_{\mathcal{G}_{y_1}}&:=\sup_{y\in [0,y_1]}\langle y\rangle^{1/2}\abs{f},\\
\norm{f}_{\mathcal{F}_{y_1}}&:=\sup_{y\in [0,y_1]}\langle y\rangle^{1/2}(\abs{f}+\ev{y}\abs{\partial_y f}).
\end{align*}
\end{defin}
We then solve for $w\in Z_{y_1},u\in Y_{y_1}$ by putting $G,F$ in $\mathcal{G}_{y_1}$ and $\mathcal{F}_{y_1}$ respectively. As a note, the weights are consistent with the heuristics \eqref{heuristic1}, \eqref{heuristic2}. In addition, the number of derivatives in each norm is motivated by the heuristic that inverting the divergence operator gains 1 derivative, and inverting $H$ gains 2 derivatives.

As in the corresponding in the exterior region, we summarize our construction by means of three lemmas.
\begin{lem} If $y_0\ll 1$ and $\lambda_0\leq \frac{1}{10}y_0$ is fixed, then for $r\in(\frac{1}{2},2)$ if $G,F\in \calG_{y_1},\calF_{y_1}$ respectively we have that $T_r(F), K_r(G)\in Y_{y_1}$ and $S_r(G)\in Z_{y_1}$. Moreover, we have
\[
\norm{T_r}_{L(\calF_{y_1},Y_{y_1})},\norm{S_r}_{L(\calG_{y_1},Z_{y_1})},\norm{K_r}_{L(\calG_{y_1},Y_{y_1})}\leq C_1
\]
In addition, we have that $r\mapsto (T_r,S_r,K_r)$ is $C^1$ with domain $(\frac{1}{2},2)$. For all $r_0\in (\frac{1}{2},2)$,\[
\norm{\pdv{T_r}{r}|_{r_0}}_{L(\calF_{y_1},Y_{y_1})},\norm{\pdv{S_r}{r}|_{r_0}}_{L(\calG_{y_1},Z_{y_1})},\norm{\pdv{K_r}{r}|_{r_0}}_{L(\calG_{y_1},Y_{y_1})}\leq C_1.
\]
In addition, if $r_1$ is also in $(\frac{1}{2},2)$, we have the bounds
\begin{equation}
\norm{\pdv{S_r}{r}|_{r_1}-\pdv{S_r}{r}|_{r_0}}_{L(\calG_{y_1},Z_{y_1})}\leq C_1\abs{r_1-r_0},
\label{linearderivativelipschitz}
\end{equation}
and analogous difference bounds for $T_r,K_r$.
\label{linlem}
\end{lem}
For the proof, see Subsection \ref{linlemproof}.
\begin{lem} For all $C>0$ there exists $C_2(C)>0$ (independent of $y_0,\lambda,r$) such that for $y_0\ll 1$ and fixed $\lambda_0\leq \frac{1}{10}y_0$, $r\in (\frac{1}{2},2)$,
\begin{enumerate}[wide,label=(\alph*)]
\item
If $(w,u)\in Z_{y_1}\times Y_{y_1}$ and $\norm{u}_{Y_{y_1}}+\norm{w}_{Z_{y_1}}\leq C$ then
\[
(G_{r}(w,u),F_r(w,u))\in \calG_{y_1}\times \calF_{y_1}\qquad \text{and}\qquad \norm{F_{r}(w,u)}_{\mathcal{F}_{y_1}}+\norm{G_{r}(w,u)}_{\mathcal{G}_{y_1}}\leq C_2.
\] 
\item
We have that $(r,w,u)\mapsto (G_{r}(w,u),F_{r}(w,u))$ is $C^1$ as a function of $r,w,u$ with domain
\[
(\frac{1}{2},2)\times \{(w,u)\in Z_{y_1}\times Y_{y_1}\colon \norm{w}_{Z_{y_1}}+\norm{u}_{Y_{y_1}}< C\}
\]
valued in $\calG_{y_1}\times \calF_{y_1}$ with the uniform bounds
\begin{equation}
\norm{\pdv{G_r}{r}}_{\calG_{y_1}}+\norm{\pdv{F_r}{r}}_{\calF_{y_1}}\leq C_2,
\label{derivativebounda}
\end{equation}
\begin{equation}
\norm{\pdv{G_r}{w}}_{L(Z_{y_1},\calG_{y_1})}+\qquad \norm{\pdv{G_r}{u}}_{L(Y_{y_1},\calG_{y_1})}+\norm{\pdv{F_r}{w}}_{(Z_{y_1},\calF_{y_1})}+\norm{\pdv{F_r}{u}}_{L(Y_{y_1},\calF_{y_1})}\leq C_2 y_0^2
\label{derivativeboundb}
\end{equation}
for all $r_0,w_0,u_0$ in this domain. In addition, if $(r_1,w_1,u_1)$ is also in this domain, we have the following Lipschitz control of the derivative
\begin{equation}
\norm{\pdv{G}{r}|_{r_1,p_1,w_1}-\pdv{G}{r}|_{r_0,p_0,w_0}}\leq C_1(\abs{r_1-r_0}+\norm{w_1-w_0}_{Z_{y_1}}+\norm{u_1-u_0}_{Y_{y_1}}),
\label{Grlipschitz}
\end{equation}
\begin{multline}
\norm{\pdv{G}{w}|_{r_1,p_1,w_1}-\pdv{G}{w}|_{r_0,p_0,w_0}}_{L(Z_{y_1,\calG_{y_1}})}+\norm{\pdv{G}{u}|_{r_1,p_1,w_1}-\pdv{G}{u}|_{r_0,p_0,w_0}}_{L(Y_{y_1,\calG_{y_1}})}\\
\leq y_0^2C_1(\abs{r_1-r_0}+\norm{w_1-w_0}_{Z_{y_1}}+\norm{u_1-u_0}_{Y_{y_1}})
\label{Gfunctionlipschitz}
\end{multline}
with analogs of \eqref{derivativebounda}, \eqref{derivativeboundb}, \eqref{Grlipschitz} and \eqref{Gfunctionlipschitz} holding for $F$ as well.
\end{enumerate}
\label{nonlinlemma}
\end{lem}
For proof see Subsection \ref{nonlinlemmaproof}.
\begin{lem}
\begin{enumerate}[wide,label=(\alph*)]
\item There exists a universal constant $C_3$ such that for fixed $y_0\ll 1$, $\lambda_0 < \frac{1}{10}y_0$, $r\in (\frac{1}{2},2)$, there exists a $C^2$ solution to \eqref{originselfsimeq} of the form
\[
\tilde{w}=Q_{r\lambda_0}+\lambda_0^2w[\lambda_0;r](\cdot/\lambda_0),\qquad \tilde{u}=u_{r\lambda_0}+\lambda_0^3 u[\lambda_0;r](\cdot/\lambda_0)
\]
with $(w,u)\in Z_{y_1}\times Y_{y_1}$, and $\norm{w}_{Z_{y_1}}+\norm{u}_{Y_{y_1}}\leq \frac{C_3}{100}$ (inequality independent of $y_0,\lambda_0,r$) and it is the unique solution such that $\norm{u}_{Y_{y_1}}+\norm{w}_{Z_{y_1}}\leq C_3$. 
\item
Moreover, we have that the map $r\mapsto (u[\lambda_0;r],w[\lambda_0;r])$ is $C^1$ with respect to $r\in (\frac{1}{2},2)$ valued in $Z_{y_1}\times Y_{y_1}$.
\end{enumerate}
\label{intconstructionlemmapre}
\end{lem}
For proof see Subsection \ref{intconstructionlemmapreproof}.

In preparation for the matching, we convert this statement to the original variables. As a note, a consequence of the uniqueness we have that
\begin{equation}
w[r \lambda_0;1]=\frac{1}{r^2}w[\lambda_0;r](r \cdot),\qquad u[r\lambda_0;1]=\frac{1}{r^3}u[\lambda_0;r](r\cdot).
\label{scalingidentity}
\end{equation}
In particular, we define for a fixed $\lambda_0,r$,
\begin{align*}
\rho_{int}[\lambda_0;r]:=&e^{Q_{r\lambda_0}+\lambda_0^2w[\lambda_0;r](\cdot/\lambda_0)}-e^{Q_{r\lambda_0}}\\
=&\lambda_0^2w[r;\lambda_0](\cdot/\lambda)e^{Q_{\lambda}}\int_0^1 e^{t\lambda_0^2 w[r;\lambda_0](\cdot/\lambda)}dt,\\
u_{int}[\lambda_0;r]:=&\frac{\lambda_0^3}{\lambda^3}u[\lambda_0;r](\cdot/\lambda_0).
\end{align*}
By \eqref{scalingidentity}, the righthand side depends only on $\lambda:=r\lambda_0$. In addition, we note that the $C^1$-dependence $w[\lambda_0;r]$ on $r$ implies (by keeping $\lambda_0$ fixed), $C^1$ dependence of $(\rho_{int}[\lambda](y_0), u_{int}[\lambda](y_0))$ on $\lambda$. This change of variables leads to the following corollary of Lemma \ref{intconstructionlemmapre}.

\begin{lem} If $y_0\ll 1$ and $\lambda<\frac{1}{10}y_0$, there exists $C^2$ functions $\tilde{\rho},\tilde{u}$ defined on $[0,y_0]$ with $\tilde{\rho}>0$ which solve \eqref{selfsimeq2} of the form
\[
\begin{cases}
\tilde{\rho}_{int}[\lambda](y)=e^{Q_{\lambda}}+\rho_{int}[\lambda](y)\\
\tilde{u}_{int}=u_{\lambda}+\lambda^3 u_{int}[\lambda](y)
\end{cases}
\]
where $Q_{\lambda},u_{\lambda}$ are defined by \eqref{scaledvarsdef}, $\rho_{int},u_{int}$ satisfy 
\[
\rho_{int}(0)=u_{int}(0)=u_{int}'(0)=0
\]
and they obey the following bounds on $[0,y_0]$ (uniform w.r.t. $y_0$, $\lambda$)
\begin{equation}
\abs{\rho}\lesssim \frac{y}{\lambda}\ev{\frac{y}{\lambda}}^{-3/2},\qquad \abs{u}\lesssim \frac{y^2}{\lambda^2} \ev{\frac{y}{\lambda}}^{1/2},
\label{interror}
\end{equation}
\begin{equation}
\abs{\rho'}\lesssim \frac{1}{\lambda}\ev{\frac{y}{\lambda}}^{-3/2},\qquad \abs{u'}\lesssim \frac{y}{\lambda^2}\ev{\frac{y}{\lambda}}.
\label{interrorb}
\end{equation}
Moreover,
\[
\lambda\mapsto (\rho_{int}(y_0),u_{int}(y_0))
\]
is $C^1$ with domain $(0,\frac{1}{10}y_0)$. We also have the following velocity bound on $[0,y_0]$:
\begin{equation}
\abs{\tilde{u}_{int}+y}\leq \frac{1}{2}.
\label{inthuntercontrol}
\end{equation}
\label{intconstructionlemma}
\end{lem}
\begin{proof}
We start with Lemma \ref{intconstructionlemmapre}. Converting variables, and representing
\[
\rho_{int}[\lambda]=\lambda^2w[\lambda;1](\cdot/\lambda)e^{Q_{\lambda}}\int_0^1 e^{t\lambda^2 w[\lambda;1](\cdot/\lambda)}dt,\qquad u_{int}[\lambda]=u[\lambda;1](\cdot/\lambda),
\]
yields the lemma up to and including \eqref{interrorb}. To show \eqref{inthuntercontrol}, we note that $\abs{u_{\lambda}}\lesssim \lambda\frac{y}{\lambda}\ev{\frac{y}{\lambda}}^{-1/2}$, so along with \eqref{interror}, and the fact that $z\ev{z}^{-1/2}$, $z^2\ev{z}^{1/2}$ are increasing,
\[
\abs{u_{\lambda}}\lesssim \lambda \ev{\frac{y_0}{\lambda}}^{-1/2}\lesssim y_0,\qquad \abs{\lambda^3 u_{int}}\lesssim \lambda^3 \ev{\frac{y_0}{\lambda}}^{5/2}\lesssim y_0,
\]
so by choosing $y_0$ smaller if necessary we see $\abs{y+\tilde{u}}\lesssim y_0\leq \frac{1}{2}$.
\end{proof}

\section{Proof of the theorem}
\label{matchingsection}
\subsection{Existence of $C^1$ Hunter solutions}
We now prove Theorem \ref{mainthm}, except for showing that the solution is real analytic, which is covered by an independent argument in the next subsection.
\begin{proof}
We follow the outline presented in Subsection \ref{outlinesec} and follow the proof of Proposition 2.5 of \cite{crs}. The first 8 steps correspond to the ones in their proof, and in the last step we count the sonic points.
\begin{enumerate}[wide,label=\arabic*)]
\item
We start by picking $y_0\ll 1$. We then wiggle $y_0$ if necessary so that \eqref{wigglecondition} holds. Then, for every $\epsilon\leq \epsilon_0 \ll y_0$ we apply Lemma \ref{extconstructionlemma} to obtain functions $\tilde{\rho}_{ext}[\epsilon],\tilde{u}_{ext}[\epsilon]$ defined on $[y_0,\infty)$. Then, for each $\lambda\leq \frac{1}{10}y_0$, we apply Lemma \ref{intconstructionlemma} to obtain functions $\tilde{\rho}_{int}[\lambda],\tilde{u}_{int}[\lambda]$ on $[0,y_0)$.
\item
We next define the density matching function by 
\begin{align*}
\calF[y_0](\epsilon,\lambda)&=\tilde{\rho}_{ext}[\epsilon](y_0)-\tilde{\rho}_{int}[\lambda](y_0)\\
&=\left(\frac{1}{y^2}+\frac{\epsilon}{y^2}p_{hom}(y)-e^{Q_{\lambda}}\right)+\left(\epsilon \rho_{ext}[\epsilon](y_0)-\tilde{\rho}_{int}[\lambda]\right)|_{y_0}.
\end{align*}
It is clear from the asymptotic expansions given by \eqref{0asymptot}, \eqref{Qrefinedasymptot}, \eqref{exterror}, \eqref{interror},
\[
\calF[y_0](0,0)=0.
\]
Similarly, by \eqref{epsiloncontrol} we have
\[
\partial_{\epsilon}\calF[y_0](0,0)=\frac{p_{hom}(y_0)}{y_0^2}=\frac{c_1}{y_0^{5/2}}+O(\frac{1}{y_0^{\frac{1}{2}}})\neq 0.
\]
We note that $\calF[y_0](0,\lambda)=O(y_0^{-5/2}\lambda^{1/2})$, is not $C^1$ with respect to $\lambda$, so in order to apply the implicit function theorem we change variables by picking a positive $\delta<\frac{1}{2}$, and defining $\mu=\lambda^{\delta}$, $\tilde{\calF}(\epsilon,\mu)=\calF(\epsilon,\mu^{\frac{1}{\delta}})$, so
\[
\calF(0,\lambda)=\calF(0,\mu^{\frac{1}{1/2-\delta}}):=\tilde{\calF}(0,\mu)\sim \mu^{\frac{1}{2\delta}}.
\]
Thus, by the $C^1$ dependence of the solutions on the parameters in Lemmas \ref{extconstructionlemma} and \ref{intconstructionlemma}, we have that $\tilde{\calF}(\epsilon,\mu)$ will be $C^1$ on some set of the form $(-\epsilon_0,\epsilon_0)\times \mu_0)$. Thus, by the implicit function theorem, assuming $\mu \leq \mu_1$ for some constant $\mu_0$ (depending on $y_0$), we have that $\tilde{\calF}[y_0](\epsilon,\mu)=0$ determines $\epsilon$ as a $C^1$ function of $\mu$. Converting back to $\lambda$, we have a function $\epsilon(\lambda)$ (of regularity $C^{\frac{1}{2}-}$) defined on $[0,\lambda_1)$ so that
 \begin{equation}
\epsilon(\lambda)=O(\lambda^{1/2})
\label{epsilonapriori}
 \end{equation}
and
 \[
\tilde{\rho}_{ext}[\epsilon(\lambda)](y_0)=\tilde{\rho}_{int}[\lambda](y_0).
 \]
 \item
We note that by definition, we have that
\[
\epsilon(\lambda)\left(\frac{p_{hom}(y_0)}{y_0^2}+\rho_{ext}(y_0)\right)=(e^{Q_{\lambda}}-\frac{1}{y^2})(y_0)+\rho_{int}(y_0).
\]
We let $q(y)=e^Q-\frac{1}{y^2}-\frac{c_2\sin(\fr \log(y)+d_3)}{y^{5/2}}$ (it is of size $O(\frac{1}{y^3})$ by \eqref{Qrefinedasymptot}) and have
\[
\epsilon(\lambda)(\frac{p_{hom}(y_0)}{y_0^2}+\rho_{ext}(y_0))=\lambda^{1/2}\frac{c_2 \sin(\fr \log(y_0/\lambda)+d_3)}{y_0^{5/2}}+\frac{1}{\lambda^2}q(\frac{y_0}{\lambda})+\rho_{int}(y_0).
\]
If we substitute in asymptotics, we get that if $\lambda\ll y_0$,
\begin{multline*}
\epsilon(\lambda)\frac{c_1}{y_0^{5/2}}\left(1+O(y_0^2)+O(\epsilon y_0^{-1/2})\right)\\
=\frac{c_2\lambda^{1/2}}{y_0^{5/2}}\left(\sin(\fr \log(y_0/\lambda)+d_2)+O(\frac{\lambda^{1/2}}{y_0^{1/2}})+O(y_0^2)\right).
\end{multline*}
We now without loss of generality assume $\lambda^{1/2}\lesssim y_0^{5/2}$, and hence by \eqref{epsilonapriori}, $\epsilon\lesssim y_0^{5/2}$. This allows us to simplify the errors to
\[
\epsilon(\lambda)\frac{c_1}{y_0^{5/2}}(1+O(y_0^2))=\frac{c_2\lambda^{1/2}}{y_0^{5/2}}\left(\sin(\fr \log(y_0/\lambda)+d_2)+O(y_0^2)\right).
\]
Isolating $\epsilon$ gives that as $\lambda\to 0$
\[
\epsilon(\lambda)=\frac{c_2}{c_1}\lambda^{1/2}\left(\sin(\fr \log(y_0/\lambda)+d_2)+O(y_0^2)\right),
\]
which can be rewritten via \eqref{wigglecondition} as
\begin{equation}
\epsilon(\lambda)=\frac{c_2}{c_1}\lambda^{1/2}\left(\cos(\fr \log(\lambda)+d_1-d_2)+O(y_0^2)\right).
\label{epsilonasymptot}
\end{equation}
\item
We then define the second matching function as
\[
\calG[y_0](\lambda)=\frac{1}{c_2 y_0^{1/2}\lambda^{1/2}\sin(\theta_0)}\left(\tilde{u}_{ext}[\epsilon(\lambda)](y_0)-\tilde{u}_{int}[\lambda](y_0)\right).
\]
By the $C^1$ dependence of the solutions on their parameters, and the regularity from the implicit function theorem, we see that $\calG[y_0](\lambda)$ will be $C^1$ on $(0,y_1)$. We can then decompose $\calG[y_0](\lambda)$ as
\[
\calG[y_0](\lambda)=\frac{1}{\sin(\theta)}\left(\calG_1[y_0](\lambda)+\calG_2[y_0](\lambda)\right),
\]
where
\begin{align*}
\calG_1[y_0](\lambda)&:=\frac{y_0^{1/2}\epsilon(\lambda)}{c_2 \lambda^{1/2}}\w_{hom}(y_0)-\frac{1}{c_2y_0^{1/2}\lambda^{1/2}} u_\lambda(y_0),\\
\calG_2[y_0](\lambda)&:=\frac{\epsilon(\lambda)}{c_2 y_0^{1/2}\lambda^{1/2}}u_{ext}(y_0)-\frac{\lambda^{5/2}}{c_2 y_0^{1/2}}u_{int}(y_0).
\end{align*}
We can compute these using \eqref{wigglecondition}, \eqref{0asymptot}, \eqref{exterror}, \eqref{ustarasymptot},\eqref{interror}, \eqref{epsilonasymptot}
\begin{align*}
\calG_1(y_0)&=\cos(\theta_0)\cos(\fr\log(\lambda)+d_1-d_2)-\cos(\fr\log(\lambda)+d_1-d_2+\theta_0)+O(y_0^2)\\
&=\sin(\theta_0)\sin(\fr \log(\lambda)+d_1-d_2)+O(y_0^2),\\
\calG_2(y_0)&=O(y_0^{2}).\\
\end{align*}
Combining gives
\begin{equation}
\mathcal{G}[y_0](\lambda)=\sin(\fr \log(\lambda)+d_1-d_2)+O(y_0^2).
\label{matchingcondition}
\end{equation}
\item
As mentioned in the outline, we now consider
\[
\lambda_{k,\pm}=\exp\left(\frac{2}{\sqrt{7}}(-k\pi-d_1+d_2\pm \frac{1}{10})\right).
\]
Assuming that $k$ is sufficiently large so that $\lambda_{k,\pm}\leq \lambda_0$, $\lambda_{k,\pm}\lesssim y_0^{5/2}$, we have that
\[
G[\lambda](y_0)=(-1)^k\sin(\pm \frac{1}{10})+O(y_0^2).
\]
Thus, as $y_0\ll 1$, by continuity, for each $k\geq N$ for some $N$, we have that there exists $\lambda\in (\lambda_{k,-},\lambda_{k,+})$ such that $G[\lambda_k](y_0)=0$. These matching conditions ensure that we can glue the exterior and interior solutions together to get a continuous functions which solves the equation on $[0,y_0]$ and $[y_0,\infty]$. Moreover, as a consequence of Lemma \ref{extconstructionlemma} and Lemma \ref{intconstructionlemma} we have $\abs{\tilde{u}+y}\geq \frac{1}{2}$ on $[y_0,\infty)$ and $\abs{\tilde{u}+y}\leq \frac{1}{2}$ on $[0,y_0]$. In particular, $y_0$ is not a sonic point, so by nondegeneracy of the equation, the matched solutions are classical solutions on $[0,\infty)$. It remains to count the zeroes of $y^2\tilde{\rho}-1$.
\item
We first claim that for all $\epsilon\ll y_0^{1/2}$ (so that Lemma \ref{extconstructionlemma} applies) $y^2\tilde{\rho}_{ext}[\epsilon](y)-1$ and $\frac{p_{hom}}{y^2}$ have the same number of zeroes on $[y_0,\infty)$. Indeed, to see this, for convenience, we introduce new variables by
\[
\psi(t)=e^{t/2}p_{hom}(e^t),\qquad \tilde{\psi}(t)=e^{5t/2}\tilde{\rho}_{ext}(e^t).
\]
so the problem becomes that of equating the number of zeroes of $\psi$ and $\psi+\tilde{\psi}$ on $[\log(y_0),\infty)$. We then note that by \eqref{infinityasymptot} we have that as $t\to \infty$
\[
\psi=\mu_3 e^{t/2}+O(e^{-3t/2}).
\]
and hence has no roots. Similarly, by \eqref{0asymptot} we have that as $t\to -\infty$,
\[
\psi(t)=c_1\sin(\fr t+d_1)+O(e^{2t}),
\]
in particular, by \eqref{wigglecondition}
\[
\psi(\log(y_0))=c_1+O(y_0^2)
\]
so given that $y_0\ll 1$ this is bounded away from zero. Thus, we see that $\psi$ will only have discretely many roots, and they will all be in $[\log(y_0),Y_0]$ for some large $Y_0$ independent of $\epsilon,y_0$. In addition, there exists a $\delta>0$ so that if $\{t_i\}$ are the roots of $\psi$, the $\abs{\psi'}>\delta$ on $\cup_i [t_i-\frac{1}{100},t_i+\frac{1}{100}]$ and $\abs{\psi}>\delta$ on $[\log(y_0),\infty)\setminus (\cup_i [t_i-\frac{1}{100},t_i+\frac{1}{100}])$. Moreover, we have that by \eqref{exterror} and \eqref{exterrorb}
\[
\abs*{\tilde{\psi}}, \abs*{\tilde{\psi'}}\lesssim \epsilon y_0^{-1/2}\ev{e^{t/2}}.
\]
Thus, as we may assume $\epsilon\ll y_0^{1/2}e^{-Y_0/2}$, $\abs{\tilde{\psi}},\abs{\tilde{\psi'}}\leq \frac{\delta}{100}$, which, along with the boundary conditions at $\log(y_0)$ and $\infty$, imply that $\psi$ and $\psi+\tilde{\psi}$ have the same number of roots.
\item
We next claim that if $\lambda_k$ is as in step 5, then $y^2\tilde{\rho}_{int}[\lambda_k]-1$ and $y^2e^{Q_{\lambda}}-1$ intersect $0$ the same number of times on $[0,y_0]$. To show this, we again, for convenience make a change of variables. We define
\begin{align*}
\chi(t):=&\ev{e^{t}}^{\frac{5}{2}}(\lambda_k^2 e^{Q_{\lambda_k}}(\lambda_k e^t)-e^{-2t})=\\
=&\ev{e^{t}}^{\frac{5}{2}}(e^{Q}(e^t)-e^{-2t})
\end{align*}
\[
\tilde{\chi}(t)=\lambda_k^2\ev{e^{t}}^{\frac{5}{2}}\rho_{int}(\lambda_k e^t).
\]
and it suffices to show that $\chi(t)$ and $\chi(t)+\tilde{\chi}(t)$ have the same number of zeroes on $(-\infty,\log(y_0/\lambda_k)]$. We then see that as $t\to -\infty$,
\[
\chi(t)=-e^{-2t}+O(1)
\]
and is thus not vanishing in this regime. Similarly, as $t\to \infty$, we have that, by \eqref{Qrefinedasymptot},
\begin{equation}
\chi(t)=c_2\sin(\fr t+d_2)+O(e^{-t/2}).
\end{equation}
If we use that $\lambda_k^{1/2}\lesssim y_0^{2}$ and change independent variables from $t$ to $x$ defined by
\[
t=\log(y_0)-\frac{2}{\sqrt{7}}(x-d_1+d_2)
\]
we have that by \eqref{wigglecondition}
\[
\chi(t)=c_2\cos(x)+O(y_0^2)
\]
and $t_{k,\pm}=\log(y_0/\lambda_{k,\pm})$ corresponds to $x_{k,\pm}=-k \pi \pm \frac{1}{10}$. Thus, as $\abs{\cos(x)}$ is uniformly bounded away from zero on $[-k\pi-\frac{1}{5},-k\pi+\frac{1}{5}]$ we can assume $\abs{\chi(t)}$ is bounded away from zero on a neighborhood of $\log(y_0/\lambda_k)$. Similarly, near the roots of $\chi(t)$, the derivative of $\chi(t)$ is bounded away from zero. By Lemma \ref{intconstructionlemma}, we have that on $(-\infty,\log(y_0/\lambda))$,
\[
\abs{\tilde{\chi}(t)}, \abs{\tilde{\chi}'(t)}\lesssim \lambda^2 e^t\ev{e^t}\lesssim y_0^2.
\]
Thus, as with before, we may conclude that $\chi$ and $\tilde{\chi}$ have the same number of zeroes.
\item
By the previous two steps, the number of zeroes of $y^2\tilde{\rho}-1$ will be equal to the number of zeroes of $\frac{p_1}{y^2}$ on $[y_0,\infty)$, a fixed number with respect to $k$, plus the number of zeroes of $y^2e^{Q_{\lambda_k}}-1$ on $[0,y_0]$, which changes with $k$. Thus, to show that the number of zeroes of $y^2\tilde{\rho}-1$ increases by $1$ with each increment of $k$, it suffices to show that $y^2e^Q-1$ has exactly one zero on $[y_0/\lambda_k],[y_0/\lambda_{k+1}]$ for $k$ sufficiently large. Doing the same change of variables as in the previous step, it suffices to show that $\chi(t)$ has exactly one root in $[\log(y_0/\lambda_{k}),\log(y_0/\lambda_{k+1})]$. To do this, we define
\[
\bar{\chi}(t)=\chi(t)-c_2\sin(\fr t+d_2).
\]
By \eqref{Qrefinedasymptot}, we have that as $t\to \infty$,
\[
\bar{\chi}(t)=O(e^{-t/2})=O(y_0^2),\qquad \bar{\chi}'(t)=O(e^{-t/2})=O(y_0^2).
\]
so $\chi(t)$ has the same number of roots as $c_2\sin(\fr t+d_2)$ on this interval. By changing variables from $t$ to $x$, we see that $c_2\sin(\fr t+d_2)$ will have exactly one zero on $[\log(y_0/\lambda_{k}),\log(y_0/\lambda_k)]$ because $\cos(x)$ has at least one zero in $[\frac{1}{10},\pi-\frac{1}{10}]$ and at most one zero in $[-\frac{1}{10},\pi+\frac{1}{10}]$. 

Thus, for $k$ sufficiently large we have that $y^2\tilde{\rho}_k-1$ has a finite number of zeroes and $y^2\tilde{\rho}_{k+1}-1$ has exactly one more zero then $y^2\tilde{\rho}_k-1$. Thus, by shifting the labels if necessary we get that $y^2\tilde{\rho}_N-1$ has exactly $N+1$ zeroes.
\item
All that remains is to count the sonic points. As $\abs{\tilde{u}+y}\leq \frac{1}{2}$, we have that there are no sonic points on $[0,y_0]$. As $\tilde{u}+y$ is increasing past $y_0$, we conclude that on $[0,\infty)$ there is a unique point where the sonic point condition $\tilde{u}+y=\pm 1$ is satisfied. From the construction in the exterior region, we know that $y_*=1+\epsilon(\lambda_k)$ is a sonic point where the Hunter conditions are satisfied.
\end{enumerate}
\end{proof}

\subsection{Proof that $C^1$ Hunter solutions are real analytic.}
\label{regularitysubsection}
It remains to show that the $C^1$ Hunter solution is analytic. As a note, this subsection is independent from the rest of the paper and some notations are redefined.

Analyticity on $(0,\infty)\setminus y_*$ is easy as we can locally rewrite \eqref{selfsimeq2} as
\[
(\tilde{\rho}',\tilde{u}')=f(y,\tilde{\rho},\tilde{u})
\]
where
\[
f(y,a,b)=\left(-\frac{2ab}{(y+b)^2-1}(\frac{b+y}{y}-a),-\frac{2ab}{(b+y)^2-1}(-\frac{1}{ya}+b+y)\right)
\]
is analytic on $\C^3\setminus \{y=0\},\{y+b=\pm 1\}$. Thus, by standard local ODE theory we have uniqueness of $C^1$ solutions and existence of real analytic solutions, so any $C^1$ solution is real analytic.

We next show that if $(\tilde{\rho},\tilde{u})$ be a $C^1$ solution to the self-similar Euler-Poisson system defined on an interval containing $y_*\in (\frac{1}{2},\frac{3}{2})$ satisfying the Hunter boundary conditions
\begin{equation}
\tilde{\rho}_1(y_*)=\frac{1}{y_*},\qquad \tilde{\rho}_1'(y_*)=\frac{1}{y_*}-\frac{3}{y_*^2},\qquad \tilde{u}_1(y_*)=1-y_*,\qquad \tilde{u}_1'(y_*)=\frac{1}{y_*}-1.
\label{localhunterboundaryconds}
\end{equation}
then $(\tilde{\rho},\tilde{u})$ is analytic at $y_*$. This is a corollary of the following Lemma.

\begin{lem}
\begin{enumerate}[wide,label=(\alph*)]
\item
If one fixes $y_*\in (\frac{1}{2},\frac{3}{2})$ then there exists a real analytic solution to the self-similar Euler-Poisson system, $(\tilde{\rho}_0,\tilde{u}_0)$, defined on a nbhd of $y_*$ satisfying the boundary conditions \eqref{localhunterboundaryconds}.
\item
Let $(\tilde{\rho}_1,\tilde{u}_1)$ be a $C^1$ solution to the self-similar Euler-Poisson system defined on an interval containing $y_*$ satisfying the Hunter boundary conditions \eqref{localhunterboundaryconds} then there exists an interval containing $y_*$ s.t. $(\tilde{\rho}_1,\tilde{u}_1)=(\tilde{\rho}_0,\tilde{u}_0)$.
\end{enumerate}
\label{locallem}
\end{lem}

The proof of part (a) essentially follows from the proof of Theorem 2.10 of \cite{ghj}. The only difference is that the matrix $\mathcal{A}_N$ must be replaced by $\mathcal{A}_N^{H}$ (defined in Remark 2.5), leading to slight changes in inequalities (2.49) and (2.50). The rest of the proof carries over.

 To prove part (b), we start with the following formulation of the ODEs:
\[
\mqty(\tilde{u}+y & \tilde{\rho}\\ \frac{1}{\tilde{\rho}} & \tilde{u}+y)\mqty(\tilde{\rho}'\\ \tilde{u}')+2\tilde{\rho}(\tilde{u}+y)\mqty(\frac{1}{y}\\ 1)=0.
\]
We then linearize the equation about $(\tilde{\rho}_0,\tilde{u}_0)$ to get the operator
\[
L_0(\rho,u)=\mqty(\tilde{u}_0+y & \tilde{\rho}_0\\ \frac{1}{\tilde{\rho}_0} & \tilde{u}_0+y)\mqty(\rho\\ u)'+\mqty(u & \rho\\ -\frac{\rho}{\tilde{\rho}_0^2} & u)\mqty(\tilde{\rho}_0'\\ \tilde{u}_0')+2(\rho(\tilde{u}_0+y)+\tilde{\rho}_0u)\mqty(\frac{1}{y}\\ 1).
\]
which we may rewrite as
\[
L_0(\rho,u)=
\mqty(\tilde{u}_0+y & \tilde{\rho}\\ \frac{1}{\tilde{\rho}_0} & \tilde{u}_0+y)\mqty(\rho\\ u)'+\mqty(\tilde{u}_0'+2\frac{\tilde{u}_0+y}{y} & \tilde{\rho}_0'+\frac{2\tilde{\rho}_0}{y}\\ -\frac{\tilde{\rho}_0'}{\tilde{\rho}_0^2}+2(\tilde{u}_0+y) & \tilde{u}_0'+2\tilde{\rho}_0)\mqty(\rho\\ u).
\]
We then write the equation satisfied by their difference $(\rho,u)=(\tilde{\rho}_1-\tilde{\rho}_0,\tilde{u}_1-\tilde{u}_0)$ as
\begin{equation}
L_0(\rho, u)+N_0(\rho,u)=0.
\label{diffeq}
\end{equation}
where
\begin{equation}
N_0(\rho,u)=\mqty(u \rho' + \rho u'\\ -\frac{\rho' \rho}{\tilde{\rho}_0(\tilde{\rho}_0+\rho)}+\frac{\rho^2 \tilde{\rho}_0'}{\tilde{\rho_0}^2(\tilde{\rho}_0+\rho)} + u u')+2\rho u\mqty(\frac{1}{y}\\ 1).
\label{N0defin}
\end{equation}
We consider the solvability of
\[
L_{\tilde{\rho}_0,\tilde{u}_0}\mqty(\rho\\ u)=\mqty(R\\ U).
\]
To simplify the ensuing linear algebra we multiply on the left by $\mqty(y_* & -1\\ y_* & 1)$, and then Taylor expand the coefficients in $y-y_*$. When we do this, we get
\begin{multline*}
\mqty((y_*-2)(y-y_*)+O((y-y_*)^2) & (1-\frac{4}{y_*})(y-y_*)+O((y-y_*)^2)\\ 2y_*+O(y-y_*) & 2+O(y-y_*))\mqty(\rho'\\ u')\\
+\mqty(-2 +O(y-y_*) & 2-\frac{4}{y_*} + O(y-y_*)\\ 8-2y_* + O(y-y_*) & \frac{2}{y_*}+ O(y-y_*))\mqty(\rho\\ u)=\mqty(y_* R- U\\ y_* R + U)
\end{multline*}
From the first line, if we substitute $y=y_*$, we get the constraint $-2\rho(y_*)+2-\frac{4}{y_*}u(y_*)=y_*R(y_*)-U(y_*)$. If we divide by $y-y_*$, we then see that in order to have $\rho,u\in C^1$ we need
\[
\frac{y_*(R-R(y_*))-(U-U(y_*))}{y-y_*}
\]
to be continuous at $y_*$. This motivates the following spaces.
\begin{defin} (function spaces)
We let $I_\epsilon=(y_*-\epsilon,y_*+\epsilon)$. We then define our nonlinear space to be
\begin{multline*}
\calN_{\epsilon}=\{(R,U)\in C(I_{\epsilon})\colon \frac{y_* R-U}{y-y_*}\in C(I_{\epsilon}),\\
 R(y_*)=U(y_*)=(y_*R-U)'(y_*)=0,\quad \norm{(R,U)}_{\calN_{\epsilon}}<\infty\}
\end{multline*}
where
\[
\norm{(R,U)}_{\calN_{\epsilon}}=\sup_{y\in I_\epsilon}\abs{R}+\abs{U}+\abs{\frac{y_* R-U}{y-y_*}}.
\]
We then define our linear space to be
\[
\calX_{\epsilon}=\{(\rho,u)\in C^1(I)\colon \rho(y_*)=\rho'(y_*)=\rho(y_*)=\rho'(y_*)=0,\quad \norm{(\rho,u)}_{\calX_{\epsilon}}<\infty\}
\]
where
\[
\norm{(\rho,u)}_{\calX_{\epsilon}}=\sup_{y\in I_\epsilon} \abs{\rho}+\abs{\rho'}+\abs{u}+\abs{u'}
\]
\end{defin}

\begin{lem} (linear continuity)
If epsilon is sufficiently small then there exists a resolvant $T_0$ and constant $C_1>0$ (independent of $\epsilon$) so that $L_0T_0=\id$ and if $(R,U)\in \calN_{\epsilon}$ then
\[
(\rho,u):=T_0(R,U)\in \calX_{\epsilon}\qquad \text{and}\qquad \norm{(\rho,u)}_{\calX}\leq C_1 \norm{(R,U)}_{\calN_{\epsilon}}.
\]
\end{lem}

\begin{proof}
As a first step towards defining $T_0$, we construct the homogeneous solutions. In particular, we rewrite $L_0(\rho,u)=0$ (dropping the subscripts of $\tilde{\rho}_0,\ut_0$) as
\begin{equation}
\mqty(\rho\\ u)'+\frac{1}{(\ut+y)^2-1}\mqty(\ut+y & -\rt\\ -\frac{1}{\rt} & \ut+y)\mqty(\ut'+2\frac{\ut+y}{y} & \rt'+\frac{2\rt}{y}\\ -\frac{\rt'}{\rt^2}+ 2(\ut+y) & \ut'+2\rt)\mqty(\rho\\ u)=0.
\label{rewrittenhomeq}
\end{equation}
We then let $\eta=y-y_*$ and Taylor expand in powers of $\eta$,
\begin{align*}
\frac{1}{(\ut+y)^2-1}&=\frac{y_*}{2\eta}+O(1)\\
\mqty(\ut+y & -\rt\\ -\frac{1}{\rt} & \ut+y)&=\mqty(1 & -\frac{1}{y_*}\\ -y_* & 1)+O(\eta)\\
\mqty(\ut'+2\frac{\ut+y}{y} & \rt'+\frac{2\rt}{y}\\ -\frac{\rt'}{\rt}+ 2(\ut+y) & \ut'+2\rt)&=\mqty(\frac{3}{y_*}-1 & \frac{1}{y_*}-\frac{1}{y_*^2}\\ 5-y_*& \frac{3}{y_*}-1)+O(\eta)
\end{align*}

Thus, we have that
\[
\mqty(\rho\\ u)'+\left(\frac{1}{\eta}\mqty(-1 & 1-\frac{2}{y_*}\\ y_* & 2-y_*)+O(1)\right)\mqty(\rho\\ u)=0.
\]
We note that the coefficient of $\frac{1}{y-y_*}$ has eigenvalues $0,y_*-1$, in particular if $y_*\neq 1$ (the generic case), the linearized system has a pair of homogeneous solutions which are $1,\abs{y-y_*}^{y_*-1}$ near $y=y_*$. The residue of the (negative of) the coefficient of the zeroth order term is
\[
\mqty(1 & \frac{2}{y_*}-1\\ -y_*& y_*-2)=\mqty(y_*-2 & -1\\ y_* & y_*)\mqty(0 & 0\\ 0 & y_*-1)\mqty(y_*-2 & -1\\ y_* & y_*)^{-1}.
\]
By analyticity of $(\rt,\ut)$ (here we assume $\epsilon$ is sufficiently small so that $I_{\epsilon}$ is contained in the radius of convergence of the series at $y_*$), it follows from the Frobenius method that there exists homogeneous solutions $(\rho_{h,0},u_{h,0}),(\rho_{h,1},u_{h,1})$ whose asymptotics as $y\to y_*$ are given by
\[
U(y):=\mqty(\rho_{h,0} & \rho_{h,1}\\ u_{h,0} & u_{h,1})=\mqty(y_*-2+O(\eta) & -\abs{\eta}^{y_*-1}+O(\abs{\eta}^{y_*})\\ y_*+O(\eta) & y_* \abs{\eta}^{y_*-1}+O(\abs{\eta}^{y_*})).
\]
We can then represent the principal matrix solution to $L_{0}(\rho,u)=0$ as
\begin{align*}
S_0(y,y')&:=U(y)U^{-1}(y')\mqty(\ut(y')+y' & \rt(y')\\ \frac{1}{\rt(y')} & \ut(y')+y')^{-1}\\
\end{align*}
We may then define our desired resolvant as
\[
T_0\mqty(R\\ U)(y):=\int_{y_*}^y S_0(y,y')\mqty(R\\ U)(y')dy'
\]
We can asymptotically expand
\[
U^{-1}(y)=\frac{1}{y_*^2-y_*}\mqty(y_* + O(\eta) & 1+O(\eta)\\ -y_*\abs{\eta}^{1-y_*}+O(\abs{\eta}^{2-y_*}) & (y_*-2)\abs{\eta}^{1-y_*}+O(\abs{\eta}^{2-y_*})),
\]
\[
U^{-1}\mqty(\ut(y')+y' & \rt(y')\\ \frac{1}{\rt(y')} & \ut(y')+y')^{-1}=\mqty(0 & 0\\ \frac{-y_*}{2}\abs{\eta}^{-y_*} & \frac{1}{2}\abs{\eta}^{-y_*})+\mqty(O(1) & O(1)\\ O(\abs{\eta}^{1-y_*}) & O(\abs{\eta}^{1-y_*})).
\]
In particular, if $\frac{y_*R-U}{\eta}$ is continuous, then we have that as $\eta\to 0$,
\[
U^{-1}\mqty(\ut(y')+y' & \rt(y')\\ \frac{1}{\rt(y')} & \ut(y')+y')^{-1}\mqty(R\\ U)=\mqty(o(\eta)\\ o(\abs{\eta}^{2-y_*}))
\]
and hence $(\rho,u):=T_0(R,U)$ are $o(\eta)$ as desired, and so we have that
\[
\rho(y_*)=\rho'(y_*)=u(y_*)=u'(y_*)=0.
\]
Using the asymptotics of the fundamental solutions it is straightforward to also check that $\norm{(\rho,u)}_{\calX_{\epsilon}}\lesssim \norm{(R,U)}_{\calN_{\epsilon}}$.
\end{proof}

\begin{lem}(nonlinear estimate)
We have (independent of $\epsilon$ sufficiently small) that there exists $C_2>0$ so that if $(\rho,u)\in \calX$ then 
\[
(R,U):=N_0(\rho,u)\in \calN\qquad \text{and}\qquad \norm{(R,U)}_{\calX}\leq C_2 \norm{(\rho,u)}^2.
\]
\end{lem}
\begin{proof}
We recall \eqref{N0defin}, the definition of $N_0$, and decompose (dropping the subscripts of $\tilde{\rt}_0,\tilde{\ut}_0$)
\[
\frac{R}{y-y_*}=\frac{u}{y-y_*}\rho'+\frac{\rho}{y-y_*}u'+2\frac{\rho}{y-y_*}\frac{u}{y},
\]
\[
\frac{U}{y-y_*}=-\frac{\rho}{y-y_*}\frac{\rho'}{\tilde{\rho}_0(\rt+\rho)}+\frac{\rt}{y-y_*}\frac{\rho \rt'}{\rt^2(\rt+\rho)}+\frac{u}{y-y_*}u'+2\frac{\rho}{y-y_*}u.
\]
As $\rho,u\in C^1$, $\frac{\rho}{y-y_*}$, $\frac{u}{y-y_*}$ converge as $y\to y_*$, so we see that $\frac{R}{y-y_*},\frac{U}{y-y_*}$ are continuous at $y_*$. Similarly, we may bound $\frac{\rho}{y-y_*},\frac{u}{y-y_*}$ by $\rho',u'$ respectively and as each term is quadratic in $\rho,u$, we obtain
\[
\abs{\frac{R}{y-y_*}},\abs{\frac{U}{y-y_*}}\lesssim \norm{(\rho,u)}_{X_{\epsilon}}^2
\]
which implies the desired continuity bound.
\end{proof}

We now prove Lemma \ref{locallem}.b.
\begin{proof}

We rewrite \eqref{diffeq} as
\[
(\rho,u)=-T_0 N_0(\rho,u).
\]
As $(\rho,u)\in C^1$ and their first order Taylor coefficients vanish at $y_*$, there exists some small $\epsilon$ such that $\norm{(\rho,u)}_{\calX_{\epsilon}}\leq \frac{1}{2 C_1 C_2}$. We then have that
\[
\norm{(\rho,u)}_{\calX}\leq C_1C_2\norm{(\rho,u)}_{\calX}^2\leq \frac{1}{2}\norm{(\rho,u)}_{\calX}
\]
which implies that $\norm{(\rho,u)}_{\calX_{\epsilon}}=0$.
\end{proof}

We finally show analyticity at $y=0$. This is easier, and we proceed in a manner which is similar to the sonic point case. In particular, we show the following lemma.
\begin{lem}
\begin{enumerate}[wide,label=(\alph*)]
\item
If one fixes $\rho_0>0$, then there exists a real analytic solution to the self-similar Euler-Poisson system, $\tilde{\rho}_0,\tilde{u}_0$, defined on a nbhd of $0$ satisfying the boundary conditions
\begin{equation}
\tilde{\rho}_0(0)=\rho_0,\qquad \tilde{\rho}_0'(0)=0,\qquad \tilde{u}_0(0)=0,\qquad \tilde{u}_0(0)=-\frac{2}{3}y.
\label{originboundaryconds}
\end{equation}
\item
Let $(\tilde{\rho}_1,\tilde{u}_1)$ be a $C^1$ solution to the self-similar Euler-Poisson system defined on an interval containing $0$ satisfying the boundary conditions \eqref{originboundaryconds}. Then there exists an interval containing $0$ such that $(\tilde{\rho}_1,\tilde{u}_1)=(\tilde{\rho_0},\tilde{u}_0)$.
\end{enumerate}
\end{lem}
\begin{proof}
In this case part (a) follows from Theorem 2.12 of \cite{ghj}. For part (b), we consider the operators $L_0,N_0$ defined previously, and define the following spaces.
\[
\tilde{\calN_{\epsilon}}=\{(R,U)\in C([0,\epsilon)\colon R(0)=U(0)=0,\quad \norm{(R,U)}_{\tilde{\calN}_{\epsilon}}<\infty)\},
\]
\[
\norm{(R,U)}_{\tilde{\calN}_{\epsilon}}=\sup_{y\in [0,\epsilon)}\abs{R}+\abs{U},
\]
\[
\tilde{\calX_{\epsilon}}=\{(\rho,u)\in C^1([0,\epsilon))\colon \rho(0)=\rho'(0)=u(0)=u'(0)=0,\quad \norm{(\rho,u)}_{\tilde{\calX}_{\epsilon}}<\infty\},
\]
\[
\norm{(\rho,u)}_{\tilde{\calX_{\epsilon}}}=\sup_{y\in [0,\epsilon)}\abs{\rho}+\abs{\rho'}+\abs{u}+\abs{u'}.
\]
We then note that near $y=0$, we can Taylor expand the coefficient in \eqref{rewrittenhomeq} as
\[
\frac{1}{(\ut+y)^2-1}\mqty(\ut+y & -\rt\\ -\frac{1}{\rt} & \ut+y)\mqty(\ut'+2\frac{\ut+y}{y} & \rt'+\frac{2\rt}{y}\\ -\frac{\rt'}{\rt^2}+ 2(\ut+y) & \ut'+2\rt)=\mqty(0 & 0\\ 0 & \frac{2}{y})+O(1)
\]
and hence we may use the Frobenius method to construct fundamental solutions whose leading order asymptotics as $y\to 0$ are given by
\[
\tilde{U}=\mqty(\rho_{h,2} & \rho_{h,3}\\ u_{h,2} & u_{h,3})=\mqty(1 + O(y) & 0+O(\frac{1}{y})\\ 0+O(y) & \frac{1}{y^2} + O(\frac{1}{y})).
\]
We then define $\tilde{S}_0,\tilde{T_0}$ by
\[
\tilde{S_0}(y,y'):=\tilde{U}(y)\tilde{U}^{-1}(y')\mqty(\ut(y')+y' & \rt(y')\\ \frac{1}{\rt(y')} & \ut(y')+y')^{-1},\quad \tilde{T_0}\mqty(R\\ U)(y):=\int_0^y\tilde{S}_0(y,y')\mqty(R\\ U)(y')dy'.
\]
We can then asymptotically expand
\[
\tilde{U}^{-1}\mqty(\ut(y)+y & \rt(y)\\ \frac{1}{\rt(y)} & \ut(y)+y)^{-1}=\mqty(0 & \frac{1}{\rho_0}\\y^2 \rho_0 & 0)+\mqty(O(y) & O(y)\\ O(y^3) & O(y^3))
\]
from which it is evident that $\tilde{T}_0$ is bounded from $\tilde{\calN_{\epsilon}}$ to $\tilde{\calX}_{\epsilon}$ and there exists $\tilde{C}_1$ independent of $\epsilon$ such that if $(\rho,u)=\tilde{T_0}(R,U)$,
\[
\norm{(\rho,u)}_{\tilde{\calX}_{\epsilon}}\leq \tilde{C}_1 \norm{(R,U)}_{\tilde{\calN}_{\epsilon}}.
\]
Again, straightforward estimates using \eqref{N0defin} show that if $(\rho,u)\in \tilde{\calX}_{\epsilon}$, then $(R,U):=N_0(\rho,u)\in \tilde{\calN}_{\epsilon}$ and that there exists $\tilde{C}_2>0$ such that
\[
\norm{(R,U)}_{\tilde{\calN}_{\epsilon}}\leq \tilde{C}_2\norm{(\rho,u)}_{\tilde{\calX}_{\epsilon}}^2.
\]
We may then conclude the proof of part (b) as before, namely by defining $(\rho,u)=(\tilde{\rho}_1-\tilde{\rho}_0,\tilde{u}_1-\tilde{u}_0)$, noting that it satisfies
\[
(\rho,u)=-\tilde{T}_0N_0(\rho,u)
\]
then choosing $\epsilon$ sufficiently small so that $\norm{\rho,u}_{\tilde{\calX}_{\epsilon}}<\frac{1}{2 \tilde{C}_1\tilde{C}_2}$, and then deriving
\[
\norm{(\rho,u)}_{\tilde{\calX}_{\epsilon}}\leq \frac{1}{2}\norm{(\rho,u)}_{\tilde{\calX}_{\epsilon}}
\]
and concluding $(\rho,u)=0$.
\end{proof}
\section{Deferred proofs from Section \ref{extsection}}
\label{extleftovers}
\subsection{Proof of Lemma \ref{explicitlemma}}
\label{explicitlemmaproof}
\begin{proof}
We discuss how to explicitly solve the homogeneous ODE
\[
\mqty(z & z\\ \frac{1}{z} & z)\mqty(p'\\ \w')+\mqty(0 & 1\\ \frac{2}{z^2} & \frac{2}{z^2}+1)\mqty(p\\\w)=0.\\
\]
Our first step is to temporarily assume $z\neq 1$, and invert the first matrix to obtain
\[
\mqty(p'\\ \w')+\frac{1}{z(z^2-1)}\mqty(-2 & -2\\ 2 & z^2+1)\mqty(p\\\w)=0.
\]
We then let $q=p+\w$ and note that we have
\[
\mqty(p'\\ q')+\mqty(0 & -\frac{2}{z(z^2-1)}\\ \frac{-1}{z} & \frac{1}{z})\mqty(p\\ q)=0.
\]
If we then substitute $q=\frac{z(z^2-1)}{2}p'$ into the second row we get that $p$ solves
\begin{equation}
p''+\frac{4z^2-2}{z(z^2-1)}p'-\frac{2}{z^2(z^2-1)}p=0.
\label{symmetriceq}
\end{equation}
We then let $p(z)=f(z^2)$. In this case, if we let $x=z^2$, we have that $f(x)$ solves
\[
f''+\frac{5x-3}{2x(x-1)}f'-\frac{1}{2x^2(x-1)}f=0.
\]
We observe that this ODE is singular at $x=\{0,1,\infty\}$ with simple singularities, and so by classical ODE theory we can solve it using hypergeometric functions. It is convenient at this point to switch variables to $f(x)=g(1-\frac{1}{x})$, $g(\xi)$ satisfies
\begin{equation}
g''+\frac{\xi-2}{2\xi(\xi-1)}g'+\frac{1}{2\xi(\xi-1)}g=0.
\label{hypergeomeq}
\end{equation}
where $\xi=1-\frac{1}{x}$. We recognize this as the hypergeometric equation with $a=-\frac{\gamma}{2},b=-\frac{\bar{\gamma}}{2},c=1$, where
\[
\gamma=\frac{1}{2}+i\frac{\sqrt{7}}{2}.
\]
The unique solution to \eqref{hypergeomeq} which is smooth on $(-\infty,1)$ and satisfies the normalization condition $g_{1}(0)=1$ is
\begin{equation}
g_1(\xi)=F(-\frac{\gamma}{2},-\frac{\bar{\gamma}}{2},1,\xi),
\label{u1def}
\end{equation}
 where $F(a,b,c,\xi)$ denotes the principal branch of Gauss's hypergeometric function. By inverting the changes of variables, we recover the following solution to $L(p,\omega)=0$:
 \[
\mqty(p_{hom}(z)\\ \omega_{hom}(z)):=\mqty(1 & 0\\ -1 & \frac{z^2-1}{z^2})\mqty(g_1(1-\frac{1}{z^2})\\ g_1'(1-\frac{1}{z^2})).
 \]
From this representation, \eqref{homogeneoustaylor} follows immediately. This yields (a).

To construct $U_{\infty}$, we start by defining the following analytic functions for $\xi\in (0,1)$:
\[
g_3(\xi)=F(-\frac{\gamma}{2},-\frac{\bar{\gamma}}{2},-\frac{1}{2},1-\xi),\qquad g_{4}(\xi)=-\frac{2}{3}(1-\xi)^{3/2}F(1+\frac{\gamma}{2},1+\frac{\bar{\gamma}}{2},\frac{5}{2},1-\xi).
\]
We use these fundamental solutions to construct a fundamental matrix for \eqref{hypergeomeq},
\[
V_{\infty}(\xi):=\mqty(g_3(\xi) & g_4(\xi)\\ g_3'(\xi) & g_4'(\xi)).
\]
We note that the Wronskian $W_{\infty}:=\det V_{\infty}$ satisfies the following ODE and asymptotic expansion as $\xi\to 1$
\[
W_{\infty}'+\frac{\xi-2}{2\xi(\xi-1)}W_{\infty}=0,\qquad \qquad W_{\infty}=(1-\xi)^{1/2}+O((1-\xi)^{3/2}).
\]
Solving this ODE gives 
\[
W_{\infty}(\xi)=\frac{\sqrt{1-\xi}}{\xi},\qquad V_{\infty}^{-1}(\xi)=\frac{\xi}{\sqrt{1-\xi}}\mqty(g_4'(\xi) & -g_4(\xi)\\ -g_3'(\xi) & g_3(\xi)).
\]
We note that these are two linearly independent solutions to \eqref{hypergeomeq}. Thus, we may define on the following for $z\in (1,\infty)$
\[
U_{\infty}(z):=\mqty(1 & 0\\ -1 & \frac{z^2-1}{z^2})V_{\infty}(1-\frac{1}{z^2}).
\]
and it satisfies $LU_{\infty}=0$, and as $z\to \infty$. Taking the limit as $z\to \infty$ yields the stated asymptotics for $U_{\infty}$. Similarly, using the formula for $V_{\infty}^{-1}(\xi)$, one can straightforwardly compute the asymptotics for $U_{\infty}^{-1}$.
To construct $U_0$, we start by defining the following functions for $\xi\in (-\infty,0)$,
\[
g_a(\xi)=(-\xi)^{\gamma/2}F(-\frac{\gamma}{2},-\frac{\gamma}{2},1-i\frac{\sqrt{7}}{2},\frac{1}{\xi}),\qquad g_b=(-\xi)^{\bar{\gamma}/2}F(-\frac{\bar{\gamma}}{2},-\frac{\bar{\gamma}}{2},1+i\frac{\sqrt{7}}{2},\frac{1}{\xi}).
\]
Then, one lets
\[
g_5=\frac{g_a+g_b}{2},\qquad g_6=-\frac{g_a-g_b}{2i}.
\]
These are now real valued. One then defines for $\xi\in (-\infty,0)$,
\[
V_0(\xi)=\mqty(g_5(\xi) & g_{6}(\xi)\\ g_{5}'(\xi) & g_{6}'(\xi)),\qquad W_0=\det V_{0}.
\]
We obtain the same ODE for the Wronskian with the boundary condition that as $\xi\to -\infty$,
\[
W_{0}=\frac{\sqrt{7}}{4}(-\xi)^{-1/2}+O((-\xi)^{-3/2}).
\]
Solving the ODE for with this boundary condition gives
\[
W_0=-\frac{\sqrt{7}}{4}\frac{\sqrt{1-\xi}}{\xi},\qquad V_{0}^{-1}(\xi)=\frac{-4}{\sqrt{7}}\frac{\xi}{\sqrt{1-\xi}}\mqty(g_6'(\xi) & -g_6(\xi)\\ -g_5'(\xi) & g_5(\xi)).
\]
Thus, we may define for $z\in (0,1)$
\[
U_0(z):=\mqty(1 & 0\\ -1 & \frac{z^2-1}{z^2})V_0(1-\frac{1}{z^2}).
\]
which satisfies $L U_0(z)=0$. As with before, we obtain the asymptotics for $U_0$ and $U_0^{-1}$. This concludes the proof of (b). 

For part (c), we first note that by the connection formulae for hypergeometric functions (equation 5.10.12 of \cite{olver} for instance) we see that for $\xi\in (0,1)$,
\[
g_1(\xi)=\frac{\Gamma(\frac{3}{2})}{\Gamma(1+\frac{\gamma}{2})\Gamma(1+\frac{\bar{\gamma}}{2})}g_3(\xi)-\frac{3}{2}\frac{\Gamma(-\frac{3}{2})}{\Gamma(-\frac{\gamma}{2})\Gamma(-\frac{\bar{\gamma}}{2})}g_4(\xi).
\]
If we let $\mu_3=\frac{\Gamma(\frac{3}{2})}{\Gamma(1+\frac{\gamma}{2})\Gamma(1+\frac{\bar{\gamma}}{2})}$ and $\mu_4=-\frac{3}{2}\frac{\Gamma(-\frac{3}{2})}{\Gamma(-\frac{\gamma}{2})\Gamma(-\frac{\bar{\gamma}}{2})}$, then by linearity we have
\[
\mqty(p_1\\ \w_1)=U_\infty\mqty(\mu_3\\\mu_4),
\]
yielding the stated asymptotics for $(p_{hom},\omega_{hom})$. By the another connection formula (equation 5.10.14 of \cite{olver} for instance), for $\xi\in (-\infty,0)$, we have
\[
g_1(\xi)=\frac{\Gamma(1)\Gamma(-\frac{\bar{\gamma}-\gamma}{2})}{\Gamma(-\frac{\bar{\gamma}}{2})\Gamma(1+\frac{\gamma}{2})}g_a(\xi)+\frac{\Gamma(1)\Gamma(-\frac{\gamma-\bar{\gamma}}{2})}{\Gamma(-\frac{\gamma}{2})\Gamma(1+\frac{\bar{\gamma}}{2})}g_b(\xi).
\]
We define real constants $\mu_5,\mu_6$ by 
\[
\frac{\Gamma(1)\Gamma(-\frac{\bar{\gamma}-\gamma}{2})}{\Gamma(-\frac{\bar{\gamma}}{2})\Gamma(1+\frac{\gamma}{2})}=\frac{\mu_5}{2}+i\frac{\mu_6}{2},
\]
so this becomes
\[
g_1(\zeta)=\mu_5 g_5(\zeta)+\mu_6 g_6(\zeta).
\]
By linearity, we have
\[
\mqty(p_{hom}\\\w_{hom})=U_0\mqty(\mu_5\\ \mu_6).
\]
We then compute that if $c_1:=\sqrt{\mu_5^2+\mu_6^2}$, then $c_1\neq 0$, which yields the stated asymptotics for $(p_{hom},\omega_{hom})$ for some $d_1\in \R$ finishing the proof of (c).

For part (d), we start by defining, for $\xi\in (0,1)$,
\[
g_{2,\infty}(\xi)=\mqty(1 & 0)V_{\infty}\mqty(-\frac{1}{2\mu_4}\\ \frac{1}{2\mu_3}).
\]
We note that
\[
\mqty(g_1 & g_{2,\infty}\\ g_{1}' & g_{2,\infty}')=V_{\infty}\mqty(\mu_3 & -\frac{1}{2\mu_4}\\ \mu_4 & \frac{1}{2\mu_3}).
\]
and take a determinant to derive
\[
g_1 g_{2,\infty}'-g_{1}'g_{2,\infty}=\frac{\sqrt{1-\xi}}{\xi}.
\]
We then note that because $g_{1}=1+O(\xi)$ as $\xi\to 0$, in this regime we can rearrange the equation as
\[
(\frac{g_{2,\infty}}{g_1})'=\frac{1}{\xi}+(\frac{\sqrt{1-\xi}}{\xi g_1^2}-\frac{1}{\xi}).
\]
which is equivalent to
\[
(\frac{g_{2,\infty}-\log\abs{\xi}g_1}{g_{1}})'=\frac{\sqrt{1-\xi}-g_1^2}{\xi g_1^2}.
\]
Thus, for some constant $C$, we have
\[
g_{2,\infty}=g_{1}(\log\abs{\xi}+C+\int_{0}^{\xi}\frac{\sqrt{1-\xi'}-g_1(\xi')^2}{\xi'g_1(\xi')^2}).
\]
In particular, as $\xi\to 0$, for some constant $C$
\[
g_{2,\infty}=\log\abs{\xi}+C+O(\abs{\xi}\log\abs{\xi}),\qquad g_{2,\infty}'=\frac{1}{\xi}+O(\log\abs{\xi}).
\]
Thus, as $z\to 1$, we have
\[
V_{\infty}(1-\frac{1}{z^2})\mqty(\mu_3 & -\frac{1}{2\mu_4}\\ \mu_4 & \frac{1}{2\mu_3})=\mqty(1+O(z-1)& \log\abs{z-1}+C+\log(2)+O(\abs{z-1}\log\abs{z-1})\\ \frac{1}{2}+O(z-1) & \frac{1}{2(z-1)}+O(\log\abs{z-1})).
\]
Thus, as $z\to 1$, we have that
\[
U_{\infty}\mqty(\mu_3 & -\frac{1}{2\mu_4}\\ \mu_4 & \frac{1}{2\mu_3})=\mqty(1 & \log\abs{z-1}+C+\log(2)\\-1 & -\log\abs{z-1}+1-C-\log(2))+\mqty(O(z-1) & O(\abs{z-1}\log\abs{z-1})\\ O(z-1) & O(\abs{z-1}\log\abs{z-1})).
\]
Thus, if
\[
M_{\infty}=\mqty(\mu_3 & -\frac{1}{2\mu_4}\\ \mu_4 & \frac{1}{2\mu_3})\mqty(1 & -C-\log(2)+\frac{1}{2}\\ 0 & 1)
\]
we obtain the asymptotics for $U_{\infty}M_{\infty}$ claimed in \eqref{near1asymptot}. One can similarly define
\[
g_{2,0}(\xi)=\mqty(1 & 0)V_{\infty}\mqty(-\frac{1}{2\mu_6}\\ \frac{1}{2\mu_5})\frac{-4}{\sqrt{7}},
\]
and through a similar process construct $M_{0}$. In these cases, we have that
\[
\det(U_{\infty}M_{\infty})=\det(U_0M_0)=\frac{1}{z}
\]
and asymptotically expanding the cofactor form of the inverse gives (d).

For part (e), we note that for $z,z'>1$ we may express
\begin{align*}
S(z,z')\mqty(z' & z'\\ \frac{1}{z'} & z')^{-1}&=U_{\infty}(z)M_{\infty} (U_{\infty}M_{\infty})^{-1}(z')\mqty(z' & z'\\ \frac{1}{z'} & z')^{-1},
\end{align*}
and as $z'\to 1$, we can asymptotically expand using part (d) to obtain
\begin{equation}
(U_{\infty}M_{\infty})^{-1}(z')\mqty(z' & z'\\ \frac{1}{z'} & z')^{-1}=\mqty(\frac{1}{2(z-1)} & -\frac{1}{2(z-1)}\\ 0 & 0)+\mqty(O(\log\abs{z-1}) & O(\log\abs{z-1})\\ O(1) & O(1)).
\label{asymptoticcancellation}
\end{equation}
Thus, as $\frac{P-\Omega}{z-1},P,\Omega$ are continuous and $P(0)=\Omega(0)=0$, we have that as $z\to 1^+$
\[
\int_{1}^z(U_{\infty}M_{\infty})^{-1}(z')\mqty(z' & z'\\ \frac{1}{z'} & z')\mqty(P\\ \Omega)(z')=\mqty(o(\abs{z-1}\log\abs{z-1})\\ o(z-1)).
\]
and hence we see that $p,\omega=o(\abs{z-1}\log\abs{z-1})$ as $z\to 1^+$ and $p+\omega=o(\abs{z-1}^2\log\abs{z-1})$ (to be used shortly). In particular, $p,\omega$ are continuous on $[1,\infty)$. One argues similarly using $U_{0},M_0$ for $(0,1]$. Moreover, the function is $C^1$ on $(0,\infty)\setminus 1$ and satisfies $L(p,\omega)=P,\Omega$. For $z\neq 1$, we can invert the leading coefficient of $L$ to obtain
\begin{equation}
\mqty(p'\\ \omega')=\mqty(z & z\\ \frac{1}{z} & z)^{-1}\mqty(P\\ \Omega)-\frac{1}{z(z^2-1)}\mqty(-2 & -2\\ 2 & z^2+1)\mqty(p\\ \omega)
\label{rewrittenform}
\end{equation}
 and taking the limit as $z\to 1$ (using $p+\omega=o(\abs{z-1}^2\log\abs{z-1})$ to control the latter term) shows that $(p,\omega)\in C^1(0,\infty)$ and
\[
p(1)=p'(1)=\omega(1)=\omega'(1)=0.
\]
\end{proof}
\subsection{Proof of Lemma \ref{linearlemma}}
\label{linearlemmaproof}
\begin{proof}
We suppose $(P,\Omega)\in N_{y_0}$. We note that by Lemma \ref{explicitlemma}.e, it follows that $(p,\omega):=R(P,\Omega)$ is well-defined, is in $C^1[y_0,\infty)$ and satisfies
\[
p(1)=p'(1)=\omega(1)=\omega'(1)=0.
\]

By \eqref{asymptoticcancellation}, and the asymptotics of $U_{\infty}M_{\infty}$ as $z\to 0$, we from the same arguments as in the proof of Lemma \ref{explicitlemma}.e, now done quantitatively that
\[
\sup_{1\leq \abs{z}\leq 2} \abs{p}+\abs{\omega}+\abs{\frac{p+\omega}{z-1}}\lesssim \sup_{1\leq z\leq 2}\abs{P}+\abs{\Omega}+\abs{\frac{P-\Omega}{z-1}}.
\]
One then uses \eqref{rewrittenform} and these bounds to conclude
\[
\sup_{1\leq \abs{z}\leq 2} \abs{p'},\abs{\omega'}\lesssim \sup_{1\leq z\leq 2}\abs{P}+\abs{\Omega}+\abs{\frac{P-\Omega}{z-1}}.
\]
As with before, the argument on $[\frac{1}{2},1]$ is the same, so we get
\[
\sup_{\frac{1}{2}\leq \abs{z}\leq 2} \abs{p'},\abs{\omega'}\lesssim \sup_{\frac{1}{2}\leq z\leq 2}\abs{P}+\abs{\Omega}+\abs{\frac{P-\Omega}{z-1}}.
\]

For $z\geq 2$, we express \eqref{Rdef} as
\begin{multline*}
\mqty(p\\ \omega)=U_{\infty}(z)\left( M_{\infty}\int_{1}^2 (U_{\infty}M_{\infty})^{-1}(z')\mqty(z' & z'\\ \frac{1}{z'} & z')^{-1}\mqty(P\\ \Omega)+\int_{2}^{z}U_{\infty}^{-1}(z')\mqty(z' & z'\\ \frac{1}{z'} & z')^{-1}\mqty(P\\ \Omega)(z')dz'\right).
\end{multline*}
As in the control near $z=1$, the first integral is controlled by $\sup_{1\leq z\leq 2}\abs{P}+\abs{\Omega}+\abs{\frac{P-\Omega}{z-1}}$. For the second integral, we then note that by Lemma \ref{explicitlemma}.b, we have that as $z\to \infty,$
\[
U_{\infty}^{-1}(z)\mqty(z & z\\ \frac{1}{z} & z)^{-1}=\mqty(\frac{1}{z} & -\frac{1}{z}\\ 0 & 1)+\mqty(O(\frac{1}{z^3}) & O(\frac{1}{z^3})\\ O(\frac{1}{z^2}) & O(\frac{1}{z^2}))
\]
so that for $z'\geq 2$,
\[
\abs{U_{\infty}^{-1}(z')\mqty(z' & z'\\ \frac{1}{z'} & z')^{-1}\mqty(P\\ \Omega)}\lesssim \frac{1}{(z')^2}\left(\sup_{z\geq 2}\abs{zP}+\abs{z^2 \Omega}\right),
\]
so the second integral appearing above converges as $z\to \infty$ by $\sup_{z\geq 2}\abs{zP}+\abs{z^2 \Omega}$. By the asymptotics of $U_{\infty}(z)$ as $z\to \infty$, we see that
\[
\sup_{z\geq 2}\abs{p}+\abs{z \omega}\lesssim \sup_{1\leq z\leq 2}\abs{P}+\abs{\Omega}+\abs{\frac{P-\Omega}{z-1}}+\sup_{z\geq 2}\abs{zP}+\abs{z^2 \Omega},
\]
as well as the qualitative observation that $\lim_{z\to \infty}p$, $\lim_{z\to \infty} z\omega$ exist. For $z\leq \frac{1}{2}$, we use a similar integral formulation, namely
\begin{multline*}
\mqty(p\\ \omega)=-U_{0}(z)\left( M_{0}\int_{\frac{1}{2}}^1 (U_{0}M_{0})^{-1}(z')\mqty(z' & z'\\ \frac{1}{z'} & z')^{-1}\mqty(P\\ \Omega)\right.\\
\left.+\int_{z}^{\frac{1}{2}}U_{\infty}^{-1}(z')\mqty(z' & z'\\ \frac{1}{z'} & z')^{-1}\mqty(P\\ \Omega)(z')dz'\right).
\end{multline*}
 and to control the second integral we compute that as $z\to 0$,
\[
U_{0}^{-1}(z)\mqty(z & z\\ \frac{1}{z} & z)^{-1}=\mqty(-z^{-1/2}\sin(\fr\log(z)) & 0\\ z^{-\frac{1}{2}}\cos(\fr \log(z)) & 0)+\mqty(O(z^{\frac{3}{2}}) & O(z^{\frac{3}{2}})\\ O(z^{\frac{3}{2}}) & O(z^{\frac{3}{2}}))
\]
so that for $y_0\leq z'\leq \frac{1}{2}$
\[
\abs{U_{0}^{-1}(z')\mqty(z' & z'\\ \frac{1}{z'} & z')^{-1}\mqty(P\\ \Omega)}\lesssim \frac{1}{(z')^{3/2}}\left(\sup_{y_0\leq z\leq \frac{1}{2}}\abs{z P}+\abs{z^3 \Omega}\right).
\]
By considering the asymptotics of $U_{0}(z)$ as $z\to 0$, and using $z^{-1/2}\leq y_0^{-1/2}$, we obtain
\[
\sup_{y_0\leq z\leq \frac{1}{2}}\abs{z^{1/2}p}+\abs{z^{1/2}\omega}\lesssim y_0^{-1/2}\left(\sup_{\frac{1}{2}\leq y\leq 1}\abs{P}+\abs{\Omega}+\abs{\frac{P-\Omega}{z-1}}+\sup_{y_0\leq z\leq \frac{1}{2}}\abs{zP}+\abs{z^3\Omega}\right).
\]
As a note, this regime, $z\leq \frac{1}{2}$, is where the factor of $y_0^{-1/2}$ in the Lemma arises.

In summary, we have shown so far that $p,\omega\in C^1([y_0,\infty))$, they satisfy the correct boundary conditions at the origin, $\lim_{z\to \infty}(p,z\omega)$ exists and
\[
\sup_{z\geq 1}\abs{p}+\abs{z \omega}+\sup_{y_0\leq z\leq 1}z^{1/2}(\abs{p}+\abs{\omega})\lesssim y_0^{-\frac{1}{2}} \norm{P,\Omega}_{N_{y_0}}.
\]
We then note that for $z\neq 1$, we may solve for $(p',\omega')$ using \eqref{rewrittenform}. Taking the limit as $z\to\infty$ shows that
\[
\sup_{z\geq 1}\abs{D_{\infty}p}+\abs{D_{\infty}(z\omega)}\lesssim y_0^{-\frac{1}{2}}\norm{P,\Omega}_{N_{y_0}}
\]
and $\lim_{z\to\infty}(D_{\infty}p,D_{\infty}\omega)$ exists. Looking at the limit as $z\to 0$ shows the final desired estimate, namely
\[
\sup_{y_0\leq z\leq 1}z^{3/2}(\abs{p'}+\abs{\omega'})\lesssim y_0^{-\frac{1}{2}}\norm{P,\Omega}_{N_{y_0}}.
\]
\end{proof}

\subsection{Proof of Lemma \ref{nonlinearlemma}}
\label{nonlinearlemmaproof}
\begin{proof}
As a first step, we will check that if $p,\omega\in C^{1}[y_0,\infty)$ with vanishing first order Taylor polynomials at $z=1$ then $P,\Omega:=N(p,\omega)$ satisfy $\frac{P-\Omega}{z-1}\in C^0[y_0,\infty)$. As a note, this is a bit subtle as $\frac{p'}{z-1},\frac{\omega'}{z-1}$ are not necessarily continuous. Before demonstrating this, we first rewrite $P-\Omega$ to make more explicit the smallness as we take $\epsilon$ to zero. Indeed, we start by computing
\begin{align*}
A_1-A_2&=\hat{\omega}'(z-\beta)\epsilon^2(\frac{\hat{p}}{y_*\alpha}-\frac{1}{y_*\alpha}-\hat{\omega}-1)\\
&+\epsilon \hat{p}'\left(\frac{(z-\beta)}{y_*\alpha}(y_*+\epsilon \hat{\omega})-\frac{y_*\alpha}{(z-\beta)}\frac{1}{1+\frac{\epsilon}{y_*\alpha}(\hat{p}-1)}\right)\\
&+\epsilon^2(\frac{\hat{p}}{y_*\alpha}-\frac{1}{y_*\alpha}-\hat{\omega}-1)\hat{\omega}-\frac{2(y_*\alpha)^2}{(z-\beta)^2}\left((1+\frac{\epsilon}{y_*\alpha}(\hat{p}-1))(y_*+\epsilon \hat{\omega})-1\right),\\
\epsilon(L_1-L_2)&=\epsilon(z-\frac{1}{z})\hat{p}'-\frac{2\epsilon}{z^2}(\hat{p}+\hat{\omega}),\\
-\epsilon(P-\Omega)&=\hat{\omega}'(z-\beta)\epsilon^2(\frac{\hat{p}}{y_*\alpha}-\frac{1}{y_*\alpha}-\hat{\omega}-1)+\hat{p}'\epsilon \left(\frac{(z-\beta)}{y_*\alpha}(y_*+\epsilon \hat{\omega})\right.
\\
&\qquad \left.-\frac{y_*\alpha}{(z-\beta)}\frac{1}{1+\frac{\epsilon}{y_*\alpha}(\hat{p}-1)}-\frac{(z^2-1)}{z}\right)+\epsilon^2(\frac{\hat{p}}{y_*\alpha}-\frac{1}{y_*\alpha}-\hat{\omega}-1)\hat{\omega}\\
&\qquad -\frac{2(y_*\alpha)^2}{(z-\beta)^2}\left((1+\frac{\epsilon}{y_*\alpha}(\hat{p}-1))(y_*+\epsilon \hat{\omega})-1\right)+\frac{2\epsilon }{z^2}(\hat{p}+\hat{\omega}).\\
\end{align*}
We then compute
\begin{multline*}
\frac{z-\beta}{y_*\alpha}(y_*+\epsilon \hat{\omega})-\frac{y_*\alpha}{z-\beta}\frac{1}{1+\frac{\epsilon}{y_*\alpha}(\hat{p}-1)}-\frac{z^2-1}{z}\\
=(z-1)(\frac{\beta}{y_*\alpha}+\frac{\beta}{z(z-\beta)})+\epsilon(1+\frac{z-1}{y_*\alpha})(\hat{\omega}+1)+\frac{\epsilon}{y_*\alpha}(1-\frac{z-1}{z-\beta})\frac{\hat{p}-1}{1+\frac{\epsilon}{y_* \alpha}(\hat{p}-1)}\\
=\epsilon\left((z-1)(\frac{\beta}{\epsilon y_*\alpha}+\frac{\beta}{\epsilon z(z-\beta)})+(1+\frac{z-1}{y_*\alpha})(\hat{\omega}+1)+\frac{1}{y_*\alpha}(1-\frac{z-1}{z-\beta})\frac{\hat{p}-1}{1+\frac{\epsilon}{y_* \alpha}(\hat{p}-1)}\right)
\end{multline*}
and also
\begin{multline*}
-\frac{2(y_*\alpha)^2}{(z-\beta)^2}\left((1+\frac{\epsilon}{y_*\alpha}(\hat{p}-1))(y_*+\epsilon \hat{\omega})-1\right)+\frac{2\epsilon }{z^2}(\hat{p}+\hat{\omega})\\
=-\frac{2(y_*\alpha)^2}{(z-\beta)^2}\left(\epsilon(\hat{\omega}+1)+\frac{\epsilon}{y_*\alpha}(\hat{p}-1)+\frac{\epsilon^2}{y_*\alpha}(\hat{p}-1)(\hat{\omega}+1)\right)+\frac{2\epsilon }{z^2}(\hat{p}+\hat{\omega})\\
=\epsilon^2\left(\frac{-2(y_*\alpha)}{(z-\beta)^2}\left((\hat{p}-1)(\hat{\omega}+1)-\frac{\beta}{\epsilon}\right)+2\hat{p}\frac{\beta}{\epsilon}\frac{z^2-2z+\beta)}{z^2(z-\beta)^2}+2\hat{\omega}\frac{\beta}{\epsilon}\frac{(z-1)(2z-\beta-\beta z)}{z^2(z-\beta)^2}\right).
\end{multline*}
Thus, if we define
\begin{align*}
M_1(a,b)&:=-(z-\beta)(\frac{a-1}{y_*\alpha}-(b+1)),\\
M_2(a,b)&:=-\left((z-1)(\frac{\beta}{\epsilon y_*\alpha}+\frac{\beta}{\epsilon z(z-\beta)})+(1+\frac{z-1}{y_*\alpha})(b+1)\right.\\
&\qquad\left.+\frac{1}{y_*\alpha}(1-\frac{z-1}{z-\beta})\frac{a-1}{1+\frac{\epsilon}{y_* \alpha}(a-1)}\right),\\
M_3(a,b)&:=-(\frac{a-1}{y_*\alpha}-(b+1))b,\\
M_4(a,b)&:=\frac{2(y_*\alpha)}{(z-\beta)^2}\left((a-1)(b+1)-\frac{\beta}{\epsilon}\right)\\
&\qquad-2a\left(\frac{\beta}{\epsilon}\frac{z^2-2z+\beta)}{z^2(z-\beta)^2}\right)-2b \frac{\beta}{\epsilon}\left(\frac{(z-1)(2z-\beta-\beta z)}{z^2(z-\beta)^2}\right).
\end{align*}
we have
\begin{equation}
P-\Omega=\epsilon \hat{\omega}'M_1(\hat{p},\hat{\omega})+\epsilon\hat{p}'M_2(\hat{p},\hat{\omega})+\epsilon M_3(\hat{p},\hat{\omega})+\epsilon M_4(\hat{p},\hat{\omega}).
\label{diffrewritten}
\end{equation}
In preparation for considering the behavior as $z\to \infty$, we can similarly factor a power of $\epsilon$ out from $P$ and $\Omega$ individually. To do this, we first compute
\begin{align*}
\epsilon(1+\frac{\epsilon}{y_*\alpha}(\hat{p}-1))(\hat{\omega}(z-\beta))'-\epsilon (z\hat{\omega})'
&=\epsilon^2\left(\frac{1}{y_*\alpha}(\hat{p}-1)(\hat{\omega} z)'-\frac{\beta}{\epsilon}(1+\frac{\epsilon}{y_*\alpha}(\hat{p}-1))\hat{\omega}'\right)\\
\frac{\epsilon}{y_*\alpha}\hat{p}'(y_*+\epsilon \hat{\omega})(z-\beta)-\epsilon z\hat{p}'
&=\frac{\epsilon^2\hat{p}'}{y_*\alpha}(\frac{z\beta}{\epsilon}-\frac{\beta}{\epsilon}+(\hat{\omega}+1)(z-\beta)),\\
\epsilon(1+\epsilon (\hat{\omega}+1))(\hat{\omega}(z-\beta))'-\epsilon (\hat{\omega}z)'
&=\epsilon^2\left((1+\epsilon(\hat{\omega}+1))\frac{-\beta}{\epsilon}\hat{\omega}'+(\hat{\omega}+1)(\hat{\omega}z)'\right),\\
\epsilon\hat{p}'\frac{\alpha y_*}{1-\beta}\frac{1}{1+\frac{\epsilon}{y_*\alpha}(\hat{p}-1)}-\frac{\epsilon}{z}\hat{p}'
&=\epsilon^2\hat{p}'\left(-\frac{\beta}{\epsilon}\frac{z-1}{(z-\beta)z(1+\frac{\epsilon}{y_*\alpha}(\hat{p}-1))}\right.\\
&\left.\qquad -\frac{\hat{p}-1}{y_*\alpha z(1+\frac{\epsilon}{y_*\alpha}(\hat{p}-1))}\right).
\end{align*}
We may thus define
\begin{align*}
M_5&:=-\left(\frac{1}{y_*\alpha}(\hat{p}-1)(\hat{\omega} z)'-\frac{\beta}{\epsilon}(1+\frac{\epsilon}{y_*\alpha}(\hat{p}-1))\hat{\omega}'\right)-\frac{\hat{p}'}{y_*\alpha}(\frac{z\beta}{\epsilon}-\frac{\beta}{\epsilon}+(\hat{\omega}+1)(z-\beta)),\\
M_6&:=-\left((y_*+\epsilon\hat{\omega})\frac{-\beta}{\epsilon}\hat{\omega}'+(\hat{\omega}+1)(\hat{\omega}z)'\right)\\
&\qquad -\hat{p}'\left(-\frac{\beta}{\epsilon}\frac{z-1}{(z-\beta)z(1+\frac{\epsilon}{y_*\alpha}(\hat{p}-1))}-\frac{\hat{p}-1}{y_*\alpha z(1+\frac{\epsilon}{y_*\alpha}(\hat{p}-1))}\right)-M_4(\hat{p},\hat{\omega}),
\end{align*}
and then we can write
\begin{equation}
P=\epsilon M_5,\qquad \Omega=\epsilon M_6.
\label{diffrewritten2}
\end{equation}
We now let $\zeta=z-1$, and note that as $\zeta\to 0$, we have
\[
\hat{p}=1+\zeta+o(\zeta),\qquad \hat{\omega}=-1+o(\zeta),\qquad \hat{p}'=1+o(1),\qquad \hat{\omega}'=o(1).
\]
We then use these expansions to compute that as $\zeta\to 0$
\begin{multline}
\epsilon M_1=O(\zeta),\qquad \epsilon M_2=-\frac{2\zeta}{y_*\alpha}-\frac{\epsilon}{y_*\alpha}\zeta+2\zeta+o(\zeta),\\
\epsilon M_3=\epsilon \frac{\zeta}{y_*\alpha}+o(\zeta),\qquad \epsilon M_4=\frac{2}{y_*\alpha}\zeta-2 \zeta+o(\zeta).
\label{z1boundaryconds}
\end{multline}
If we substitute these into \eqref{diffrewritten}, we see that $\frac{P-\Omega}{z-1}$ is continuous near $z=1$ and
\begin{equation}
(P-\Omega)(1)=(P-\Omega)'(1)=0.
\label{differenceboundaryconditions}
\end{equation}

We note that as $\zeta\to 0$, we have
\[
M_5=o(1)
\]
so by \eqref{diffrewritten2} $P(1)=0$, which along with \eqref{differenceboundaryconditions} ensures $P(1)=\Omega(1)=0$.

We next consider the behavior near $z=\infty$. In this case, we note that assume $p,\omega\in X_{y_0}$, we have that $\norm{\hat{p},\hat{\omega}}<\infty$, so there exists constants $p_0,p_1,\omega_0,\omega_1$ such that
\[
\hat{p}(z)=p_0+\frac{p_1}{z}+o(\frac{1}{z}),\qquad \hat{\omega}(z)=\frac{\omega_0}{z}+\frac{\omega_1}{z^2}+o(\frac{1}{z^2}),
\]
\[
\hat{p}'(z)=-\frac{p_1}{z^2}+o(\frac{1}{z^2}),\qquad \hat{\omega}'(z)=-\frac{\omega_0}{z^2}-\frac{2\omega_1}{z^3}+o(\frac{1}{z^3}).
\]
Using this, one computes that as $z\to\infty$,
\[
M_5=\frac{p_1}{(1-y_0)}\frac{1}{z}+o(\frac{1}{z}),
\]
\[
M_6=-\frac{y_*\beta \omega_0}{\epsilon z^2}+\frac{\omega_1}{z^2}-\frac{2y_*\alpha(p_0-1-\frac{\beta}{\epsilon})+2p_0\frac{\beta}{\epsilon}}{z^2}+o(\frac{1}{z^2}).
\]

In summary, these computations show we have shown so far that if $(p,\omega)\in X_{y_0}$ then $(P,\Omega)$ satisfies the correct boundary conditions at $z=1,\infty$ for $(P,\Omega)$ to be in $N_{y_0}$. Moreover, it is clear that $(P,\Omega)$ are continuous away from $z=1$. We next show the inequality \eqref{bddnesseq}. As a note, we use ``$\lesssim$" to denote inequalities which are independent of $y_0,\epsilon$.

To show \eqref{bddnesseq}, we first note that if $\norm{(p,\omega)}_{X_{y_0}}\lesssim 1$, then we have that $\norm{\hat{p},\hat{\omega}}_{X_{y_0}}\lesssim 1$. Thus, by making $\epsilon$ sufficiently small, we have that
\[
\abs{\frac{1}{1+\frac{\epsilon}{y_*\alpha}(\hat{p}-1)}}\lesssim 1
\]
and so the denominator causes no problems. Near zero it is useful to note that
\begin{equation}
\sup_{\frac{1}{2}\leq z\leq 2}\abs{\hat{p}}+\abs{\hat{p}'}+\abs{\hat{\omega}}+\abs{\hat{\omega}'}\lesssim 1.
\label{localcontrol}
\end{equation}
We then compute
\[
\frac{P-\Omega}{z-1}=\epsilon \frac{M_1}{z-1}\hat{\omega}'+\epsilon \frac{M_2}{z-1}\hat{p}'+\epsilon \frac{M_3}{z-1}+\epsilon \frac{M_4}{z-1}.
\]
We then note that as, for $i\in 1,2,3,4$ $M_i(1)=0$, we have
\[
\sup_{\frac{1}{2}\leq z\leq 2}\abs{\frac{M_i}{z-1}}\lesssim \sup_{\frac{1}{2}\leq z\leq 2}\abs{M_i'}.
\]
Moreover, for these indices $M_i$ does not contain $\hat{p}'$ or $\hat{\omega}'$, so by \eqref{localcontrol}
\[
\sup_{\frac{1}{2}\leq z\leq 2}\abs{M_i'}\lesssim 1.
\]
This gives
\[
\sup_{\frac{1}{2}\leq z\leq 2}\abs{\frac{P-\Omega}{z-1}}\lesssim \abs{\epsilon}.
\]
Similarly, we have by \eqref{localcontrol}
\[
\abs{M_5}\lesssim 1
\]
so (by writing $\Omega=P-(z-1)\frac{P-\Omega}{z-1}$)
\[
\sup_{\frac{1}{2}\leq z\leq 2}\abs{P}+\abs{\Omega}\lesssim \abs{\epsilon}.
\]
We now consider the regime $z\geq 2$ and note that as $\norm{(\hat{p},\hat{\omega})}\lesssim 1$ we have for some constants $p_0$, $\omega_0$,
\[
\abs{\hat{p}-p_0}\lesssim \frac{1}{z},\qquad \abs{\hat{p}'}\lesssim \frac{1}{z^2},\qquad \abs{\hat{\omega}-\frac{\omega_0}{z}}\lesssim \frac{1}{z^2},\qquad \abs{(z\hat{\omega})'}\lesssim \frac{1}{z^2}.
\]
This gives $\abs{M_5}\lesssim \frac{1}{z}$ and $\abs{M_6}\lesssim \frac{1}{z^2}$, so we have
\[
\sup_{z\geq 2}\abs{zP}+\abs{z^2\Omega}\lesssim \abs{\epsilon}.
\]
Finally, for $z\leq \frac{1}{2}$, we have from $\norm{(\hat{p},\hat{\omega})}_{X_{y_0}}\lesssim 1$ that
\[
\abs{\hat{p}},\abs{\hat{\omega}}\lesssim z^{-1/2},\qquad \abs{\hat{p}'},\abs{\hat{\omega}'}\lesssim z^{-3/2}.
\]
We also note by choosing $\epsilon$ small, 
\[
\abs{\beta}=\frac{\epsilon}{1+\epsilon-y_0} y_0\leq \frac{y_0}{2}\leq \frac{\abs{z}}{2}
\]
So we may also estimate $\abs{\frac{1}{z-\beta}}\lesssim \frac{1}{z}$. We also record $\abs{\frac{\beta}{\epsilon}}\lesssim z$. We thus can compute
\[
\abs{M_5}\lesssim \frac{1}{z},\qquad \abs{M_6}\lesssim \frac{1}{z^3}.
\]
Thus,
\[
\sup_{y_0\leq z\leq \frac{1}{2}} \abs{z P}+\abs{z^3 \Omega}\lesssim \epsilon
\]
concluding the proof of inequality \eqref{bddnesseq}.

We now consider differentiating the nonlinearity with respect to $\hat{p}$ and $\hat{\omega}$. Near $z=1$, we represent the formal derivative of $P-\Omega$ with respect to $p$, by computing
\begin{align*}
\left(\pdv{M_1}{a},\pdv{M_1}{b}\right)&=\left(-(z-\beta)\frac{1}{y_*\alpha},(z-\beta)\right)\\
\left(\pdv{M_2}{a},\pdv{M_2}{b}\right)&=\left(\frac{1}{y_*\alpha}(1-\frac{z-1}{z-\beta})\frac{1}{(1+\frac{\epsilon}{y_*\alpha}(a-1))^2},1+\frac{z-1}{y_*\alpha}\right)\\
\left(\pdv{M_3}{a},\pdv{M_3}{b}\right)&=\left(-\frac{b}{y_*\alpha},-\frac{a-1}{y_*\alpha}+2b+1\right)\\
\left(\pdv{M_4}{a},\pdv{M_4}{b}\right)&=\left(2\frac{y_*\alpha(b+1)}{(z-\beta)^2}-\frac{2\beta}{\epsilon}\frac{z^2-2z+\beta}{z^2(z-\beta)^2},\frac{2y_*\alpha(a-1)}{(z-\beta)^2}-\frac{2\beta}{\epsilon}\frac{(z-1)(2z-\beta-\beta z)}{z^2(z-\beta)^2}\right).
\end{align*}
This then gives the linear operators (partial derivatives of $M_i$ evaluated at $\hat{p}_0,\hat{\omega}_0$)
\begin{align*}
\pdv{P-\Omega}{p}:=\epsilon M_2 \dv{z}+\epsilon \hat{\omega}_0'\pdv{M_1}{a}+\epsilon \hat{p}_0'\pdv{M_2}{a}+\epsilon \pdv{M_3}{a}+\epsilon \pdv{M_4}{a}\\
\pdv{P-\Omega}{\omega}:=\epsilon M_1\dv{z}+\epsilon \hat{\omega}_0'\pdv{M_1}{b}+\epsilon \hat{p}_0'\pdv{M_2}{b}+\epsilon \pdv{M_3}{b}+\epsilon \pdv{M_4}{b}.
\end{align*}
We then note that if $\norm{\hat{p},\hat{\omega}}_{X_{y_0}}\lesssim 1$, we have that 
\begin{multline*}
\sup_{\frac{1}{2}\leq z\leq 2}\abs{\frac{M_2}{z-1}}+\abs{\hat{\omega}_0 \pdv{M_1}{a}}+\abs{\hat{p_0}'\pdv{M_2}{a}}+\abs{\pdv{M_3}{a}}+\abs{\pdv{M_4}{a}}\\
+\abs{\frac{M_1}{z-1}}+\abs{\hat{\omega}_0 \pdv{M_1}{b}}+\abs{\hat{p_0}'\pdv{M_2}{b}}+\abs{\pdv{M_3}{b}}+\abs{\pdv{M_4}{b}}\lesssim 1.
\end{multline*}
so for any $(p,\omega)\in X_{y_0}$,
\begin{align*}
\sup_{\frac{1}{2}\leq z\leq 2} \abs{\pdv{P-\Omega}{p}p}+\abs{\pdv{P-\Omega}{\omega} \omega}
&\lesssim \epsilon \sup_{\frac{1}{2}\leq z\leq 2}\abs{p'}+\abs{\frac{p}{z-1}}+\abs{\omega'}+\abs{\frac{\omega}{z-1}}\\
&\lesssim \epsilon \sup_{\frac{1}{2}\leq z\leq 2}\abs{p'}+\abs{\omega}'\\
&\lesssim \epsilon \norm{(p,\omega)}_{X_{y_0}}.
\end{align*}
We next estimate the derivative away from $z=1$. This is done by formally differentiating the formulae for $M_5,M_6$. When we do this, we obtain
\begin{align*}
\pdv{P}{p}&=\epsilon\left(-\frac{1}{y_*\alpha}(\frac{(z-1)\beta}{\epsilon}+(\hat{\omega}_0+1)(z-\beta))\dv{z}+\frac{(\hat{\omega}_0(z-\beta))'}{y_*\alpha}\right)\\
\pdv{P}{\omega}&=\epsilon\left((\frac{z-\beta}{y_*\alpha}(\hat{p}_0-1)-\frac{\beta}{\epsilon})\dv{z}+\frac{\hat{p}_0-1}{y_*\alpha}-\frac{(z-\beta)\hat{p}_0'}{y_*\alpha}\right).\\
\pdv{\Omega}{p}&=\epsilon\left(\left(\frac{\beta}{\epsilon}\frac{z-1}{(z-\beta)z(1+\frac{\epsilon}{y_*\alpha}(\hat{p}_0-1))}+\frac{\hat{p}_0-1}{y_*\alpha z(1+\frac{\epsilon}{y_*\alpha}(\hat{p}_0-1))}\right)\dv{z}-\right.\\
&\qquad \left.-\hat{p}_0'\left(\frac{\beta(z-1)}{y_*\alpha(z-\beta)z(1+\frac{\epsilon}{y_*\alpha}(\hat{p}_0-1))^2}+\frac{1+(-1+\frac{\epsilon}{y_*\alpha}(\hat{p}-1))}{y_*\alpha z(1+\frac{\epsilon}{y_*\alpha}(\hat{p}_0-1))^2}\right)-\pdv{M_4}{a}\right)\\
\pdv{\Omega}{\omega}&=-\epsilon\left((-\frac{\beta}{\epsilon}+(z-\beta)(\hat{\omega}_0+1))\dv{z}+(\hat{\omega}_0(z-\beta))'+\hat{\omega}_0+1-\frac{\hat{p}_0'}{y_*\alpha}(z-\beta)\right).
\end{align*}
Thus, under the hypotheses of the lemma, for $(p,\omega)\in X_{y_0}$, for $z\geq 2$, we have the desired bounds
\[
\abs{\pdv{P}{p}p}+\abs{\pdv{P}{\omega}\omega}\lesssim \epsilon(\abs{z p'}+\abs{z \omega'}+\abs{\frac{p}{z}}+\abs{\omega})\lesssim \frac{\epsilon }{z}\norm{(p,\omega)}_{X_{y_0}}
\]
and (noting that $\abs{z\omega'+\omega}\lesssim \abs{\frac{1}{z^2} D_{\infty} z\omega}$)
\[
\abs{\pdv{\Omega}{p}p}+\abs{\pdv{\Omega}{\omega} \omega}\lesssim \epsilon(\abs{\frac{p'}{z}}+\abs{\frac{p}{z^2}}+\frac{1}{z^2}\abs{D_{\infty}\omega}+\frac{1}{z^2}\abs{z \omega})\lesssim \frac{\epsilon}{z^2}\norm{(p,\omega)}_{X_{y_0}}.
\]
It finally remains to examine the region $y_0\leq z\leq \frac{1}{2}$. In this case, we have again
\[
\abs{\pdv{P}{p}p}+\abs{\pdv{O}{\omega}\omega}\lesssim \epsilon(\abs{z^{1/2}p'}+\abs{z^{1/2}\omega'}+\abs{z^{-1/2}p}+\abs{z^{-1/2}\omega})\lesssim \frac{\epsilon}{z}\norm{(p,\omega)}_{X_{y_0}}
\]
and
\[
\abs{\pdv{\Omega}{p}p+\pdv{\Omega}{\omega}\omega}\lesssim \epsilon(\abs{z^{-3/2} p'}+\abs{z^{-5/2} p}+\abs{z^{1/2}\omega'}+\abs{z^{-1/2}\omega })\lesssim \frac{\epsilon}{z^3}\norm{(p,\omega)}_{X_{y_0}}.
\]

These estimates prove the first part of \eqref{derivativebddnesseq}. As the coefficients of $N$ vary smoothly with respect to $\epsilon$, the first statement follows from the boundedness of $N$, and the fact that when we differentiate the representations \eqref{diffrewritten} and \eqref{diffrewritten2}, we can no longer factor out a factor of $\epsilon$, resulting in the second inequality of \eqref{derivativebddnesseq} not having a factor of $\epsilon$ in front. To show that the formal derivatives truly approximate $N$ to first order, it suffices to show
\begin{multline*}
\norm{N(\epsilon_1,p_1,\omega_1)-N(\epsilon_0,p_0,\omega_0)-(\pdv{N}{\epsilon},\pdv{N}{p},\pdv{N}{\omega})|_{(\epsilon_0,p_0,\omega_0)}\cdot (\epsilon_1-\epsilon_0,p_1-p_0,\omega_1-\omega_0)}_{N_{y_0}}\\
\leq C_2\left(\abs{\epsilon_1-\epsilon_0}^2+\abs{\epsilon_1-\epsilon_0}\norm{(p_1,\omega_1)-(p_0,\omega_0)}_{X_{y_0}}\right.\\
\left.+\max(\abs{\epsilon_0},\abs{\epsilon_1})\norm{(p_1,\omega_1)-(p_0,\omega_0)}_{X_{y_0}}^2\right).
\label{secondorderestimate}
\end{multline*}
which follows from straightforward estimates. Once this is done, \eqref{derivativelipschitz1} and \eqref{derivativelipschitz2} follow similarly. 
\end{proof}

\subsection{Proof of Lemma \ref{extconstructionlemmapre}}
\label{extconstructionlemmapreproof}
This is a classical fixed point argument.
\begin{proof}
We begin by proving (a). We let $(p_0,\omega_0)=(0,0)$. We suppose 
\begin{equation}
\abs{\epsilon}\leq \min(\frac{y_0^{1/2}}{C_2(\frac{1}{2})C_1}\frac{1}{10},y_0^{1/2}C_0(\frac{1}{2})).
\label{epsiloncondition}
\end{equation}
For $n\geq 1$, we define $(P_n,\Omega_n)=N(\epsilon, p_{n-1},\omega_{n-1})$, $(p_n,\omega_n)=R(P_n,\Omega_n)$. We inductively suppose $\norm{(p_n,\omega_n)}_{X_{y_0}}\leq y_0^{-1/2}\epsilon C_2(\frac{1}{2})C_1$. If this is true for $n=k-1$, then $\norm{(p_{k-1},\omega_{k-1})}_{X_{y_0}}<\frac{1}{2}$, so by Lemma \ref{nonlinearlemma}.a we have $\norm{(P_k,\Omega_k)}_{N_{y_0}}\leq \abs{\epsilon}C_2(\frac{1}{2})$. By Lemma \ref{linearlemma}, $\norm{(p_k,\omega_k)}\leq \abs{\epsilon}C_1C_2(\frac{1}{2})y_0^{-1/2}$ as desired.

In addition, for $k\geq 1$, we have that
\begin{align*}
\norm{(p_{k+1},\omega_{k+1})-(p_k,\omega_k)}_{X_{y_0}}&\leq C_1y_0^{-1/2}\norm{N(\epsilon, p_{k},\omega_{k})-N(\epsilon, p_{k-1},\omega_{k-1})}_{N_{y_0}}\\
&\leq C_1y_{0}^{1/2}C_2(\frac{1}{2})\abs{\epsilon}\norm{(p_{k},\omega_k)-(p_{k-1},\omega_{k-1})}_{X_{y_0}}\\
&\leq \frac{1}{2}\norm{(p_{k},\omega_k)-(p_{k-1},\omega_{k-1})}_{X_{y_0}},
\end{align*}
where to get from the first line to the second we used the uniform bound on $\pdv{N}{p},\pdv{N}{\omega}$ and the fundamental theorem of calculus. Thus, this iteration converges to a limit, $(p[\epsilon],\omega[\epsilon])$, with the desired bound. Uniqueness follows similarly from the contractive estimate, concluding the proof of part (a). We note that if $\epsilon=0$, as $N(0,0,0)=0$, we obtain $(p[0],\omega[0])=0$.

To prove part (b), we consider $\epsilon_0,\epsilon_1$ satisfying \eqref{epsiloncondition}. We let $(p_0,\omega_0)=(p[\epsilon_0],\omega[\epsilon_0])$, $(p_1,\omega_1)=(p[\epsilon_1],\omega[\epsilon_1])$, which are well defined by part (a) and satisfy
\[
(p_0,\omega_0)=RN(\epsilon_0,p_0,\omega_0),\qquad (p_1,\omega_1)=RN(\epsilon_1,p_1,\omega_1).
\]
If we subtract these identities we obtain, by the bounds on the partial derivatives of $N$,
\begin{align*}
\norm{(p_1,\omega_1)-(p_0,\omega_0)}_{X_{y_0}}&=\norm{R \left(N(\epsilon_1, p_1,\omega_1)-N(\epsilon_0, p_0,\omega_0)\right)}_{X_{y_0}}\\
&\leq C_1y_0^{-1/2}\norm{N(\epsilon_1,p_1,\omega_1)-N(\epsilon_0,p_0,\omega_0)}\\
&\leq C_1y_0^{-1/2}C_2\abs{\epsilon_1-\epsilon_0}+C_1y_0^{-1/2}C_2\max(\abs{\epsilon_1},\abs{\epsilon_2}) \norm{(p_1,\omega_1)-(p_0,\omega_0)}_{X_{y_0}}\\
&\leq C_1y_0^{-1/2}C_2(\epsilon_1-\epsilon_0)+\frac{1}{2}\norm{(p_1,\omega_1)-(p_0,\omega_0)}_{X_{y_0}}.
\end{align*}
By absorbing the second term into the right-hand side we conclude
\begin{equation}
\norm{(p_1,\omega_1)-(p_0,\omega_0)}_{X_{y_0}}\leq 2 C_1 C_2y_0^{-1/2}\abs{\epsilon_1-\epsilon_0}.
\label{lipschitzbound}
\end{equation}
Next, for each $\epsilon_0$, we consider the linearization of \eqref{integraleq}, namely
\begin{equation}
(p_{lin},\omega_{lin})=R\pdv{N}{\epsilon}|_{(\epsilon,p[\epsilon],\omega[\epsilon])}+R(\pdv{N}{p},\pdv{N}{\omega})|_{(\epsilon,p[\epsilon],\omega[\epsilon])}\cdot(p_{lin},\omega_{lin}).
\label{linearrecursion}
\end{equation}
and note that 
\[
\norm{R\pdv{N}{\epsilon}}_{X_{y_0}}\leq C_2(\frac{1}{2})C_1 y_0^{-1/2},\qquad \norm{R(\pdv{N}{p},\pdv{N}{\omega})|_{(\epsilon,p[\epsilon],\omega[\epsilon])}}_{X_{y_0}\to X_{y_0}}\leq C_2(\frac{1}{2})C_1 y_0^{-1/2}\epsilon \leq \frac{1}{2}
\]
so we can iteratively solve \eqref{linearrecursion} resulting in a solution satisfying
\begin{equation}
\norm{(p_{lin},\omega_{lin})}_{X_{y_0}}\leq 2C_2(\frac{1}{2})C_1y_0^{-1/2}.
\label{linearizedbound}
\end{equation}
We next define
\[
(p_{diff},\omega_{diff})=(p_1,\omega_1)-(p_0,\omega_0)-(\epsilon_1-\epsilon_0)(p_{lin}[\epsilon_0],\omega_{lin}[\epsilon_0]).
\]
We note that
\begin{align*}
(p_{diff},\omega_{diff})&=R\left(N(p_1,\omega_1)-N(p_0,\omega_0)-(\pdv{N}{\epsilon},\pdv{N}{p},\pdv{N}{\omega})(\epsilon_1-\epsilon_0,(\epsilon_1-\epsilon_0)p_{lin}[\epsilon_0],(\epsilon_1-\epsilon_0)\omega_{lin}[\epsilon_0])\right)\\
&=R(N(p_1,\omega_1)-N(p_0,\omega_0)-(\pdv{N}{\epsilon},\pdv{N}{p},\pdv{N}{\omega})|_{(\epsilon_0,p_0,\omega_0)}(\epsilon_1-\epsilon_0,p_1-p_0,\omega_1-\omega_0))\\
&+R(\pdv{N}{\epsilon},\pdv{N}{p},\pdv{N}{\omega})|_{(\epsilon_0,p_0,p_1)}(0,p_{diff},\omega_{diff})
\end{align*}
We use \eqref{derivativelipschitz1}, \eqref{derivativelipschitz2} and \eqref{lipschitzbound} to control the first term and absorb the second term on the left hand side to obtain
\[
\norm{(p_{diff},\omega_{diff})}_{N_{y_0}}\leq 10 (C_1C_2(\frac{1}{2}) y_0^{-1/2})^2\abs{\epsilon_1-\epsilon_0}^2.
\]
Thus, we see that $(p[\epsilon],\omega[\epsilon])$ is differentiable with derivative $(p_{lin}[\epsilon],\omega_{lin}[\epsilon])$.

To show continuity of $(p_{lin}[\epsilon],\omega_{lin}[\epsilon])$ with respect to $\epsilon$, we express
\begin{align*}
&(p_{lin}[\epsilon_1],\omega_{lin}[\epsilon_1])-(p_{lin}[\epsilon_0],\omega_{lin}[\epsilon_0])\\
&=R\left((\pdv{N}{\epsilon},\pdv{N}{p},\pdv{N}{\omega})_{(\epsilon_1,p_1,\omega_1)}\cdot (1,p_{lin}[\epsilon_1],\omega_{lin}[\epsilon_1])-(\pdv{N}{\epsilon},\pdv{N}{p},\pdv{N}{\omega})_{(\epsilon_0,p_0,\omega_0)}\cdot (1,p_{lin}[\epsilon_0],\omega_{lin}[\epsilon_0])\right)\\
&=R\left((\pdv{N}{\epsilon},\pdv{N}{p},\pdv{N}{\omega})_{(\epsilon_1,p_1,\omega_1)}\cdot (0,p_{lin}[\epsilon_1]-p_{lin}[\epsilon_0],\omega_{lin}[\epsilon_1]-\omega_{lin}[\epsilon_0])\right)\\
&+R\left(\left((\pdv{N}{\epsilon},\pdv{N}{p},\pdv{N}{\omega})_{(\epsilon_1,p_1,\omega_1)}-(\pdv{N}{\epsilon},\pdv{N}{p},\pdv{N}{\omega})_{(\epsilon_0,p_0,\omega_0)}\right)\cdot (1,p_{lin}[\epsilon_0],\omega_{lin}[\epsilon_0])\right)
\end{align*}
We bound the first term using the control of the derivative of $\pdv{N}{p},\pdv{N}{\omega}$ and bound the second term using \eqref{derivativelipschitz1}, \eqref{derivativelipschitz2}, \eqref{lipschitzbound} to obtain
\begin{multline*}
\norm{(p_{lin}[\epsilon_1],\omega_{lin}[\epsilon_1])-(p_{lin}[\epsilon_0],\omega_{lin}[\epsilon_0])}_{X_{y_0}}\leq \frac{1}{2}\norm{(p_{lin}[\epsilon_1],\omega_{lin}[\epsilon_1])-(p_{lin}[\epsilon_0],\omega_{lin}[\epsilon_0])}_{X_{y_0}}\\
+5(C_2(\frac{1}{2})C_1y_0^{-1/2})^2\abs{\epsilon_1-\epsilon_0}.
\end{multline*}
Absorbing the first term shows the desired continuity of the derivative, namely
\[
\norm{(p_{lin}[\epsilon_1],\omega_{lin}[\epsilon_1])-(p_{lin}[\epsilon_0],\omega_{lin}[\epsilon_0])}_{X_{y_0}}\leq 10 (C_1C_2(\frac{1}{2})y_0^{-1/2})^2\abs{\epsilon_1-\epsilon_0}.
\]
\end{proof}

\section{Deferred proofs from Section \ref{intsection}}
\label{intleftovers}
\subsection{Proof of Lemma \ref{kernelasymptotics}}
\label{kernelasymptoticsproof}
\begin{proof}
We let $v$ denote $v_1$ or $v_2$. For $y>0$ we then rewrite $Hv=0$ as $\tilde{H}v=f$, where
\[
f:=2(e^Q-\frac{1}{y^2})v
\]
We now choose the following basis of homogeneous solutions for $\tilde{H}$,
\[
\tilde{v}_1:=\frac{\cos(\fr \log(y))}{y^{1/2}},\qquad \tilde{v}_2:=\frac{\sin(\fr \log(y))}{y^{1/2}}.
\]
Their Wronskian is $\frac{\fr}{y^2}$. We introduce now the following integral operator so that $\tilde{H}\circ \tilde{S}_{\infty}=\id$:
\[
\tilde{S}_{\infty}(f):=\tilde{v}_1\int_{y}^{\infty} \frac{(y')^2}{\fr}\tilde{v}_2 fdy'-\tilde{v}_2\int_{y}^{\infty}\frac{(y')^2}{\fr}\tilde{v}_1 fdy'.
\]
We now recall that we have $f=O(y^{-3})$ as $y\to \infty$, so $\tilde{S}_{\infty}\circ \tilde{H}v=O(\frac{1}{y})$ as $y\to \infty$. We then note that
\[
v-\tilde{S}_{\infty}\circ \tilde{H}v\in \ker \tilde{H},
\]
so we have for some constants $a_1,a_2$
\begin{align*}
v&=a_1\tilde{v}_1+a_2\tilde{v}_2+\tilde{S}_{\infty}\circ \tilde{H}v\\
&=a_1\tilde{v}_1+a_2\tilde{v}_2+O(\frac{1}{y}).
\end{align*}
This, up to a linear change of variables, yields the claimed asymptotic expansions. To see why $c_3,c_4\neq 0$, we have by \eqref{kernelasymptoticseq}, $W=\frac{c_3c_4\fr \sin(d_1-d_2)}{y^2}$, and also $W=\frac{1}{y^2}$.
\end{proof}

\subsection{Proof of Lemma \ref{linlem}}
\label{linlemproof}
\begin{proof}
We first consider the operator $S_r$. We make the qualitative observation that if $f(x)$ is continuous then $S_rf$ is $C^2$ and $S_rf(0)=0$. We next assume that $f\in \calG_{y_1}$ and $\norm{f}_{\calG_{y_1}}\leq 1$. To show that $S_rf\in Z_{y_1}$ we argue differently for $0\leq y\leq 1$ and $1\leq y\leq y_1$. For $y\leq 1$, we note that as $y\to 0$, we have (uniformly in $r$),
\begin{equation}
v_{1,r}=O(y),\qquad v_{2,r}=O(\frac{1}{y})
\label{v0asymptots}
\end{equation}
and so we have (all inequalities independent of $y_0,\lambda_0,r$)
\[
\sup_{0\leq y\leq 1}\abs{Sf}\lesssim \sup_{0\leq y\leq 1}\abs{f}\lesssim \norm{f}_{\calG_{y_1}}.
\]
Next, for $y\geq 1$, we recall that the homogeneous solutions $v_{1,r},v_{2,r}$ have the following asymptotics as $y\to \infty$,
\begin{equation}
v_{1,r},v_{2,r}=O(\frac{1}{y^{1/2}}).
\label{vinftyasym}
\end{equation}
We then use the expression
\begin{align*}
S_rf(y)&=\frac{1}{r^3}\left(\left(\int_0^1fv_{2,r}(y')^2dy'\right)v_{1,r}-\left(\int_0^1 fv_{1,r} (y')^2 dy'\right)v_{2,r}\right)\\
&+\frac{1}{r^3}\left(\left(\int_1^yfv_{2,r}(y')^2dy'\right)v_{1,r}-\left(\int_1^y fv_{1,r} (y')^2 dy'\right)v_{2,r}\right)\\
&=S_1(y)+S_2(y).
\end{align*}
We then use the asymptotics of the fundamental solutions near zero and infinity to bound
\[
S_1(y)\leq \ev{y}^{-1/2}\norm{f}_{L^{\infty}[0,1]}\lesssim \ev{y}^{3/2}\norm{ f}_{\calG_{y_1}},
\]
and the asymptotics of the fundamental solutions near infinity to bound
\[
\qquad S_2(y)\lesssim \ev{y}^{3/2}\norm{\ev{y}^{1/2} f}_{L^{\infty}[1,y_1]}\lesssim \ev{y}^{3/2}\norm{ f}_{G_{y_1}},
\]
giving the desired pointwise control of $f$. To show control of the derivatives, we note that when we differentiate the definition of $S$ we get a cancellation, so
\begin{align*}
(Sf)'=\frac{1}{r^2}\left(\left(\int_0^y fv_{2,r} (y')^2 dy'\right)v_{1,r}'-\left(\int_0^y fu_2 (y')^2 dy'\right)v_2,r'\right).
\end{align*}
This allows one to control $(Sf)'$ using only pointwise control of $f$. Similarly,
\begin{align*}
(Sf)''&=-(HS)f-\frac{2}{y}\partial_y(Sf)-2e^Q(Sf)\\
&=-f-\frac{2}{y}\partial_y(Sf)-2e^Q(Sf).
\end{align*}
As the RHS can be controlled with pointwise control of $f$, we see that $\partial_r^2(Sf)$ can also be controlled with pointwise control of $f$. Using these observations, the control of the derivatives of $S$ works in the same way as the pointwise control. Thus, $S_r$ is a bounded as a map from $\calG_{y_1}$ to $Z_{y_1}$.

We then define the formal derivative of $S$ with respect to $r$ as
\begin{multline*}
S_{lin,r}f:=-\frac{3}{r}S_rf+\frac{1}{r^3}\left(\left(\int_0^y f \pdv{v_{2,r}}{r}\ (y')^2 dy'\right)v_{1,r}+\left(\int_0^y f v_{2,r}\ (y')^2 dy'\right)\pdv{v_{1,r}}{r}+\right.\\
\left.-\left(\int_0^y f \pdv{v_{1,r}}{r} (y')^2 dy'\right)v_{2,r}-\left(\int_0^y f v_{1,r} (y')^2 dy'\right)\pdv{v_{2,r}}{r}\ \right)
\end{multline*}
where we note that
\[
\pdv{v_{i,r}}{r}(y)=-\frac{y}{r^2}v_i'(\frac{y}{r}).
\]
and as we may differentiate the asymptotics of $v_i$, we see that \eqref{v0asymptots} and \eqref{vinftyasym} continue to hold for $\pdv{v_{i,r}}{r}$, so the same arguments showing the boundedness of $S$ show that this operator is bounded. Straightforward difference estimates prove that
\[
\norm{S_{r_1}-S_{r_0}-(r_1-r_0)S_{lin,r_0}}_{L(\calG_{y_1},Z_{y_1})}\lesssim \abs{r_1-r_0}^2
\]
and
\[
\norm{S_{lin,r_1}-S_{lin,r_0}}_{L(\calG_{y_1},Z_{y_1})}\lesssim \abs{r_1-r_0}
\]
which show that $S_{lin,r}$ is the derivative of $S_r$ and it is continuous with respect to $r$ as desired.

We next consider $T_r$. In this case, it is clear that if $f\in C^1$ and $f(0)=0$ then $T_rf\in C^2$ and $T_rf(0)=(T_rf)'(0)=0$. If one supposes $\norm{f}_{\calF_{y_1}}\leq 1$, then for $0\leq y\leq 1$, we use that $e^{-Q_r}$ goes to $r^2$ at zero to compute
\[
T_rf(y)=-\frac{1}{r^2 y^2e^Q}\int_0^y(y')^2 fdy'\lesssim \norm{f}_{L^{\infty}[0,1]}\lesssim \norm{f}_{\calF_{y_1}}.
\]
For $y\geq 1$, we use that $e^{-Q}=y^2+O(y^{3/2})$ as $y\to\infty$ to compute
\begin{align*}
T_rf(y)&=-\frac{1}{r^2 y^2e^Q}\int_0^1 (y')^2fdy'-\frac{1}{r^2 y^2e^Q}\int_1^y(y')^2 fdr'\\
&\lesssim \norm{f}_{L^{\infty}[0,1]}+\ev{y}^{5/2}\norm{\ev{y}^{1/2}f}_{L^{\infty}[1,y_1]}\lesssim \ev{y}^{5/2}\norm{f}_{\calF_{y_1}}.
\end{align*}
To control a derivative, we use the identity
\[
(T_rf)'=\left(e^Q(\frac{1}{e^Q})'-\frac{2}{y}\right)T_rf-\frac{1}{e^Q}f.
\]
and note that the right side involves no derivatives of $f$. This allows one to control $1$ derivative of $T_rf$ using pointwise control of $f$. Control of the second derivative of $T_rf$ requires pointwise control of the first derivative of $f$, which is supplied by the norm of $\calF_{y_1}$. From this one obtains that $T_r$ is bounded from $\calF_{y_1}$ to $Y_{y_1}$. One can then compute the formal derivative of $T_r$ with respect to $r$ to be
\[
T_{lin,r}f=-\frac{2}{r}T_rf+\frac{1}{r^2y^2e^{Q_r}}\pdv{Q_r}{r}\int_0^y f(y')(y')^2 dy'
\]
where one can compute
\[
\pdv{Q_r}{r}(y)=-\frac{y}{r^2}Q'(\frac{y}{r})-\frac{2}{r}.
\]
It is then straightforward to check that this operator is bounded, approximates $T_r$ to first order in $r$ and is Lipschitz with respect to $r$ valued in $L(\calF_{y_1},Y_{y_1})$.

Finally, to show that $K_r$ has the desired property, we note that it suffices to check that $L_r\colon Z_{y_1}\to \calF_{y_1}$ is bounded, and has the desired regularity when $r$ is varied. It is clear that, if $f$ is $C^1$, $f(0)=0$ then $Tf$ is $C^2$ and $Tf(0)=(Tf)'(0)=0$. Finally, if $f(0)=0$ and $f$ is $C^2$, then $L_rf$ is $C^1$. The boundedness follows from the asymptotics of $Q_r$. Finally, one can compute the formal derivative
 \[
L_{lin,r}w=\frac{2}{r}L_{lin,r}w+r^2\left((2+y\p_y)(e^{Q_r}\pdv{Q_r}{r}w)+\div(u_re^{Q_r}\pdv{Q_r}{r}w)+\div(\pdv{u_r}{r}e^{Q_r} w)\right).
\]
Once again, straightforward computations show that $L_{lin,r}$ is a bounded map from $Z_{y_1}$ to $\calF_{y_1}$, $L_{lin,r}$ approximates $L_r$ to first order and $L_{lin,r}$ is Lipschitz with respect to $r$. 
\end{proof}

\subsection{Proof of Lemma \ref{nonlinlemma}}
\label{nonlinlemmaproof}
\begin{proof}
We note that qualitatively if $w,u\in C^2([0,y_0])$ and $w(0)=u(0)=u'(0)=0$ then $G_{r}(w,u)\in C^1([0,y_1])$, $G_r(w,u)\in C^0([0,y_1])$ and $F_{r}(w,u)=0$. As a first step in showing boundedness of the linear operator, we define
\begin{align*}
F_{r,1}(w,u)&=\lambda_0^2r^2(2+y\partial_y)(w^2\int_0^1 (1-t)e^{Q_r+t\lambda_0^2 w}dt),\\
F_{r,2}(w,u)&:=\lambda_0^2r^2 \div(u w\int_0^1 e^{Q_r+t\lambda_0^2 w}dt),\\
F_{r,3}(w,u)&:=\lambda_0^2r^2 \div(u_r w^2\int_0^1 (1-t)e^{Q_r+t\lambda_0^2 w}dt),\\
G_{r,1}(w,u)&=r^2\div((y+u_r)\partial_y u_r),\\
G_{r,2}(w,u)&=2\lambda_0^2r^2 w^2\int_0^1(1-t)e^{Q_r+t\lambda_0^2 w}dt,\\
G_{r,3}(w,u)&=\lambda_0^2r^2\div \left(u\partial_y u_r+(y+u_r)\partial_y u+\lambda_0^2 u\partial_y u\right),
\end{align*}
so that we can express the nonlinearities as
\[
F_{r}=F_{r,1}+F_{r,2}+F_{r,3}\qquad \text{and}\qquad G_{r}=G_{r,1}+G_{r,2}+G_{r,3}.
\]
We now assume that $\norm{w}_{Z_{y_1}},\norm{u}_{Y_{y_1}}\lesssim 1$ and discuss bounding $F_{r,i},G_{r_i}$. As a first step, we note that as $y\leq \frac{y_0}{\lambda_0}$, and $\lambda\leq \frac{y_0}{10}$, we have $\ev{y}\lesssim \frac{y_0}{\lambda_0}$, or in other words
\begin{equation}
\lambda_0\lesssim y_0\ev{y}^{-1}.
\label{lambda0bound}
\end{equation}
Similarly, by \eqref{Qrefinedasymptot}, \eqref{ustarasymptot} we have
\begin{equation}
\abs{e^{Q_r},y\partial_y (e^{Q_r})}\lesssim \ev{y}^{-2},\qquad \abs{u_r}\lesssim \ev{y}^{1/2}
\label{finalbounds}
\end{equation}
These pointwise estimates allow us to conclude
\begin{align*}
\abs{F_1}&\lesssim y_0^2\ev{y}^{-1},\qquad \abs{F_2}\lesssim y_0^2\ev{y}^{-1},\qquad \abs{F_3}\lesssim y_0^2\ev{y}^{-3/2},\\
\abs{G_1}&\lesssim \ev{y}^{-1/2},\qquad \abs{G_2}\lesssim y_0^2\ev{y}^{-1},\qquad \abs{G_3}\lesssim y_0^2\ev{y}^{-1/2}.
\end{align*}
These bounds, along with \eqref{lambda0bound}, show that $\abs{F},\abs{G}\lesssim \ev{y}^{1/2}$ as desired. We can similarly see that $\abs{\partial_y F}\lesssim \ev{y}^{-1/2}$, finishing the proof of (a).

To show differentiability with respect to $r$, one formally computes
\begin{align*}
\pdv{F_{r,1}}{r}|_{r_0}&=\lambda_0^2 r_0(2+y\p_y)(w^2(2+r_0\pdv{Q_r}{r}|_{r_0})\int_0^1(1-t)e^{Q_{r_0}+t\lambda_0^2 w}dt,\\
\pdv{F_{r,2}}{r}|_{r_0}&=\lambda_0^2 r_0\div(u w (2+r_0\pdv{Q_r}{r}|_{r_0})\int_0^1e^{Q_r+t\lambda_0^2 w}),\\
\pdv{F_{r,3}}{r}|_{r_0}&=\lambda_0^2 r_0\div(((2+r_0\pdv{Q_r}{r}|_{r_0})u_{r_0}+r_0\pdv{u_r}{r}|_{r_0})w^2\int_0^t(1-t)e^{Q_r+t\lambda_0^2 w}dt),\\
\pdv{G_{r,1}}{r}|_{r_0}&=r_0\div(2(y+u_{r_0})\p_y u_{r_0}+r_0\pdv{u_{r}}{r}|_{r_0}\p_y u_{r_0}+r_0(y+u_{r_0})\p_y \pdv{u_r}{r}|_{r_0}),\\
\pdv{G_{r,2}}{r}|_{r_0}&=2\lambda_0^2 r_0w^2(2+r_0\pdv{Q_r}{r}|_{r_0})\int_0^1(1-t)e^{Q_r+t\lambda_0^2 w}dt,\\
\pdv{G_{r,3}}{r}|_{r_0}&=\lambda_0^2r\div(2(u \p_y u_{r_0}+(y+u_{r_0})\p_y u+\lambda_0^2 u\p_y u)+r_0u\p_y \pdv{u_r}{r}|_{r_0}+r_0\pdv{u_r}{r}|_{r_0}\p_y u_{r_0}).
\end{align*}
From these formulas it is then straightforward to deduce \eqref{derivativebounda}.

To compute the derivatives of $F_r,G_r$ with respect to $w$ and $u$, it is convenient to express
\begin{align*}
F_r(w,u)&=\frac{r^2}{\lambda_0^2}\left((2+y\p_y)(e^{Q_r+\lambda_0^2 w})+\div(e^{Q_r+\lambda_0^2 w}(u_r+\lambda_0^2 u))\right)-L_r w-J_r u,\\
G_r(w,u)&=r^2\div((y+u_r+\lambda_0^2 u)\p_y(u_r+\lambda_0^2 u))+\frac{r^2}{\lambda_0^2}\Delta(Q_r+\lambda_0^2 w)+\frac{r^2}{\lambda_0^2}2e^{Q_r+\lambda_0^2 w}+H_rw.
\end{align*}
from which one computes the formal derivatives
\begin{align*}
\pdv{F_r}{w}|_{w_0,u_0}w&=\lambda_0^2 r^2(2+y\p_y)(w_0(\int_0^1 e^{Q_r+t\lambda_0^2 w_0} dt) w)\\
&\qquad +\lambda_0^2 r^2\div\left((\int_0^1 e^{Q_r+t\lambda_0^2 w_0} dt) w_0u_r w+e^{Q_r+\lambda_0^2 w}u_0 w\right),\\
\pdv{F_r}{u}|_{w_0,u_0}u&=\lambda_0^2 r^2\div((\int_0^1 e^{Q_r+t\lambda_0^2 w} dt) w_0 u_0),\\
\pdv{G_r}{w}|_{w_0,u_0}w&=\lambda_0^2r^2 2(\int_0^1 e^{Q_r+t\lambda_0^2 w}dt) w_0 w,\\
\pdv{G_r}{u}|_{w_0,u_0}u&=\lambda_0^2 r^2 \div((y+u_r+\lambda_0^2 u_0)\p_y u+u\p_y(u_r+\lambda_0^2 u_0)).
\end{align*}
We note that each term contains a factor of $\lambda_0^2$, and if we use \eqref{lambda0bound} and \eqref{finalbounds} to estimate the rest of the terms, we see that if $\abs{w},\abs{w_0}\lesssim \ev{y}^{3/2}$, $\abs{u},\abs{u_0}\lesssim \ev{y}^{5/2}$,
\[
\abs{\pdv{F_r}{w}|_{w_0,u_0}w}\lesssim y_0^2 \ev{y}^{-1},\qquad \abs{\pdv{F_r}{u}|_{w_0,u_0}u}\lesssim y_0^2\ev{y}^{-1}
\]
\[
\abs{\pdv{G_r}{w}|_{w_0,u_0}w}\lesssim y_0^2 \ev{y}^{-1},\qquad \abs{\pdv{G_r}{u}|_{w_0,u_0}u}\lesssim y_0^2 \ev{y}^{-1/2}.
\]
which is the boundedness described in \eqref{derivativeboundb}.

With these computations done, straightforward difference estimates allow one to prove that the formal derivatives so far computed are indeed first order approximations, and that one has the Lipschitz continuity statements \eqref{Grlipschitz} and \eqref{Gfunctionlipschitz}.
\end{proof}

\subsection{Proof of Lemma \ref{intconstructionlemmapre}}
\label{intconstructionlemmapreproof}
\begin{proof}
This is a standard fixed point argument using the integral formulation \eqref{integralform}. The proof is structured similarly to the proof of Lemma \ref{extconstructionlemmapre}. We let $(u_0,w_0)=(0,0)$ and we shall inductively define, for $n\geq 1$,
\[
(W_n,U_n)=(G_r(w_{n-1},u_{n-1}),F_r(w_{n-1},u_{n-1})),\qquad (w_n,u_n)=(S_r(W_n),T_r(U_n),K_r(W_n))
\]
We claim that this sequence will converge in $Z_{y_1}\times Y_{y_1}$. Indeed, we see that that $\norm{w_1}_{Z_{y_1}},\norm{u_1}_{Y_{y_1}}\leq C_1 C_2(1)$. We let $\tilde{C}=2\sup(C_1 C_2(1),1)$. We then require that $y_0$ be sufficiently small so that $4 C_1C_2(\tilde{C})y_0^2\leq \frac{1}{10}$. Assuming that for $k\geq 1$,
\[
\norm{w_{k}}_{Z_{y_1}},\norm{w_{k-1}}_{Z_{y_1}},\norm{u_{k}}_{Y_{y_1}},\norm{u_{k-1}}_{Y_{y_1}}\leq \tilde{C},
\]
we obtain
\begin{align*}
\norm{w_{k+1}-w_{k}}_{Z_{y_1}}+\norm{u_{k+1}-u_{k}}_{Y_{y_1}}&\leq C_1\left(\norm{(W_{k+1}-W_{k}}_{G_{y_1}}+\norm{U_{k+1}-U_{k}}_{F_{y_1}}\right)\\
&\leq C_1 C_2(\tilde{C})y_0^2(\norm{w_k-w_{k-1}}_{Z_{y_1}}+\norm{u_{k}-u_{k-1}}_{Y_{y_1}})\\
&\leq \frac{1}{10}(\norm{w_{k}-w_{k-1}}_{Z_{y_1}}+\norm{u_{k}-u_{k-1}}_{Y_{y_1}}).
\end{align*}
where to get from the first line to the second line we use \eqref{derivativeboundb} and the fundamental theorem of calculus.
Thus, as $w_0=u_0=0$ and $\norm{w_1}_{Z_{y_1}},\norm{u_1}_{Y_{y_1}}\leq \frac{\tilde{C}}{2}$, we obtain inductively that
\[
\norm{u_{n}}_{Y_{y_1}},\norm{w_{n}}_{Z_{y_1}}\leq \tilde{C}(\frac{1}{2}+\frac{1}{10}+\dots+\frac{1}{10^{n}})\leq \tilde{C}.
\]
\[
\norm{w_n-w_{n-1}}_{Y_{y_1}}+\norm{u_n-u_{n-1}}_{Z_{y_1}}\leq \frac{1}{10^{n-1}}\left(\norm{u_{1}}_{Y_{y_1}}+\norm{w_1}_{Z_{y_1}}\right)
\]
Thus, if we let $C_3=100 \tilde{C}$ the sequence converges to a fixed point with norm $\norm{w}_{Z_{y_1}},\norm{u}_{Y_{y_1}}\leq \frac{C_3}{100}$. We note that if $w',u'$ is a fixed point with $\norm{w'}_{Z_{y_1}},\norm{u'}_{Y_{y_1}}\leq C_3$, then by the Lipschitz bounds we have
\[
\norm{w-w'}_{Z_{y_1}}+\norm{u-u'}_{Y_{y_0}}\leq C_2(C_3)C_1y_0^2(\norm{w-w'}_{Z_{y_1}}+\norm{u-u'}_{Y_{y_0}})
\]
and hence by imposing $y_0^2C_2(C_3)C_1\leq \frac{1}{2}$ we obtain the desired uniqueness concluding the proof of part (a).

To prove part (b), we consider $r_0,r_1\in (\frac{1}{2},2)$. We let
\[
(w_0,u_0)=(w[\lambda_0;r_0],u[\lambda_0;r_0]),\qquad (w_1,u_1)=(w[\lambda_0;r_1],u[\lambda_0;r_1]).
\]
As they are solutions to \eqref{integralform}, we have
\begin{multline*}
w_1-w_0=S_1\left(G_{r_1}(w_1,u_1)-G_{r_0}(w_1,u_1)\right)\\
+(S_1-S_0)(G_{r_0}(w_1,u_1))+S_0(G_{r_0}(w_1,u_1)-G_{r_0}(w_0,u_0)).
\end{multline*}
\end{proof}
Expanding $u_1-u_0$ similarly, and using the differentiability with respect to $r$, we obtain
\begin{align*}
\norm{w_1-w_0}_{Z_{y_1}}+\norm{u_1-u_0}_{Y_{y_1}}&\leq C_1 C_2(\tilde{C})\abs{r_0-r_1}+2C_1C_2(\tilde{C})y_0^2\left(\norm{w_1-w_0}_{Z_{y_1}}+\norm{u_1-u_0}_{Y_{y_1}}\right)\\
&\leq C_1 C_2(\tilde{C})\abs{r_0-r_1}+\frac{1}{2}\left(\norm{w_1-w_0}_{Z_{y_1}}+\norm{u_1-u_0}_{Y_{y_1}}\right)
\end{align*}
and hence by absorbing the second term into the left-hand side, we conclude
\begin{equation}
\norm{w_1-w_0}_{Z_{y_1}}+\norm{u_1-u_0}_{Y_{y_1}}\leq 2 C_1 C_2(\tilde{C})\abs{r_0-r_1}.
\label{differencebound}
\end{equation}
Next, for each $r\in (\frac{1}{2},2)$, we consider the linearization of \eqref{integralform}, namely
\begin{equation}
\begin{cases}
u_{lin}=\pdv{T_r}{r}F_r+T_r\pdv{F_r}{r}+\pdv{K_r}{r}G_r+K_r\pdv{G_r}{r}+\left(T_r(\pdv{F_r}{w},\pdv{F_r}{u})+K_r(\pdv{G_r}{w},\pdv{F_r}{u})\right)\cdot (w_{lin},u_{lin})\\
w_{lin}=\pdv{S_r}{r}G_r+S_r\pdv{G_r}{r}+S_r(\pdv{G_r}{w},\pdv{G_r}{u})\cdot (w_{lin},u_{lin})
\end{cases}
\label{linearizediteration}
\end{equation}
with $T,S,K,F,G$ and their derivatives evaluated at $(r,w[\lambda_0;r],u[\lambda_0;r])$. By the Lemmas \ref{linlem} and \ref{nonlinlemma},
\begin{align*}
\norm{\pdv{S_r}{r}G_r+S_r\pdv{G_r}{r}}_{Z_{y_1}}&\leq 5 C_1C_2(C_3),\\
\norm{S_r\pdv{G_r}{w}}_{Z_{y_1}\to Z_{y_1}}+\norm{S_r\pdv{G_r}{u}}_{Y_{y_1}\to Z_{y_1}}&\leq 5 y_0^2C_1C_2(C_3).
\end{align*}
There are similar estimates for the first equation. Thus, as we may assume $10 y_0^2C_1C_2(C_3)\leq \frac{1}{2}$, we may iteratively solve \eqref{linearizediteration} resulting in a solution $(w_{lin}[\lambda_0;r],u_{lin}[\lambda_0;r])$ satisfying
\[
\norm{w_{lin}}_{Z_{y_1}}+\norm{u_{lin}}_{Y_{y_1}}\leq 20 C_1C_2(C_3).
\]
We next define (recall $w_i,u_i=w[\lambda_0;r_i],u[\lambda_0;r_i]$)
\[
(w_{diff},u_{diff})=(w_1,u_1)-(w_0,u_0)-(r_1-r_0)(w_{lin}[\lambda_o;r_0],u_{lin}[\lambda_0;r_0]).
\]
We note that
\begin{align*}
w_{diff}&=S_{r_1}G_{r_1}(w_1,u_1)-S_{r_{0}}G_{r_0}(w_0,u_0)-(r_1-r_0)(\pdv{S_r}{r} G_r+S_r \pdv{G_r}{r})\\
&\qquad -S_{r_0}\pdv{G_{r_0}}{w}\ (w_1-w_0)-\pdv{G_{r_0}}{u}\ (u_1-u_0)+S_{r_0}\pdv{G_{r_0}}{w}w_{diff}+S_{r_0}\pdv{G_{r_0}}{u}u_{diff}\\
&=(S_{r_1}-S_{r_0})(G_{r_1}-G_{r_0})(w_1,u_1)+(S_{r_1}-S_{r_0})(G_{r_0}(w_1,u_1)-G_{r_0}(w_0,u_0))\\
&\qquad+\left(S_{r_1}-S_{r_0}-(r_1-r_0)\pdv{S_r}{r}|_{r_0}\right)G_{r_0}(w_0,u_0)\\
&\qquad +S_{r_0}\left((G_{r_1}-G_{r_0})(w_1,u_1)-(G_{r_1}-G_{r_0})(w_0,u_0)\right)\\
&\qquad +S_{r_0}(G_{r_1}-G_{r_0}-(r_1-r_0)\pdv{G_r}{r}|_{r_0})(w_0,u_0)\\
&\qquad +S_{r_0}\left(G_{r_0}(w_1,u_1)-G_{r_0}(w_0,u_0)-(\pdv{G_{r_0}}{w}\ (w_1-w_0)+\pdv{G_{r_0}}{u}\ (u_1-u_0))\right).\\
&\qquad +S_{r_0}(\pdv{G_{r_0}}{w},\pdv{G_{r_0}}{u})(w_{diff},u_{diff})
\end{align*}
We absorb the last term into the left-hand side, and use \eqref{linearderivativelipschitz}, \eqref{Grlipschitz}, \eqref{Gfunctionlipschitz} and \eqref{differencebound} to obtain
\[
\norm{w_{diff}}_{Z_{y_1}}\leq 100(C_1C_2(\tilde{C}))^2\abs{r_1-r_0}^2.
\]
One can estimate $u_{diff}$ similarly, and taking $r_1\to r_0$ we see that $(w[\lambda_0;r],u[\lambda_0;r])$ is differentiable with respect to $r$ at $r_0$ with derivative $(w_{lin}[\lambda_0;r_0],u_{lin}[\lambda_0;r_0])$. 

Finally, to show continuity of $(w_{lin}[\lambda_0;r],u_{lin}[\lambda_0;r])$ with respect to $r$, we introduce the notation $w_{lin,i}=w_{lin}[\lambda_0;r_i]$ and compute
\begin{align*}
&w_{lin,1}-w_{lin,0}\\
&=\pdv{S_{r}}{r}|_{r_1}G_{r_1}(w_1,u_1)+S_{r_1}\pdv{G_{r}}{r}|_{r_1}(w_1,u_1)-\pdv{S_r}{r}G_r|_{r_0}(w_0,u_0)-S_{r_0}\pdv{G_r}{r}|_{r_0}(w_0,u_0)\\
&+S_{r_1}(\pdv{G_{r_1}}{w},\pdv{G_{r_1}}{u})(w_{lin,1},u_{lin,1})-S_{r_0}(\pdv{G_{r_0}}{w},\pdv{G_{r_0}}{u})(w_{lin,0},u_{lin,0})\\
&=(\pdv{S_r}{r}|_{r_1}-\pdv{S_r}{r}|_{r_0})G_{r_1}(w_1,u_1)+\pdv{S_r}{r}|_{r_0}(G_{r_1}-G_{r_0})(w_1,u_1)+\pdv{S_r}{r}|_{r_0}(G_{r_0}(w_1,u_1)-G_{r_0}(w_0,u_0))\\
&+(S_{r_1}-S_{r_0})\pdv{G}{r}|_{r_1}(w_1,u_1)+S_{r_0}(\pdv{G_r}{r}|_{r_1}-\pdv{G_r}{r}|_{r_0})(w_1,u_1)+S_{r_0}\pdv{G_r}{r}|_{r_0}(w_1-w_0,u_1-u_0)\\
&+(S_{r_1}-S_{r_0})\pdv{G_{r_1}}{w}w_{lin,1}+S_{r_0}(\pdv{G_{r_1}}{w}-\pdv{G_{r_0}}{w})w_{lin,1}+S_{r_0}\pdv{G_{r_0}}{w}\ (w_{lin,1}-w_{lin,0})\\
&+(S_{r_1}-S_{r_0})\pdv{G_{r_1}}{u}u_{lin,1}+S_{r_0}(\pdv{G_{r_1}}{u}-\pdv{G_{r_0}}{u})u_{lin,1}+S_{r_0}\pdv{G_{r_0}}{u}\ (u_{lin,1}-u_{lin,0}).
\end{align*}
There is a similar decomposition for $u_{lin,1}-u_{lin,0}$ we may then absorb the terms containing $w_{lin,1}-w_{lin,0}$ and $u_{lin,1}-u_{lin,0}$ to the left side, and bound the rest using \eqref{linearderivativelipschitz}, \eqref{Grlipschitz}, \eqref{Gfunctionlipschitz} and \eqref{differencebound}, allowing one to show
\[
\norm{w_{lin,1}-w_{lin,0}}_{Z_{y_1}}+\norm{u_{lin,1}-0_{lin,0}}_{Z_{y_1}}\leq 1000C_1C_2(\tilde{C})\abs{r_1-r_0}
\]
and hence continuity.

\printbibliography

@article {ghj,
    AUTHOR = {Guo, Yan and Had\v{z}i\'{c}, Mahir and Jang, Juhi},
     TITLE = {Larson-{P}enston self-similar gravitational collapse},
   JOURNAL = {Comm. Math. Phys.},
  FJOURNAL = {Communications in Mathematical Physics},
    VOLUME = {386},
      YEAR = {2021},
    NUMBER = {3},
     PAGES = {1551--1601},
      ISSN = {0010-3616,1432-0916},
   MRCLASS = {35Q75 (76N99)},
  MRNUMBER = {4299130},
       DOI = {10.1007/s00220-021-04175-y},
       URL = {https://doi.org/10.1007/s00220-021-04175-y},
}

@article {crs,
    AUTHOR = {Collot, Charles and Rapha\"{e}l, Pierre and Szeftel, Jeremie},
     TITLE = {On the stability of type {I} blow up for the energy super
              critical heat equation},
   JOURNAL = {Mem. Amer. Math. Soc.},
  FJOURNAL = {Memoirs of the American Mathematical Society},
    VOLUME = {260},
      YEAR = {2019},
    NUMBER = {1255},
     PAGES = {v+97},
      ISSN = {0065-9266,1947-6221},
      ISBN = {978-1-4704-3626-1; 978-1-4704-5334-3},
   MRCLASS = {35K91 (35B44 35C06 35J61 35K15)},
  MRNUMBER = {3986939},
MRREVIEWER = {Christian\ Stinner},
       DOI = {10.1090/memo/1255},
       URL = {https://doi.org/10.1090/memo/1255},
}

@article{maedaharada,
  title = {Critical phenomena in Newtonian gravity},
  author = {Maeda, Hideki and Harada, Tomohiro},
  journal = {Phys. Rev. D},
  volume = {64},
  number = {12},
  pages = {124024},
  numpages = {7},
  year = {2001},
  month = {Nov},
  publisher = {American Physical Society},
  doi = {10.1103/PhysRevD.64.124024},
  url = {https://link.aps.org/doi/10.1103/PhysRevD.64.124024}
}

@article {brennerwitelski,
    AUTHOR = {Brenner, Michael P. and Witelski, Thomas P.},
     TITLE = {On spherically symmetric gravitational collapse},
   JOURNAL = {J. Statist. Phys.},
  FJOURNAL = {Journal of Statistical Physics},
    VOLUME = {93},
      YEAR = {1998},
    NUMBER = {3-4},
     PAGES = {863--899},
      ISSN = {0022-4715,1572-9613},
   MRCLASS = {85A30 (76N10)},
  MRNUMBER = {1666522},
       DOI = {10.1023/B:JOSS.0000033167.19114.b8},
       URL = {https://doi.org/10.1023/B:JOSS.0000033167.19114.b8},
}

@book{chandrasekhar,
  title={An Introduction to the Study of Stellar Structure},
  author={Subramanyan Chandrasekhar},
  year={1939},
  publisher={University of Chicago Press}
}

@article{haradamaedasemelin,
  title = {Criticality and convergence in Newtonian collapse},
  author = {Harada, Tomohiro and Maeda, Hideki and Semelin, Benoit},
  journal = {Phys. Rev. D},
  volume = {67},
  number = {8},
  pages = {084003},
  numpages = {10},
  year = {2003},
  month = {Apr},
  publisher = {American Physical Society},
  doi = {10.1103/PhysRevD.67.084003},
  url = {https://link.aps.org/doi/10.1103/PhysRevD.67.084003}
}

@book {olver,
    AUTHOR = {Olver, Frank W. J.},
     TITLE = {Asymptotics and special functions},
    SERIES = {AKP Classics},
      NOTE = {Reprint of the 1974 original [Academic Press, New York]},
 PUBLISHER = {A K Peters, Ltd., Wellesley, MA},
      YEAR = {1997},
     PAGES = {xviii+572},
      ISBN = {1-56881-069-5},
   MRCLASS = {41-02 (33Cxx 41A60 65D20)},
  MRNUMBER = {1429619},
}

@article {ghj2,
    AUTHOR = {Guo, Yan and Had\v{z}i\'{c}, Mahir and Jang, Juhi},
     TITLE = {Naked singularities in the {E}instein-{E}uler system},
   JOURNAL = {Ann. PDE},
  FJOURNAL = {Annals of PDE. Journal Dedicated to the Analysis of Problems
              from Physical Sciences},
    VOLUME = {9},
      YEAR = {2023},
    NUMBER = {1},
     PAGES = {Paper No. 4, 182},
      ISSN = {2524-5317,2199-2576},
   MRCLASS = {83C75 (35Q76 83C05 83C20)},
  MRNUMBER = {4548531},
MRREVIEWER = {Matthew\ R. I. Schrecker},
       DOI = {10.1007/s40818-022-00144-3},
       URL = {https://doi.org/10.1007/s40818-022-00144-3},
}

@article {ghjs,
    AUTHOR = {Guo, Yan and Had\v{z}i\'{c}, Mahir and Jang, Juhi and
              Schrecker, Matthew},
     TITLE = {Gravitational collapse for polytropic gaseous stars:
              self-similar solutions},
   JOURNAL = {Arch. Ration. Mech. Anal.},
  FJOURNAL = {Archive for Rational Mechanics and Analysis},
    VOLUME = {246},
      YEAR = {2022},
    NUMBER = {2-3},
     PAGES = {957--1066},
      ISSN = {0003-9527,1432-0673},
   MRCLASS = {85A05 (35Q31 35Q85)},
  MRNUMBER = {4514067},
MRREVIEWER = {Umananda\ Dev\ Goswami},
       DOI = {10.1007/s00205-022-01827-8},
       URL = {https://doi.org/10.1007/s00205-022-01827-8},
}

@article{hunter,
  author = {C. Hunter},
  title = {The collapse of unstable isothermal spheres},
  year = {1977},
  journal = {The Astrophysical Journal},
  volume = {218},
  pages = {834},
  month = {Jan-12-1977},
  issn = {0004-637X},
  url = {http://articles.adsabs.harvard.edu/full/1977ApJ...218..834H},
  doi = {10.1086/155739},
}

@article{larson,
    author = {Larson, Richard B.},
    title = "{ Numerical Calculations of the Dynamics of a Collapsing Proto-Star}",
    journal = {Monthly Notices of the Royal Astronomical Society},
    volume = {145},
    number = {3},
    pages = {271-295},
    year = {1969},
    month = {08},
    issn = {0035-8711},
    doi = {10.1093/mnras/145.3.271},
    url = {https://doi.org/10.1093/mnras/145.3.271},
}

@article{penston,
    author = {Penston, M. V.},
    title = "{Dynamics of Self-Gravitating Gaseous Spheres—III: Analytical Results in the Free-fall of Isothermal Cases}",
    journal = {Monthly Notices of the Royal Astronomical Society},
    volume = {144},
    number = {4},
    pages = {425-448},
    year = {1969},
    month = {06},
    issn = {0035-8711},
    doi = {10.1093/mnras/144.4.425},
    url = {https://doi.org/10.1093/mnras/144.4.425},
}

@article {buddqi,
    AUTHOR = {Budd, C. J. and Qi, Yuan-Wei},
     TITLE = {The existence of bounded solutions of a semilinear elliptic
              equation},
   JOURNAL = {J. Differential Equations},
  FJOURNAL = {Journal of Differential Equations},
    VOLUME = {82},
      YEAR = {1989},
    NUMBER = {2},
     PAGES = {207--218},
      ISSN = {0022-0396,1090-2732},
   MRCLASS = {35B05 (35J60)},
  MRNUMBER = {1027967},
MRREVIEWER = {Zhen\ Chao\ Cao},
       DOI = {10.1016/0022-0396(89)90131-9},
       URL = {https://doi.org/10.1016/0022-0396(89)90131-9},
}

@article {troy,
    AUTHOR = {Troy, William C.},
     TITLE = {The existence of bounded solutions of a semilinear heat
              equation},
   JOURNAL = {SIAM J. Math. Anal.},
  FJOURNAL = {SIAM Journal on Mathematical Analysis},
    VOLUME = {18},
      YEAR = {1987},
    NUMBER = {2},
     PAGES = {332--336},
      ISSN = {0036-1410},
   MRCLASS = {35K55},
  MRNUMBER = {876275},
MRREVIEWER = {Guang\ Chang\ Dong},
       DOI = {10.1137/0518026},
       URL = {https://doi.org/10.1137/0518026},
}

@article {lepin,
    AUTHOR = {Lepin, L. A.},
     TITLE = {Self-similar solutions of a semilinear heat equation},
   JOURNAL = {Mat. Model.},
  FJOURNAL = {Matematicheskoe Modelirovanie},
    VOLUME = {2},
      YEAR = {1990},
    NUMBER = {3},
     PAGES = {63--74},
      ISSN = {0234-0879},
   MRCLASS = {35K60 (35B05)},
  MRNUMBER = {1059601},
MRREVIEWER = {Reinhard\ Redlinger},
}

@ARTICLE{fosterchevalier,
       author = {{Foster}, Prudence N. and {Chevalier}, Roger A.},
        title = "{Gravitational Collapse of an Isothermal Sphere}",
      journal = {The Astrophysical Journal},
         year = 1993,
        month = oct,
       volume = {416},
        pages = {303},
          doi = {10.1086/173236},
       adsurl = {https://ui.adsabs.harvard.edu/abs/1993ApJ...416..303F}
}

@ARTICLE{hanawanakayama,
       author = {{Hanawa}, Tomoyuki and {Nakayama}, Kunji},
        title = "{Stability of Similarity Solutions for a Gravitationally Contracting Isothermal Sphere: Convergence to the Larson-Penston Solution}",
      journal = {The Astrophysical Journal},
         year = 1997,
        month = jul,
       volume = {484},
       number = {1},
        pages = {238-244},
          doi = {10.1086/304315},
       adsurl = {https://ui.adsabs.harvard.edu/abs/1997ApJ...484..238H}
}

@article{op1,
  title = {Naked singularities in self-similar spherical gravitational collapse},
  author = {Ori, Amos and Piran, Tsvi},
  journal = {Phys. Rev. Lett.},
  volume = {59},
  number = {19},
  pages = {2137--2140},
  numpages = {0},
  year = {1987},
  month = {Nov},
  publisher = {American Physical Society},
  doi = {10.1103/PhysRevLett.59.2137},
  url = {https://link.aps.org/doi/10.1103/PhysRevLett.59.2137}
}

@ARTICLE{op2,
       author = {{Ori}, Amos and {Piran}, Tsvi},
        title = "{Self-similar spherical gravitational collapse and the cosmic censorship hypothesis.}",
      journal = {General Relativity and Gravitation},
         year = 1988,
        month = jan,
       volume = {20},
       number = {1},
        pages = {7-13},
          doi = {10.1007/BF00759251},
       adsurl = {https://ui.adsabs.harvard.edu/abs/1988GReGr..20....7O}
}

@article{op3,
  title = {Naked singularities and other features of self-similar general-relativistic gravitational collapse},
  author = {Ori, Amos and Piran, Tsvi},
  journal = {Phys. Rev. D},
  volume = {42},
  number = {4},
  pages = {1068--1090},
  numpages = {0},
  year = {1990},
  month = {Aug},
  publisher = {American Physical Society},
  doi = {10.1103/PhysRevD.42.1068},
  url = {https://link.aps.org/doi/10.1103/PhysRevD.42.1068}
}

@ARTICLE{yahil,
       author = {{Yahil}, A.},
        title = "{Self-similar stellar collapse}",
      journal = {Astrophysical Journal},
         year = 1983,
        month = feb,
       volume = {265},
        pages = {1047-1055},
          doi = {10.1086/160746},
       adsurl = {https://ui.adsabs.harvard.edu/abs/1983ApJ...265.1047Y}
}

@article {mrrs1,
    AUTHOR = {Merle, Frank and Rapha\"{e}l, Pierre and Rodnianski, Igor and
              Szeftel, Jeremie},
     TITLE = {On the implosion of a compressible fluid {I}: {S}mooth
              self-similar inviscid profiles},
   JOURNAL = {Ann. of Math. (2)},
  FJOURNAL = {Annals of Mathematics. Second Series},
    VOLUME = {196},
      YEAR = {2022},
    NUMBER = {2},
     PAGES = {567--778},
      ISSN = {0003-486X,1939-8980},
   MRCLASS = {35Q35 (34C37)},
  MRNUMBER = {4445442},
MRREVIEWER = {Wei\ Lian},
       DOI = {10.4007/annals.2022.196.2.3},
       URL = {https://doi.org/10.4007/annals.2022.196.2.3},
}

@article {mrrs2,
    AUTHOR = {Merle, Frank and Rapha\"{e}l, Pierre and Rodnianski, Igor and
              Szeftel, Jeremie},
     TITLE = {On the implosion of a compressible fluid {II}: {S}ingularity
              formation},
   JOURNAL = {Ann. of Math. (2)},
  FJOURNAL = {Annals of Mathematics. Second Series},
    VOLUME = {196},
      YEAR = {2022},
    NUMBER = {2},
     PAGES = {779--889},
      ISSN = {0003-486X,1939-8980},
   MRCLASS = {35Q35 (35B44)},
  MRNUMBER = {4445443},
MRREVIEWER = {Robert\ Schippa},
       DOI = {10.4007/annals.2022.196.2.4},
       URL = {https://doi.org/10.4007/annals.2022.196.2.4},
}

@article {mrrs3,
    AUTHOR = {Merle, Frank and Rapha\"{e}l, Pierre and Rodnianski, Igor and
              Szeftel, Jeremie},
     TITLE = {On blow up for the energy super critical defocusing nonlinear
              {S}chr\"{o}dinger equations},
   JOURNAL = {Invent. Math.},
  FJOURNAL = {Inventiones Mathematicae},
    VOLUME = {227},
      YEAR = {2022},
    NUMBER = {1},
     PAGES = {247--413},
      ISSN = {0020-9910,1432-1297},
   MRCLASS = {35Q55},
  MRNUMBER = {4359478},
       DOI = {10.1007/s00222-021-01067-9},
       URL = {https://doi.org/10.1007/s00222-021-01067-9},
}

@article {buddnorbury,
    AUTHOR = {Budd, C. and Norbury, J.},
     TITLE = {Semilinear elliptic equations and supercritical growth},
   JOURNAL = {J. Differential Equations},
  FJOURNAL = {Journal of Differential Equations},
    VOLUME = {68},
      YEAR = {1987},
    NUMBER = {2},
     PAGES = {169--197},
      ISSN = {0022-0396,1090-2732},
   MRCLASS = {35J65},
  MRNUMBER = {892022},
MRREVIEWER = {H.\ J.\ Kuiper},
       DOI = {10.1016/0022-0396(87)90190-2},
       URL = {https://doi.org/10.1016/0022-0396(87)90190-2},
}

@article {dancerguowei,
    AUTHOR = {Dancer, E. N. and Guo, Zongming and Wei, Juncheng},
     TITLE = {Non-radial singular solutions of the {L}ane-{E}mden equation
              in {$\Bbb R^N$}},
   JOURNAL = {Indiana Univ. Math. J.},
  FJOURNAL = {Indiana University Mathematics Journal},
    VOLUME = {61},
      YEAR = {2012},
    NUMBER = {5},
     PAGES = {1971--1996},
      ISSN = {0022-2518,1943-5258},
   MRCLASS = {35J91 (35B08 58J05)},
  MRNUMBER = {3119607},
MRREVIEWER = {Vicen\c{t}iu\ D.\ R\u{a}dulescu},
       DOI = {10.1512/iumj.2012.61.4749},
       URL = {https://doi.org/10.1512/iumj.2012.61.4749},
}

@article {koch,
    AUTHOR = {Koch, Herbert},
     TITLE = {Self-similar solutions to super-critical g{K}d{V}},
   JOURNAL = {Nonlinearity},
  FJOURNAL = {Nonlinearity},
    VOLUME = {28},
      YEAR = {2015},
    NUMBER = {3},
     PAGES = {545--575},
      ISSN = {0951-7715,1361-6544},
   MRCLASS = {35Q53 (35B44 35C06)},
  MRNUMBER = {3311593},
MRREVIEWER = {Amin\ Esfahani},
       DOI = {10.1088/0951-7715/28/3/545},
       URL = {https://doi.org/10.1088/0951-7715/28/3/545},
}

@ARTICLE{op4,
       author = {{Ori}, Amos and {Piran}, Tsvi},
        title = "{A simple stability criterion for isothermal spherical self-similar flow}",
      journal = {Monthly Notices of the Royal Astronomical Society},
         year = 1988,
        month = oct,
       volume = {234},
        pages = {821-829},
          doi = {10.1093/mnras/234.4.821},
       adsurl = {https://ui.adsabs.harvard.edu/abs/1988MNRAS.234..821O}
}

@ARTICLE{shu,
       author = {{Shu}, F.~H.},
        title = "{Self-similar collapse of isothermal spheres and star formation.}",
      journal = {The Astrophysical Journal},
         year = 1977,
        month = jun,
       volume = {214},
        pages = {488-497},
          doi = {10.1086/155274},
       adsurl = {https://ui.adsabs.harvard.edu/abs/1977ApJ...214..488S}
}

@ARTICLE{whitworthsummers,
       author = {{Whitworth}, A. and {Summers}, D.},
        title = "{Self-similar condensation of spherically symmetric self-gravitating isothermal gas clouds}",
      journal = {Monthly Notices of the Royal Astronomical Society},
         year = 1985,
        month = may,
       volume = {214},
        pages = {1-25},
          doi = {10.1093/mnras/214.1.1},
       adsurl = {https://ui.adsabs.harvard.edu/abs/1985MNRAS.214....1W}
}

@article {biernat,
    AUTHOR = {Biernat, Pawel and Bizo\'{n}, Piotr},
     TITLE = {Shrinkers, expanders, and the unique continuation beyond
              generic blowup in the heat flow for harmonic maps between
              spheres},
   JOURNAL = {Nonlinearity},
  FJOURNAL = {Nonlinearity},
    VOLUME = {24},
      YEAR = {2011},
    NUMBER = {8},
     PAGES = {2211--2228},
      ISSN = {0951-7715,1361-6544},
   MRCLASS = {58E20 (35A20 35C06 58J35 65M06 65M50)},
  MRNUMBER = {2813584},
MRREVIEWER = {Roger\ Moser},
       DOI = {10.1088/0951-7715/24/8/005},
       URL = {https://doi.org/10.1088/0951-7715/24/8/005},
}

@article {matanomerle2,
    AUTHOR = {Matano, Hiroshi and Merle, Frank},
     TITLE = {Classification of type {I} and type {II} behaviors for a
              supercritical nonlinear heat equation},
   JOURNAL = {J. Funct. Anal.},
  FJOURNAL = {Journal of Functional Analysis},
    VOLUME = {256},
      YEAR = {2009},
    NUMBER = {4},
     PAGES = {992--1064},
      ISSN = {0022-1236,1096-0783},
   MRCLASS = {35K55 (35B40)},
  MRNUMBER = {2488333},
MRREVIEWER = {Juli\'{a}n\ Aguirre},
       DOI = {10.1016/j.jfa.2008.05.021},
       URL = {https://doi.org/10.1016/j.jfa.2008.05.021},
}

@article {matanomerle1,
    AUTHOR = {Matano, Hiroshi and Merle, Frank},
     TITLE = {On nonexistence of type {II} blowup for a supercritical
              nonlinear heat equation},
   JOURNAL = {Comm. Pure Appl. Math.},
  FJOURNAL = {Communications on Pure and Applied Mathematics},
    VOLUME = {57},
      YEAR = {2004},
    NUMBER = {11},
     PAGES = {1494--1541},
      ISSN = {0010-3640,1097-0312},
   MRCLASS = {35K55 (35B40)},
  MRNUMBER = {2077706},
MRREVIEWER = {Juli\'{a}n\ Aguirre},
       DOI = {10.1002/cpa.20044},
       URL = {https://doi.org/10.1002/cpa.20044},
}

@misc{buckmastercaolaboragomezserrano,
      title={Smooth imploding solutions for 3D compressible fluids}, 
      author={Tristan Buckmaster and Gonzalo Cao-Labora and Javier Gómez-Serrano},
      year={2022},
      eprint={2208.09445},
      archivePrefix={arXiv},
      primaryClass={math.AP}
}

@misc{caolaboragomezserranoshistaffilani,
      title={Non-radial implosion for compressible Euler and Navier-Stokes in $\mathbb{T}^3$ and $\mathbb{R}^3$}, 
      author={Gonzalo Cao-Labora and Javier Gómez-Serrano and Jia Shi and Gigliola Staffilani},
      year={2023},
      eprint={2310.05325},
      archivePrefix={arXiv},
      primaryClass={math.AP}
}

@article {guderley,
    AUTHOR = {Guderley, G.},
     TITLE = {Starke kugelige und zylindrische {V}erdichtungsst\"{o}sse in
              der {N}\"{a}he des {K}ugelmittelpunktes bzw. der
              {Z}ylinderachse},
   JOURNAL = {Luftfahrtforschung},
  FJOURNAL = {Luftfahrtforschung},
    VOLUME = {19},
      YEAR = {1942},
     PAGES = {302--311},
      ISSN = {0368-7643},
   MRCLASS = {76.1X},
  MRNUMBER = {8522},
MRREVIEWER = {D.\ G.\ Bourgin},
}

@book {sedov,
    AUTHOR = {Sedov, L. I.},
     TITLE = {Similarity and dimensional methods in mechanics},
      NOTE = {Translated from the Russian by V. I. Kisin},
 PUBLISHER = {``Mir'', Moscow},
      YEAR = {1982},
     PAGES = {424},
   MRCLASS = {00A69 (76-01 85-01)},
  MRNUMBER = {693457},
}

@article{jenssentsikkou,
    author = {Jenssen, Helge Kristian and Tsikkou, Charis},
    title = "{Radially symmetric non-isentropic Euler flows: Continuous blowup with positive pressure}",
    journal = {Physics of Fluids},
    volume = {35},
    number = {1},
    pages = {016117},
    year = {2023},
    month = {01},
    issn = {1070-6631},
    doi = {10.1063/5.0134136},
    url = {https://doi.org/10.1063/5.0134136},
}

@article{continuedcollapse,
    AUTHOR = {Guo, Yan and Had\v{z}i\'{c}, Mahir and Jang, Juhi},
     TITLE = {Continued gravitational collapse for {N}ewtonian stars},
   JOURNAL = {Arch. Ration. Mech. Anal.},
  FJOURNAL = {Archive for Rational Mechanics and Analysis},
    VOLUME = {239},
      YEAR = {2021},
    NUMBER = {1},
     PAGES = {431--552},
      ISSN = {0003-9527,1432-0673},
   MRCLASS = {85A05 (35Q75)},
  MRNUMBER = {4198723},
MRREVIEWER = {Nikos\ Labropoulos},
       DOI = {10.1007/s00205-020-01580-w},
       URL = {https://doi.org/10.1007/s00205-020-01580-w},
}

@article{hadzicsurvey,
    AUTHOR = {Had\v{z}i\'{c}, Mahir},
     TITLE = {Star dynamics: collapse vs. expansion},
   JOURNAL = {Quart. Appl. Math.},
  FJOURNAL = {Quarterly of Applied Mathematics},
    VOLUME = {81},
      YEAR = {2023},
    NUMBER = {2},
     PAGES = {329--365},
      ISSN = {0033-569X,1552-4485},
   MRCLASS = {85A05 (35Q31 35Q75 76N10 85A30)},
  MRNUMBER = {4592097},
}

@article{goldreichweber,
       author = {{Goldreich}, P. and {Weber}, S.~V.},
        title = "{Homologously collapsing stellar cores}",
      journal = {The Astrophysical Journal},
         year = 1980,
        month = jun,
       volume = {238},
        pages = {991-997},
          doi = {10.1086/158065},
       adsurl = {https://ui.adsabs.harvard.edu/abs/1980ApJ...238..991G}
}

@article{fulin,
    AUTHOR = {Fu, Chun-Chieh and Lin, Song-Sun},
     TITLE = {On the critical mass of the collapse of a gaseous star in
              spherically symmetric and isentropic motion},
   JOURNAL = {Japan J. Indust. Appl. Math.},
  FJOURNAL = {Japan Journal of Industrial and Applied Mathematics},
    VOLUME = {15},
      YEAR = {1998},
    NUMBER = {3},
     PAGES = {461--469},
      ISSN = {0916-7005,1868-937X},
   MRCLASS = {85A30 (76N15)},
  MRNUMBER = {1651739},
MRREVIEWER = {F.\ Nahon},
       DOI = {10.1007/BF03167322},
       URL = {https://doi.org/10.1007/BF03167322},
}

@article{dengxiangyang,
    AUTHOR = {Deng, Yinbin and Xiang, Jianlin and Yang, Tong},
     TITLE = {Blowup phenomena of solutions to {E}uler-{P}oisson equations},
   JOURNAL = {J. Math. Anal. Appl.},
  FJOURNAL = {Journal of Mathematical Analysis and Applications},
    VOLUME = {286},
      YEAR = {2003},
    NUMBER = {1},
     PAGES = {295--306},
      ISSN = {0022-247X,1096-0813},
   MRCLASS = {35Q35 (35B40 76N10 85A15)},
  MRNUMBER = {2009638},
MRREVIEWER = {Ralph\ Saxton},
       DOI = {10.1016/S0022-247X(03)00487-6},
       URL = {https://doi.org/10.1016/S0022-247X(03)00487-6},
}

@inproceedings {makino,
    AUTHOR = {Makino, Tetu},
     TITLE = {Blowing up solutions of the {E}uler-{P}oisson equation for the
              evolution of gaseous stars},
 BOOKTITLE = {Proceedings of the {F}ourth {I}nternational {W}orkshop on
              {M}athematical {A}spects of {F}luid and {P}lasma {D}ynamics
              ({K}yoto, 1991)},
   JOURNAL = {Transport Theory Statist. Phys.},
  FJOURNAL = {Transport Theory and Statistical Physics},
    VOLUME = {21},
      YEAR = {1992},
    NUMBER = {4-6},
     PAGES = {615--624},
      ISSN = {0041-1450,1532-2424},
   MRCLASS = {85A15 (76N15)},
  MRNUMBER = {1194464},
       DOI = {10.1080/00411459208203801},
       URL = {https://doi.org/10.1080/00411459208203801},
}

@article {dengliuyangyao,
    AUTHOR = {Deng, Yinbin and Liu, Tai-Ping and Yang, Tong and Yao,
              Zheng-an},
     TITLE = {Solutions of {E}uler-{P}oisson equations for gaseous stars},
   JOURNAL = {Arch. Ration. Mech. Anal.},
  FJOURNAL = {Archive for Rational Mechanics and Analysis},
    VOLUME = {164},
      YEAR = {2002},
    NUMBER = {3},
     PAGES = {261--285},
      ISSN = {0003-9527,1432-0673},
   MRCLASS = {35Q35 (76N10 85A15)},
  MRNUMBER = {1930393},
       DOI = {10.1007/s00205-002-0209-6},
       URL = {https://doi.org/10.1007/s00205-002-0209-6},
}

@article {gigakohn1,
    AUTHOR = {Giga, Yoshikazu and Kohn, Robert V.},
     TITLE = {Asymptotically self-similar blow-up of semilinear heat
              equations},
   JOURNAL = {Comm. Pure Appl. Math.},
  FJOURNAL = {Communications on Pure and Applied Mathematics},
    VOLUME = {38},
      YEAR = {1985},
    NUMBER = {3},
     PAGES = {297--319},
      ISSN = {0010-3640,1097-0312},
   MRCLASS = {35K55 (35B40)},
  MRNUMBER = {784476},
MRREVIEWER = {Mitsuhiro\ Nakao},
       DOI = {10.1002/cpa.3160380304},
       URL = {https://doi.org/10.1002/cpa.3160380304},
}

@article {gigakohn2,
    AUTHOR = {Giga, Yoshikazu and Kohn, Robert V.},
     TITLE = {Characterizing blowup using similarity variables},
   JOURNAL = {Indiana Univ. Math. J.},
  FJOURNAL = {Indiana University Mathematics Journal},
    VOLUME = {36},
      YEAR = {1987},
    NUMBER = {1},
     PAGES = {1--40},
      ISSN = {0022-2518,1943-5258},
   MRCLASS = {35B40},
  MRNUMBER = {876989},
MRREVIEWER = {J.\ W.\ Bebernes},
       DOI = {10.1512/iumj.1987.36.36001},
       URL = {https://doi.org/10.1512/iumj.1987.36.36001},
}

@article {gigakohn3,
    AUTHOR = {Giga, Yoshikazu and Kohn, Robert V.},
     TITLE = {Nondegeneracy of blowup for semilinear heat equations},
   JOURNAL = {Comm. Pure Appl. Math.},
  FJOURNAL = {Communications on Pure and Applied Mathematics},
    VOLUME = {42},
      YEAR = {1989},
    NUMBER = {6},
     PAGES = {845--884},
      ISSN = {0010-3640,1097-0312},
   MRCLASS = {35B40 (35B05 35K55)},
  MRNUMBER = {1003437},
MRREVIEWER = {Frank\ Merle},
       DOI = {10.1002/cpa.3160420607},
       URL = {https://doi.org/10.1002/cpa.3160420607},
}

@article {merlezaag,
    AUTHOR = {Merle, Frank and Zaag, Hatem},
     TITLE = {Stability of the blow-up profile for equations of the type
              {$u_t=\Delta u+|u|^{p-1}u$}},
   JOURNAL = {Duke Math. J.},
  FJOURNAL = {Duke Mathematical Journal},
    VOLUME = {86},
      YEAR = {1997},
    NUMBER = {1},
     PAGES = {143--195},
      ISSN = {0012-7094,1547-7398},
   MRCLASS = {35K55 (35B35 35B40)},
  MRNUMBER = {1427848},
MRREVIEWER = {H.\ J.\ Kuiper},
       DOI = {10.1215/S0012-7094-97-08605-1},
       URL = {https://doi.org/10.1215/S0012-7094-97-08605-1},
}

@incollection {bricmontkupiainen,
    AUTHOR = {Bricmont, J. and Kupiainen, A.},
     TITLE = {Renormalization group and nonlinear {PDE}s},
 BOOKTITLE = {Quantum and non-commutative analysis ({K}yoto, 1992)},
    SERIES = {Math. Phys. Stud.},
    VOLUME = {16},
     PAGES = {113--118},
 PUBLISHER = {Kluwer Acad. Publ., Dordrecht},
      YEAR = {1993},
      ISBN = {0-7923-2532-X},
   MRCLASS = {82C28 (35R60)},
  MRNUMBER = {1276284},
}
\end{document}